\documentclass[a4paper]{article}

\usepackage{amsmath,amsthm,amssymb,empheq,mathtools}
\usepackage{microtype}
\usepackage{titlepic}
\usepackage{verbatim,enumitem}
\usepackage[cal=euler,scr=rsfso]{mathalfa}
\usepackage[alphabetic]{amsrefs}

\usepackage{hyperref}
\usepackage{cleveref}
\hypersetup{
	colorlinks,
	allcolors=[rgb]{0.45,0,0}
}

\usepackage{tikz}
\usetikzlibrary{matrix,arrows,cd}

\title{Rigidification and the Coherent Nerve for Enriched Quasicategories}
\author{Harry Gindi}
\date{\today}

\numberwithin{equation}{subsection}

\theoremstyle{plain}   

\newtheorem{thm}[equation]{Theorem}
\newtheorem{prop}[equation]{Proposition}
\newtheorem{cor}[equation]{Corollary}
\newtheorem{lemma}[equation]{Lemma}

\theoremstyle{definition}
\newtheorem{defn}[equation]{Definition}
\newtheorem{note}[equation]{Note}
\newtheorem{obs}[equation]{Observation}

\theoremstyle{remark}
\newtheorem{rem}[equation]{Remark}     
\newtheorem{example}[equation]{Example}

\theoremstyle{plain}

\DeclareMathOperator{\id}{id}

\DeclareMathOperator{\colim}{colim}
\DeclareMathOperator*{\coliml}{colim}
\DeclareMathOperator{\limw}{lim}

\DeclareMathOperator{\Ob}{Ob}

\DeclareMathOperator{\Psh}{Psh}

\newcommand{\op}{\ensuremath{\mathrm{op}}}
\newcommand{\Nec}{\ensuremath{{\mathcal{N}ec}}}
\newcommand{\Hoc}{\ensuremath{{\mathcal{H}oc}}}
\newcommand{\Cat}{\ensuremath{\mathbf{Cat}}}

\newcommand{\overcat}[2]{{\left(#1\downarrow #2\right)}}

\DeclareMathOperator{\Map}{Map}

\DeclareMathOperator{\Hom}{Hom}

\DeclareMathOperator*{\hocoliml}{hocolim}

\newcommand{\realiz}[1]{\ensuremath{\left\lvert#1\right\rvert}}

\newcommand{\psh}[1]{\ensuremath{\widehat{#1}}}

\providecommand{\C}{}
\renewcommand{\C}{\ensuremath{\mathcal{C}}}

\newcommand{\M}{\ensuremath{\mathcal{M}}}
\newcommand{\Pre}{\ensuremath{\mathcal{P}}}
\newcommand{\setS}{\ensuremath{\mathscr{S}}}

\newcommand{\defeq}{\overset{\mathrm{def}}=}

\newcommand{\cellset}{\ensuremath{\widehat{\Theta[\mathcal{C}]}}}
\newcommand{\ssetlab}{\ensuremath{\widehat{\Delta} \int \widehat{\mathcal{C}}}}
\newcommand{\spsh}{\ensuremath{\operatorname{Psh}_\Delta(\mathcal{C})}}
\begin{document}
\pagenumbering{gobble}
\maketitle

\begin{abstract}
	We introduce, for \(\C\) a regular Cartesian Reedy category a model category whose fibrant objects are an analogue of quasicategories enriched in simplicial presheaves on \(\C\).  We then develop a coherent realization and nerve for this model structure and demonstrate using an enriched version of the necklaces of Dugger and Spivak that our model category is Quillen-equivalent to the category of categories enriched in simplicial presheaves on \(\C\).  We then show that for any Cartesian-closed left-Bousfield localization of the category of simplicial presheaves on \(\C\), the coherent nerve and realization descend to a Quillen equivalence on the localizations of these model categories. As an application, we demonstrate a version of Yoneda's lemma for these enriched quasicategories.
\end{abstract}

\tableofcontents

\clearpage
\pagenumbering{arabic}

\section*{Introduction}
\label{sec:intro}
\addcontentsline{toc}{section}{\nameref{sec:intro}}
In his thesis \cite{oury}, David Oury introduced machinery to give a novel proof that his constructed model structure on \(\Theta_2\)-sets is Cartesian-monoidal closed. Around the same time, in \cite{rezk-theta-n-spaces}, Charles Rezk constructed a model structure on \(\Theta_n\)-spaces, that, in the case \(n=2\), was known by folklore to be Quillen bi-equivalent\footnote{Given two model categories \(\mathcal{M}\) and \(\mathcal{N}\) we say that they are \emph{Quillen bi-equivalent} to mean that there exists a pair of \emph{left} Quillen equivalences \(\mathcal{M}\to \mathcal{N}\) and \(\mathcal{N}\to \mathcal{M}\) that are mutually quasi-inverse on homotopy categories. This is somewhat nonstandard, but we know of no other name to refer to this strong condition} to Oury's model structure.  However, Rezk's construction allows us to model weak enrichment in a much larger class of model categories, namely Cartesian-closed model categories whose underlying categories are simplicial presheaves on a small category \(\C\) satisfying some tame restrictions.  

Bergner and Rezk, in a \cite{bergner-rezk-1} and \cite{bergner-rezk-2}, also showed by means of a zig-zag of Quillen equivalences that the category of \(\Theta_n\)-spaces equipped with Rezk's model structure models the same homotopy theory as the model category of \(\Psh_\Delta(\Theta_{n-1})\)-enriched categories, equipped with the Bergner-Lurie model structure for categories enriched in \(\Theta_{n-1}\)-spaces equipped with Rezk's model structure.  Because the equivalence is indirect, however, many of the ideas from Lurie's work on \((\infty,1)\)-categories cannot be adapted in a straightforward manner, specifically his construction of the Yoneda embedding and his proof of Yoneda's lemma in \cite{htt}.  In order to rectify this problem, we split the problem up into two parts: 

We first introduce using a novel model structure on \(\Theta_n\)-sets (or more generally \(\Theta[\C]\)-sets for an appropriate \(\C\)) that emerges naturally as a hybrid of Oury's model structure on \(\Theta_2\)-sets and Rezk's model structure on \(\Theta[\C]\)-spaces.  Specifically, we use Oury's machinery to construct a model structure on \(\Theta[\C]\)-sets that models weak enrichment in simplicial presheaves on \(\C\). We then compare this model structure with an intermediate model structure of Rezk, demonstrating they are Quillen bi-equivalent.  As a result of this bi-equivalence, we can later use results of Rezk \cite{rezk-theta-n-spaces} to localize this model structure 'hom-wise' with respect to what Rezk calls a Cartesian presentation on \(\C\), which is again equivalent to Rezk's localized model structure by merit of Cisinski's results on simplicial completion (see Appendix or \cite{cisinski-book}).  Like Rezk's model structure, ours is also Cartesian-monoidal as a model category. Since we prove many of these theorems using machinery developed by Oury in the unpublished portion of his thesis \cite{oury}, we also provide full proofs of all of his relevant results, but in our more general setting.

We then construct a version of the coherent realization and nerve adjunction between \(\Theta[\C]\)-sets and categories enriched in simplicial presheaves on \(\C\), which reduce to the classical ones in the case where we take \(\C=\ast\) the terminal category. We then demonstrate that this adjunction is a Quillen equivalence between appropriate model structures using an enhanced version of Dugger and Spivak's calculus of necklaces developed in \cite{ds1} and \cite{ds2}. 

Our direct result is strictly stronger than the result of Bergner and Rezk because it allows us to account at the very least for the new case \(\Theta=\Theta_\omega\), which satisfies all of our constraints on \(\C\) by \cite{berger-iterated-wreath}, which the Bergner-Rezk approach could not handle, since one of the categories appearing in the zig-zag (the height-\(n\) analogue of Segal categories) only makes sense for \(\C=\Theta_n\) for \(n\) finite. Their approach goes through rigidification results for homotopy-coherent simplicial models of algebraic theories due to Badzioch (see \cite{bergner-rezk-2}*{Section 5}). Moreover, all of our Quillen equivalences point in the right direction to apply Lurie's construction of the Yoneda embedding and reproduce his proof Yoneda's lemma in this generalized setting.

The paper is organized into the following chapters:

\subsection*{Formal \(\C\)-quasicategories}
In the first chapter, we apply a general construction to define what we call \emph{labeled simplicial sets} with respect to a monoidal category \(\mathcal{V}\).  We then specialize to the case where \(\mathcal{V}\) is the Cartesian-monoidal category of presheaves of sets on a small category \(\C\), which we additionally require to be a special kind of Reedy category that axiomatizes a form of the Eilenberg-Zilber shuffle decomposition for products of simplicial sets.  We then define \(\Theta[\C]\) to be the full subcategory of the labeled simplicial sets whose underlying simplicial sets are simplices and whose edges are all labeled by representable presheaves on \(\C\).  

We define the category of \emph{\(\C\)-cellular sets} to be the category of presheaves of sets on this category.  We then apply machinery of Cisinski and Oury to construct the \emph{horizontal Joyal model structure} on the category of \(\C\)-cellular sets that has many of the familiar nice properties of the Joyal model structure. We call the fibrant objects of this category the \emph{formal \(\C\)-quasicategories}.

We direct the attention of the reader to \Cref{horizontal}, which proves that the model structure is Cartesian-monoidal as well as \Cref{rezkcomparison}, where we prove a useful equivalence with an analogous model structure constructed by Rezk.    

The chapter culminates with a key technical result that gives a characterization of the fibrant objects by a simple lifting property and the fibrations between them as the \emph{isofibrations}, namely the horizontal inner fibrations that have the right lifting property with respect to the inclusion of a vertex into a freestanding isomorphism, extending an important  theorem of Joyal to this setting.

\subsection*{The Coherent Nerve, Horizontal case}
In the second chapter, we define an extension of Lurie's coherent realization functor \(\mathfrak{C}_\Delta\) to our setting.  We leverage the equivalence between \(\C\)-indexed simplicially-enriched categories with a constant set of objects and \(\spsh\)-enriched categories to induce this functor pointwise from \(\mathfrak{C}_\Delta\).  We then work to give an explicit calculation of this functor on representables and more generally on \(\psh{\C}\)-labeled simplices.  

We use the pointwise characterization of this realization to straightforwardly extend the results of Dugger and Spivak \cite{ds1} on alternative realizations to our setting, while on the other hand, we make use of the explicit characterization to demonstrate directly that the coherent realization and its right adjoint, the coherent nerve, form a Quillen pair
\[
	\mathfrak{C}:\cellset_{\mathrm{hJoyal}} \rightleftarrows \Cat_{\spsh_{\mathrm{inj}}}:\mathfrak{N}.
\]

For the next step in this chapter, we introduce cosimplicial resolutions in order to compute mapping objects for formal \(\C\)-quasicategories.  We extend ideas from \cite{ds2} to demonstrate that the coherent nerve and realization actually specify a Quillen equivalence.

\subsection*{The Coherent Nerve, Local case}
In the third and final chapter, we give a way to perform a left-Bousfield localization of the horizontal Joyal model structure with respect to Cartesian presentations of the form \((\C,\setS)\) (though still under the hypothesis that \(\C\) is regular Cartesian Reedy).  The local objects are exactly the formal \(\C\)-quasicategories whose mapping objects are \(\setS\)-local.  Using our comparison theorem with Rezk's model structure, we can apply his result to show that this model structure is again Cartesian monoidal.  

To prove the main result of the paper, we use the compatibility of the coherent realization and nerve with the formation of mapping objects to demonstrate that they remain Quillen equivalences after simultaneous localization. 

As a corollary of the main result, we apply a theorem of Lurie to construct a Yoneda embedding.  We then demonstrate that it is fully faithful and also prove Yoneda's lemma, which we then leverage to define representability.  From representability, we propose a definition of weighted limits and colimits.  

\subsection*{Appendix: Recollections on Cisinski Theory}
Throughout this paper, we will make extensive use of the extremely elegant theory of Cisinski from \cite{cisinski-book}, which allows for the construction and description of model structures on presheaf categories in which the cofibrations are exactly the monomorphisms.  As such, we will recall several key results:

In the first section of the appendix, we will need to recall how to generate Cisinski model structures by anodyne closure with respect to a cellular model, a separating cylinder functor, and a small set of injective maps of presheaves.  We will also demonstrate how this plays into the theory of Cartesian-monoidal Cisinski model categories. A theorem of Cisinski demonstrates that taking an empty set of generating anodynes together with an injective separating interval object generates the minimal Cisinski model structure on a presheaf category. In particular, this will always exist by taking this object to be the subobject classifier.

We will then recall how the existence of a minimal Cisinski model structure gives rise to the theory of localizers by applying left Bousfield-localization.  This theory generalizes the theory of presentation by generating anodynes.  In particular, given any small set of maps in a presheaf category, there is a closure of this set such that it generates a minimal Cisinski model structure in which those maps are weak equivalences.  Since localizers are defined by a closure operation and determine Cisinski model structures up to identity, it will be clear that Cisinski model structures arrange themselves into the structure of a poset ordered by inclusion of their localizers.

In the next section, we will recall Cisinski's theory of simplicial completion and discrete localizers.  In particular, it is a theorem of \cite{cisinski-book} that there is a Galois connection called the simplicial completion between localizers on \(\psh{\mathcal{A}}\) and localizers on \(\psh{\mathcal{A}\times \Delta}\), which restricts to a bijection above the simplicial completion of the minimal localizer on \(\psh{\mathcal{A}}\). Localizers belonging to the image of the simplicial completion are called \emph{discrete}.

We will then describe the tricky relationship between discrete localizers on \(\psh{\mathcal{A}\times \Delta}\) and Dugger presentations on \(\mathcal{A}\), which are the localizations of the injective model structure on simplicial presheaves on \(\mathcal{A}\).  In particular, we will recall the theorem of Cisinski that the simplicial completion of a localizer is also a Dugger presentation if and only if the localizer is \emph{regular}, which is an important property that ensures that every presheaf is canonically the \emph{homotopy colimit} of its representables.  An important fact is that every localizer admits a \emph{regular completion}, which is canonically generated by the regular completion of the minimal localizer together with any localizer.  It follows from this fact that the Galois connection also restricts to a bijection between discrete localizers admitting a Dugger presentation on \(\psh{\mathcal{A}\times \Delta}\) and regular localizers on \(\psh{\mathcal{A}}\).  

Finally, we will digress into the topic of chapter 8 of \cite{cisinski-book}, the theory of skeletal categories.  These are generalized Reedy categories with canonical cellular models, which, under certain combinatorial hypotheses, have a minimal localizer that is already regular.  These categories will be important in the rest of the paper, as they greatly simplify the generation of the model structures in which we are interested.  

\subsection*{Questions}
We suspect that the arguments here can be generalized to more general small categories \(\C\) by replacing the boundary inclusions of \(\C\) with a more general cellular model and by replacing the horizontal Joyal model structure with its regular completion (see \ref{regcompletion}).  All of our motivating examples satisfy the case where \(\C\) is regular Cartesian Reedy, so we haven't attempted to work in this generality. 

\subsection*{Looking forward}
A major challenge in the theory of higher categories is the problem of coherence, that is to say, defining functors and appropriately-natural transformations valued in a higher category of higher categories. It was observed early as the 1970 that a powerful way to deal with coherence problems even for functors from a \(1\)-category to the \(2\)-category of categories was to perform a rectification of that theory to the theory of Cartesian fibrations.  

Lurie extended this point of view to the theory of \((\infty,1)\)-categories for two reasons: Less crucially, one can use the theory of Cartesian fibrations to work with \((\infty,2)\)-categorical notions without ever actually giving a definition of \((\infty,2)\)-category.  Much more important than this shortcut, however, is the fact that Cartesian fibrations greatly simplify coherence problems.  

Unfortunately, our paper does not even begin to scratch the surface of the fibrational point of view, and as a consequence, it is much more difficult to work in our setting in light of the consequent coherence problems.  We expect that to understand the fibrational point of view, attempts will have to be made to understand higher-categorical lax structure.  Lax structure is better-understood in the strict setting due to recent work of Ara, Maltsiniotis, and Steiner, but all attempts thusfar to extend these highly combinatorial results to the theory of weak higher categories have produced no tangible results.  We suspect that this might change in the future when an equivalence theorem between the Complicial model of Verity and the \(\Theta\)-style model studied here is proven.

We hope also that new approaches to dealing with coherence problems might be discovered, and if they can be made to work, we expect that the results of this paper will be even more useful.  

\subsection*{Acknowledgements} First and foremost, we would like to thank George Raptis for advising us in the writing of the version of this paper to be submitted as a Master's thesis at Universit\"at Regensburg.  His suggestions have been invaluable, and we are extremely grateful for the time he has spent reading this paper closely.

We would also like to thank Denis-Charles Cisinski for his help with axiomatizing the Eilenberg-Zilber shuffle decomposition \Cref{cishelp}, which is an important technical condition without which many of the constructions in this paper would not function.  

We would like to give a special thanks Alexander Campbell for his invaluable help in correcting the proof of \Cref{joyalisothm}, which previously cited an incorrect result as well as his advice on strictifying the construction of labeled simplicial sets.

Finally, we would like to acknowledge Andrea Gagna for spending his time discussing these ideas with us over several long conversations as well as Eric Peterson for his sage advice and friendship over the years.  


\section{Formal \(\C\)-quasicategories}
\subsection{The wreath product with \(\Delta\)}
In this section, we will consider a slightly more general definition of the wreath product with \(\Delta\), as defined in \cite{oury}.  

Segal observed long ago that a monoidal category is classified precisely by a pseudofunctor \(M_\bullet:\Delta^\op\to \Cat\) such that \(M_0=\ast\) is the terminal category and the maps \(M_n \to {(M_1)}^n\) induced by the inclusion of the spine \(Sp[n]\hookrightarrow \Delta^n\) are all equivalences of categories.  

\begin{note}
	For the sake of readability of this section, we will consider all limits taken in \(\Cat\) to be the appropriate \(2\)-categorical pseudo-limits.  A \(\Cat\)-valued pseudofunctor with \(1\)-categorical domain will consequently be called continuous if it sends limits to pseudo-limits.  With this out of the way, we proceed to our first definition
\end{note}

\begin{defn}
	Suppose \(\mathcal{V}\) is a monoidal category.  Then we construct a Grothendieck fibration 
	\[
		\Delta\int\mathcal{V}\to \Delta
	\]
	by applying the Grothendieck construction to the pseudofunctor
	\[
		\mathcal{V}_\bullet:\Delta^\op\to \Cat
	\]
	classifying \(\mathcal{V}\). We call the total space of this fibration the \emph{wreath product} of \(\Delta\) with \(\mathcal{V}\).  The objects of \(\Delta\int \mathcal{V}\) can be identified with pairs \(([n],(v_1,\dots,v_n)),\) where \((v_1,\dots,v_n)\) is a tuple of objects of \(\mathcal{V}\).  We will write such an object as \([n](v_1,\dots,v_n)\).
\end{defn}
We will also make use of a more elaborate construction from \cite{oury} that extends the wreath product to arbitrary simplicial sets:
\begin{defn}
	Notice that since \(\Cat\) is conically complete, the pseudofunctor \(\mathcal{V}_\bullet\) extends essentially uniquely along the co-Yoneda embedding to a continuous pseudofunctor 
	\[
		\psh{\mathcal{V}^\op_\bullet}^\op: \psh{\Delta}^\op \to \Cat,
	\]
	which is exactly the pseudo-right Kan extension of \(\mathcal{V}_\bullet\) along the co-Yoneda embedding \(\Delta^\op \hookrightarrow \psh{\Delta}^\op\).  Applying the Grothendieck construction to the functor \(\psh{\mathcal{V}^\op_\bullet}^\op\), we define the Grothendieck fibration
	\[
		\psh{\Delta}\int \mathcal{V} \to \psh{\Delta}.
	\]
	The total space of this fibration is called the category of \emph{\(\mathcal{V}\)-labeled simplicial sets}.
	


\end{defn}
\begin{note}
	It will be useful to explicitly compute the value of \(\psh{\mathcal{V}^\op_\bullet}^\op(S)\) for a simplicial set \(S\) in somewhat simpler terms.  First, consider \(\Delta^n\) to be a discrete simplicial object in \(\Cat\), we can naturally identify \(\mathcal{V}_n\) with the category \(\Cat^{\Delta^\op}(\Delta^n, \mathcal{V}_\bullet)\) whose objects are pseudonatural transformations of simplicial objects and whose morphisms are modifications.  We can then compute
	\begin{align*}
		\psh{\mathcal{V}^\op_\bullet}^\op(S) &= \lim_{\Delta^n \in \overcat{\Delta}{S}} \mathcal{V}_n\\
		&\simeq  \lim_{\Delta^n \in \overcat{\Delta}{S}} \Cat^{\Delta^\op}(\Delta^n, \mathcal{V}_\bullet)\\
		&\simeq \Cat^{\Delta^\op}(\colim_{\Delta^n \in \overcat{\Delta}{S}} \Delta^n, \mathcal{V}_\bullet)
		\intertext{since the cosimplicial object \(\Delta^\bullet\) in \(\Cat^{\Delta^\op}\) is Reedy-cofibrant, and therefore}
		&\simeq \Cat^{\Delta^\op}(S, \mathcal{V}_\bullet).
	\end{align*}
	In particular, we can identify the category \(\psh{\mathcal{V}^\op_\bullet}^\op(S)\) with the category whose objects are pseudonatural transformations \(\Omega:S\to \mathcal{V}_\bullet\) and whose morphisms are modifications, viewing \(S\) as a simplicial object in \(\Cat\).  However, since \(S_n\) is a discrete category for every \(n\), every pseudonatural transformation is in fact isomorphic to a natural transformation.
	
	It follows that the objects of \(\psh{\Delta}\int \mathcal{V}\) with pairs \((S,\Omega)\) consisting of a simplicial set \(S\) and a natural transformation \(\Omega:S\to \mathcal{V}_\bullet\).  
\end{note}
\begin{note}
	It is possible, by careful application of coherence results, to rectify everything in sight.  First, notice that if \(\mathcal{V}_\bullet\) is Reedy-fibrant with respect to the canonical model structure on \(\Cat\), we can compute the pseudolimit as a strict limit in the \(1\)-category \(\Cat\) while also replacing pseudonatural transformations with strict ones.  
	
	This raises the question of how to obtain a Reedy-fibrant \(\mathcal{V}_\bullet\) from a monoidal category \(\mathcal{V}\).  For this, consider the monoidal category \(\mathcal{V}\) as a one-object bicategory and apply the \(2\)-nerve of Lack and Paoli \cite{lack-paoli}.  This produces a simplicial category whose object in degree \(0\) is the terminal category and whose object in degree \(1\) is in fact isomorphic to \(\mathcal{V}\).  This simplicial object is also Reedy-fibrant and satisfies the Segal condition.  We can unwind \(\mathcal{V}_\bullet\) as follows:
	\begin{itemize}
		\item The objects of \(\mathcal{V}_n\) are the normal pseudofunctors \([n] \to \mathbf{B}\mathcal{V}\), where \(\mathbf{B}\mathcal{V}\) denotes the associated single-object bicategory.
		\item The morphisms are given by \emph{icons} between pseudofunctors.  These are lax natural transformations whose object components are identities.
	\end{itemize}
	Unraveling this construction, we have a natural \emph{isomorphism} of categories 
	\[
		\operatorname{Nat}(\Delta^n, \mathcal{V}_\bullet) \cong \mathcal{V}_n.
	\]
	and using the Reedy-fibrancy of \(\mathcal{V}_\bullet\), we also have a natural \emph{isomorphism} of categories
	\[
		\operatorname{Nat}(S,\mathcal{V}_\bullet) \cong \psh{\mathcal{V}^\op_\bullet}^\op(S),
	\]
	where \(\operatorname{Nat}\) denotes the category of strict natural transformations and modifications between them.
\end{note}
\begin{prop} The pullback of the fibration 
	\[
		\psh{\Delta}\int \mathcal{V} \to \psh{\Delta}
	\]
	along the Yoneda embedding \(\Delta\hookrightarrow \psh{\Delta}\) is exactly the fibration 
	\[
		\Delta\int\mathcal{V}\to \Delta,
	\]
	and therefore, the induced map
	\[
		\Delta\int\mathcal{V}\hookrightarrow \psh{\Delta}\int\mathcal{V}
	\]
	is a fully faithful embedding.
\end{prop}
\begin{proof}  As the functor \(\mathcal{V}_\bullet\) factors as the composite 
	\[\Delta^\op \hookrightarrow \psh{\Delta}^\op \xrightarrow{\psh{\mathcal{V}^\op_\bullet}^\op} \Cat,\] where the first functor is fully faithful, it follows that \(\Delta\int \mathcal{V} \to \Delta\) is the pullback of the fibration \(\psh{\Delta}\int \mathcal{V} \to \psh{\Delta}\) along the fully faithful Yoneda embedding.  Ergo, the map in question is fully faithful.  
\end{proof}

For the purposes of this paper, we do not need this level of generality.  We specialize as follows: 

\begin{defn}\label{cishelp}
	A small regular skeletal Reedy category (also called a regular skeletal category) \(\C\) (see \Cref{regskelcat}) is called a \emph{regular Cartesian Reedy category} if it satisfies two conditions:

	\begin{enumerate}[leftmargin=5em,label=(CR\arabic*{})]
		\item The class of regular presheaves on \(\C\) (see \Cref{regskelcat}) is closed under finite products.
		\item If \(I\) is a finite set and \(c\to \prod\limits_{i\in I }c_i \) is a nondegenerate section, then \(\dim c \leq \sum\limits_{i\in I}\dim c_i\).
	\end{enumerate}
\end{defn}

\begin{rem}
	The axioms for regular Cartesian Reedy categories imply that \(\C\) is a Reedy multicategory in the sense of \cite{bergner-rezk-reedy}.  It also asserts a weak form of the Eilenberg-Zilber shuffle decomposition. There may be a way to prove (CR2) from (CR1), but we were unable to do so.
\end{rem}

In the sequel, we assume that \(\mathcal{V}=\psh{\C}\) is the category of presheaves of sets on a small regular Cartesian Reedy category \(\C\) admitting a terminal object.  Then we give the following definition:

\begin{defn} For any regular Cartesian Reedy category \(\C\) admitting a terminal object, we define the category of \(\C\)-cells
	\[
		\Theta[\C]\subseteq \Delta \int \psh{\C}
	\]
	to be the full subcategory spanned by the objects of the form \([n](h_{c_1},\cdots, h_{c_n})\) for \(c_1,\cdots, c_n \in \C\), and where \(h_\bullet\) denotes the Yoneda embedding.
\end{defn}

\begin{rem}
	The requirement that \(\C\) have a terminal object is a technical condition that ensures that \(\Delta\) embeds fully and faithfully in \(\Theta[\C]\).  The condition that \(\C\) is regular Cartesian Reedy is probably not necessary, but it will ensure later on that the generating cofibrations of the injective model structure on simplicial presheaves \(\spsh\) admit a very simple description.
\end{rem}

\begin{rem} 
	For any small category \(\C\), we have a projection functor \(\pi:\Theta[C]\to \Theta[\ast]=\Delta\) sending an object to the associated underlying simplex.  
	
	When \(\C\) has a terminal object, the inclusion \(\ast\to \C\) is a fully faithful right-adjoint to the terminal functor.  It can easily be seen that the construction \(\Theta[\cdot]\) preserves fully faithful right-adjoints, giving us an adjunction
	\[
		\pi: \Theta[\C]\rightleftarrows \Theta[\ast]=\Delta: \eta.
	\]
	Passing to presheaf categories, these functors also extend to a quadruple adjunction by a routine calculation of Kan extensions.  However, we will only name and make use of three of the four adjoints.     
	\begin{center}
		\begin{tikzpicture}
			\matrix (m) [matrix of math nodes, row sep=3em,column sep=3em]
			{ \cellset & \psh{\Delta} \\};
			\path[->,font=\scriptsize]
			(m-1-1) edge[transform canvas={yshift=1.2em}]   node[auto]{\(\scriptstyle \pi\)} node[auto,swap,transform canvas={yshift=0.2em}]{\(\scriptstyle{\perp}\)} (m-1-2)
			(m-1-2) edge node[auto,swap, transform canvas={yshift=-0.2em}]{\(\scriptstyle \mathscr{H}\)} (m-1-1)
			(m-1-1) edge[transform canvas={yshift=-1.2em}]   node[auto, swap]{\(\scriptstyle \mathscr{N}\)} node[auto]{\(\scriptstyle{\perp}\)} (m-1-2);
			\end{tikzpicture},
	\end{center}
	where, by abuse of notation, we denote the \emph{simplicial projection} functor \(\pi_!\) simply by \(\pi\), we denote the \emph{local termination} functor \(\pi^\ast=\eta_!\) by \(\mathscr{H}\), and we denote the \emph{underlying simplicial set} functor \(\pi_\ast=\eta^\ast\) by \(\mathscr{N}\). 
\end{rem}

\begin{defn} We define a special cosimplicial object in \(\cellset\) by the formula \[E^\bullet=\mathscr{H}(\operatorname{cosk}_0\Delta^\bullet).\]  This cosimplicial object will be a cosimplicial resolution of a point, once we define our model structures.
\end{defn}

\subsection{The generalized intertwiner and \(\ssetlab\)}\label{sec:intertwiner}
Rezk introduced a functor called the \emph{intertwiner} by means of an explicit construction in \cite{rezk-theta-n-spaces}, but Oury has given an even more powerful version in \cite{oury}, which we recall here:

\begin{defn} Recall that we have a fully-faithful embedding \[L:\Theta[\C]\hookrightarrow \Delta\int\psh{\C}\hookrightarrow \ssetlab.\]  We define the \emph{intertwiner} to be the restricted Yoneda functor 
	\[
		\square:\ssetlab \to \cellset
	\]
	by the formula 
	\[
		(S,\Omega)\mapsto S\square\Omega=\Hom_{\ssetlab}(L(\cdot), (S,\Omega)).
	\]
\end{defn}

\begin{note} The restriction of the intertwiner to \(\Delta \int \psh{\C}\) is exactly the intertwiner of Rezk.  When we apply the intertwiner to an object belonging to the full subcategory \(\Delta\int \psh{\C}\), that is, \((S,\Omega)=[n](A_1,\dots, A_n)\), we will switch to Rezk's notation, namely \[V[n](A_1,\dots,A_n)\defeq S\square \Omega\]
\end{note}

\begin{defn} An object \((S,\Omega)\) of \(\ssetlab\) is called \emph{normalized} if the image of the component \(\Omega_1\) of the natural transformation \(\Omega\) does not contain the empty presheaf on \(\C\).
\end{defn}

\begin{prop} The restriction of the intertwiner to the full subcategory of normalized objects in \(\ssetlab\) is fully faithful.
\end{prop}
\begin{proof}
	Recall before we begin that a map \((S,\Omega)\to (S^\prime,\Omega^\prime)\) is given by a morphism of simplicial sets \(f:S\to S^\prime\) and a natural modification \(\zeta:\Omega\to \Omega^\prime\circ f\).

	In order to prove fullness, let \(\gamma:S\square\Omega\to S^\prime\square\Omega^\prime\) be a map in \(\cellset\). We notice that \(\Hom([n](\varnothing,\dots,\varnothing), (S,\Omega))\) is naturally isomorphic to \(S_n\), and proceed by diagram chase. Since by assumption \((S,\Omega)\) is normalized, every map \([n](\varnothing,\dots,\varnothing)\to (S,\Omega)\) factors through at least one map \([n](c_1,\dots,c_n)\to (S,\Omega)\).  

	Choosing such a factorization, the natural transformation \(\gamma\) sends this to a map \[[n](c_1,\dots,c_n)\to (S^\prime,\Omega^\prime),\] and finally, precomposing this map with the unique map \([n](\varnothing,\dots,\varnothing)\to [n](c_1,\dots,c_n)\), we obtain a map \([n](\varnothing,\dots,\varnothing)\to (S^\prime,\Omega^\prime)\). Taking these together gives a map \(S_n\to S^\prime_n\), naturally in \(n\).

	Now assume that \(S^\prime=S\) and that the map induced by \(\gamma\) is the identity.  Then notice that a map \[[n](c_1,\dots,c_n)\to (S,\Omega)\] is completely determined by its action on the degree \(n\) part, but this amounts to picking an \(n\)-simplex of \(S\) together with its labeling \((A_1,\dots,A_n)\), and a map \((c_1,\dots,c_n)\to (A_1,\dots,A_n)\).  Then the natural transformation gives a natural map \(((A_1)_{c_1},\dots,(A_n)_{c_n})\to ((A^\prime_1)_{c_1},\dots,(A^\prime_n)_{c_n})\), taking the naturality in \(n\) and the \(c_i\), these together determine a natural modification \(\Omega\to \Omega^\prime\).

	To see faithfulness, notice that the construction in the proof of fullness defines a left-inverse to the definition of the map on morphisms defined by the intertwiner.
\end{proof}

\begin{defn} We call a presheaf of sets on \(\Theta[\C]\) a \emph{\(\C\)-cellular set}.
\end{defn}

\begin{note} Although the case when \(\C=\Theta_{n-1}\) (respectively \(\C=\Theta=\Theta_\omega\)) are not strictly the focus of this paper, note that \(\Theta[\Theta_{n-1}]=\Theta_n\) (respectively \(\Theta[\Theta]=\Theta\)).  In these cases, we call presheaves of sets on \(\Theta[\C]\) \emph{\(n\)-cellular sets} (respectively, \emph{cellular sets}).  
\end{note}

\begin{defn} We say that a \(\C\)-cellular set \(X\) is \emph{sober} if it is the image of a normalized object of \(\ssetlab\). If \(f:X\to Y\) is the image under the intertwiner of a Cartesian map of normalized labeled simplicial sets, we call \(f\) \emph{Cartesian}.
\end{defn}

\begin{prop} All representable \(\C\)-cellular sets are sober.
\end{prop}
\begin{proof} By construction.
\end{proof}

\begin{lemma}
	The category \(\ssetlab\) has finite products.
\end{lemma}
\begin{proof}
	Then we define the Cartesian product of \((S,\Omega)\) and \((S^\prime,\Omega^\prime)\) by the formula 
	\[
		S\times S^\prime \xrightarrow{\Omega \times \Omega^\prime} \psh{\C}_\bullet \times \psh{\C}_\bullet \xrightarrow{\times} \psh{\C}_\bullet.
	\]  It is clear that this satisfies the universal property of the product.
\end{proof}
\begin{lemma} Given a simplicial set \(S\), the category of labelings of \(S,\) that is, the fibre \(\left(\ssetlab\right)_S\) is closed under finite Cartesian products. 
\end{lemma}
\begin{proof} Given two labelings \((S,\Omega)\) and  \((S,\Omega^\prime)\) the product of their labelings in the fibre \(\left(\ssetlab\right)_S\) can be given by the formula 
	\[S\xrightarrow{\Delta} S\times S \xrightarrow{\Omega\times \Omega^\prime}  \psh{\C}_\bullet \times \psh{\C}_\bullet \xrightarrow{\times} \psh{\C}_\bullet.\]  We leave the verification to the reader.
\end{proof}

\begin{prop} The class of sober \(\C\)-cellular sets is closed under Cartesian product.
\end{prop}
\begin{proof} From the construction of the intertwiner, we see that since \(\ssetlab\) has all Cartesian products, the intertwiner preserves them, since
	\[
		\Hom_{\ssetlab}(L(\cdot), (S,\Omega)\times (S^\prime,\Omega^\prime))=\Hom_{\ssetlab}(L(\cdot), (S,\Omega))\times \Hom_{\ssetlab}(L(\cdot), (S^\prime,\Omega^\prime)).
	\] 
\end{proof}

\begin{prop}\label{pullbacksober}
	The projection functor \(\Theta[\C]\to \Delta\) induces an adjunction 
	\[
		\cellset \underset{\mathscr{H}}{\overset{\pi}{\rightleftarrows}} \psh{\Delta},
	\]
	(as we saw earlier). If \(X=S\square\Omega\) is sober, and \(f:S^\prime\to S\) is a map of simplicial sets, then the image of the Cartesian lift \(\tilde{f}:(S^\prime, f^\ast(\Omega))\to (S,\Omega)\) under the intertwiner is exactly the pullback of \(\mathscr{H}(f)\) along the component at \(X\) of the unit of the adjunction \(\mu_X:X\to \mathscr{H}\pi X= \mathscr{H} S\).
\end{prop}
\begin{proof} 
	By inspection of the definition of \(S\square\Omega\), we can see that \(\pi(S\square\Omega)=S\).

	We can see that \(\mathscr{H}\) factors as \(\square\circ \mathfrak{t}\), where \(\mathfrak{t}\) is the right-adjoint to the projection \(\ssetlab\to \psh{\Delta}\), which exists by explicit computation as the functor sending the simplicial set \(S\) to the object \((S,\Omega_{\mathfrak{t}})\) where \(\Omega_{\mathfrak{t}}\) is the labeling sending all simplices of \(S\) to the terminal presheaf on \(\C\).   We can see that the pullback of \((S,\Omega)\) along \(f:S^\prime\to S\) satisfies the universal property of the fibre product of the unit map \((S,\Omega)\to \mathfrak{t}(S)\) with the map \(\mathfrak{t}(f)\), so such pullbacks exist in \(\ssetlab\) and are obviously preserved by \(\square\), which by construction preserves whatever limits exist. The proposition follows immediately from these two observations.
\end{proof}

\subsection{The horizontal Joyal model structure}
We define a Cisinski model structure on \(\cellset\) and state several results that will be proven over the next few sections.

\begin{defn}\label{modelstrucdefn}
	There is a Cisinski model structure called the \emph{horizontal Joyal model structure} on \(\cellset\) where the separating interval is given by 
	\[E^1=\mathscr{H}(\operatorname{cosk}_0 \Delta^1),\]
	which is also isomorphic to \(N(G_2)\square \Omega_{\mathfrak{t}}\), where \(G_2\) is the freestanding isomorphism and \(\Omega_{\mathfrak{t}}\) is the terminal labeling of its nerve. 
	The set of generating anodynes is given by
	\[
		\mathscr{J}=\{\square_n^\lrcorner(\lambda^n_k,\delta^{c_1},\dots,\delta^{c_n}) : 0<k<n \text{ and } c_1,\dots,c_n \in \Ob \C\},
	\]
	where \(\lambda^n_k:\Lambda^n_k\hookrightarrow \Delta^n\) is the simplicial horn inclusion, and where \(\delta^c:\partial c \hookrightarrow c\) is the inclusion of the boundary of \(c\) (recall that \(\C\) was taken to be a regular Cartesian Reedy category, so this makes sense).

	We call \(\operatorname{rlp}(\mathscr{J})\) the class of \emph{horizontal inner fibrations}, and we call \(\operatorname{llp}(\operatorname{rlp}(\mathscr{J}))\) the class of \emph{horizontal inner anodynes}.
\end{defn}

\begin{rem}
	The precise definition and construction of the corner-intertwiner \(\square^\lrcorner_n\) is deferred to \Cref{cornertensor}, but in this particular case, we can compute it by hand in terms of the intertwiner to be \[V_{\Lambda^n_k}(c_1,\dots,c_n) \cup \left(\bigcup_{i=1}^n V[n](c_1,\dots,\partial c_i \dots, c_n) \right) \hookrightarrow [n](c_1,\dots,c_n),\] where \(V_{\Lambda^n_k}(c_1,\dots,c_n)\) is the pullback of \([n](c_1,\dots,c_n)\) by the inclusion \(\Lambda^n_k\hookrightarrow \Delta^n\) (whenever \(K\subseteq \Delta^n\), we can apply this formula to compute the corner tensor).
\end{rem}

\begin{defn}
	We call an object with the right lifting property with respect to \(\mathscr{J}\) a \emph{formal \(\C\)-quasicategory}.
\end{defn}

\begin{note} 
	In the case where \(\C\) is the terminal category, these are precisely the quasicategories, since the horns in the definition above become exactly the simplicial inner horn inclusions.
\end{note}

The following results are stated here without proof.  All proofs are heavily inspired by \cite{oury} and provided in full in \Cref{reedy}, \Cref{horizontal}, and \Cref{admissible}.

\begin{prop}
	The class of all monomorphisms of \(\cellset\) is exactly \(\operatorname{Cell}(\mathscr{M}),\) where \[\mathscr{M}=\{\square_n^\lrcorner(\delta^n,\delta^{c_1},\dots,\delta^{c_n}) : n\geq 0 \text{ and } c_1,\dots,c_n \in \Ob \C\},\]
	where \(\delta^n:\partial \Delta^n \hookrightarrow \Delta^n\) is the inclusion of the boundary.
\end{prop}

\begin{prop}
	For any inner anodyne inclusion \(\iota:K\hookrightarrow \Delta^n\) and any family \(f_1,\dots,f_n\) of monomorphisms of \(\psh{\C}\), the map \[\square^\lrcorner_n(\iota,f_1,\dots,f_n)\] is horizontal inner anodyne.
\end{prop}

\begin{thm} The horizontal Joyal model structure is Cartesian-closed, and in particular, \[\operatorname{Cell}(\mathscr{M})\times^\lrcorner \operatorname{Cell}(\mathscr{J}) \subseteq \operatorname{Cell}(\mathscr{J}).\]
\end{thm}

\begin{thm}\label{isofibrations}
	A horizontal inner fibration between formal \(\C\)-quasicategories is a fibration for the horizontal Joyal model structure if and only if it has the right lifting property with respect to the map \(\Delta^0\hookrightarrow E^1\).  In particular, the formal \(\C\)-quasicategories are exactly the fibrant objects for the horizontal Joyal model structure.
\end{thm}

\subsection{The corner tensor construction}\label{cornertensor} 
The overwhelming majority of this section is due to Oury, although we had to redo some of the proofs, since they contained mistakes. Following \cite{oury}*{3.1} we define the corner tensor, a vast generalization of the corner product.
\begin{defn}
	Let \(\mathcal{V}\) be a cocomplete symmetric monoidal closed category, and let all categories, functors, and natural transformations in what follows to be \(\mathcal{V}\)-enriched.  Suppose we have a category \(\mathcal{T}\) and an \(n\)-ary functor \[\wedge:\mathcal{T}^{\otimes n} \to \mathcal{T}.\]  Let \((\mathcal{A}_i)_{i=1}^n\) and \(\mathcal{D}\) be categories admitting enough colimits such that all tensors with \(\Hom_T\) exist and coends over \(\mathcal{T}\) exist.  Let
	\[\nabla_i: \mathcal{A}_i \otimes \mathcal{A}_i \to \mathcal{A}_i, \qquad \nabla:\mathcal{D}\otimes \mathcal{D} \to \mathcal{D}, \qquad \square:\mathcal{A}_1\otimes\dots \otimes \mathcal{A}_n \to \mathcal{D}\]
	be functors.  Then we define the following functors:
	\[\nabla^\lrcorner_i: \mathcal{A}^\mathcal{T}_i \otimes \mathcal{A}^\mathcal{T}_i \to \mathcal{A}_i^\mathcal{T}, \qquad \nabla^\lrcorner:\mathcal{D}^\mathcal{T}\otimes \mathcal{D}^\mathcal{T} \to \mathcal{D}^\mathcal{T}, \qquad \square^\lrcorner: \mathcal{A}_1^\mathcal{T}\otimes\dots \otimes \mathcal{A}_n^\mathcal{T} \to \mathcal{D}^\mathcal{T}\]
	by the Day convolution, for example,
	\[\square^\lrcorner(M_1,\dots,M_n)(t) = \int^{u_1,\dots,u_n \in \mathcal{T}} \mathcal{T}(\wedge(u_1,\dots,u_n),t) \otimes \square(M_1(u_1),\dots,M_n(u_n).\]
\end{defn}
\begin{lemma}
	The functor \(\square^\lrcorner\) preserves all colimits preserved by \(\square\) in each variable and by \(\otimes\) in its second variable.  
\end{lemma}
\begin{proof} By coend manipulation.
\end{proof}
We specialize now to the case where \(\mathcal{V}\) is just the category of sets and where \(\mathcal{T}=[1]\) is the categorical \(1\)-simplex.  The functor \(\wedge:[1]^n \to [1]\) is given by taking the infimum.   
\begin{note}
	We will often consider arrows in a category \(\mathcal{D}\) as functors \([1]\to \mathcal{D}\).  By abuse of notation, we will denote the functor \([1]\to \mathcal{D}\) classifying an arrow \(f:A\to B\) simply by \(f\).  We will denote the evaluation of this functor on the objects of \(x\in [1]\) by \(f(x)\), such that in the case of a map \(f:A\to B\),
	\begin{eqnarray}
		f(0) = A &\text{and}& f(1) = B
	\end{eqnarray} 
\end{note}
\begin{defn}
	Given \(\square:\mathcal{A}_1\times \dots \times \mathcal{A}_n \to \mathcal{D}\), where each category appearing is cocomplete, we define the \emph{corner tensor} \(\square^\lrcorner:\mathcal{A}_1^{[1]}\times \dots \times \mathcal{A}_n^{[1]} \to \mathcal{D}^{[1]}\) by the formula
	\[\square^\lrcorner(f_1,\dots,f_n)(t)=\int^{u_1,\dots,u_n\in [1]} [1](u_1\wedge\dots\wedge u_n, t) \cdot \square(f_1(u_1),\dots,f_n(u_n)).\]
	If \((g,h): f_i\to f^\prime_i\) is a commutative square, let 
	\[
		(g^\lrcorner,h^{-}):\square^\lrcorner(f_1,\dots,f_i,\dots,f_n) \to \square^\lrcorner(f_1,\dots,f^\prime_i,\dots,f_n)
	\]
	be the induced commutative square.  
\end{defn}
\begin{note}
	Let \([1]^n\) be the \(n\)-fold power of the poset \([1]\), which is a cube, and let \(C_n=[1]^n - \{(1,\dots,1)\}\) be the subposet of the cube removing the terminal vertex.  To unwind the coend, notice that the set \([1](u_1\wedge \dots \wedge u_n,0)\) vanishes when all of the \(u_i=1\).  We can therefore evaluate the domain of the corner tensor as the colimit of the restriction
	\[
		\square^\lrcorner(f_1,\dots,f_n)(0)=\colim \left.\square(f_1,\dots,f_n)\right|_{C_n}.
	\]
	The codomain of the corner tensor can be computed by noticing that the set \([1](u_1\wedge \dots \wedge u_n,1)\) is always a singleton, and therefore the colimit can be computed simply as the colimit of the functor \(\square(f_1,\dots,f_n):[1]^n \to \mathcal{D}\).  However, this colimit is indexed by a category with a terminal object and therefore agrees with the the evaluation at that terminal object. That is, we have
	\[
		\square^\lrcorner(f_1,\dots,f_n)(1)=\square(f_1(1),\dots,f_n(1)).
	\]  
\end{note}
\begin{example}
	If we take \(\square\) to be a bifunctor \(\mathcal{D}\times \mathcal{D} \to \mathcal{D}\) in an appropriately cocomplete category, then given \(f_1:A\to B\) and \(f_2:C\to D\), their corner tensor is the familiar corner product: 
	\[
		f_1 \square^\lrcorner f_2 = \left(A\square D \coprod_{A\square C} B\square C \to  B\square D\right).
	\]
\end{example}
\begin{example}
	In the category of \(n\)-fold multisimplicial sets \(\psh{(\Delta)^n}\), we have an \(n\)-fold exterior product functor sending an \(n\)-tuple of simplicial sets \((S_1,\dots,S_n)\) to the exterior product \(\square(S_1,\dots,S_n)\).  It can be seen that the exterior product preserves colimits argument-by-argument, so applying the corner tensor, we can compute exterior corner products of maps.  It is a fact beyond the scope of this paper that the Cisinski model structure on multisimplicial sets that models the homotopy theory of spaces has cellular generating cofibrations given by 
	\[
		\square^\lrcorner(\delta^{m_1},\dots, \delta^{m_n})
	\]
	where \(\delta^m:\partial \Delta^m \hookrightarrow \Delta^m\) denotes the boundary inclusion.  The generating anodynes are given by 
	\[
		\square^\lrcorner(\delta^{m_1},\dots,\lambda^{m_i}_k,\dots, \delta^{m_n})
	\]
	where \(i\in \{1,\dots,n\}\), \(m_i >0\), and \(k\in \{0,\dots, m_i\}\), and where \(\lambda^{m_i}_k:\Lambda^{m_i}_k \hookrightarrow \Delta^{m_i}\) is any horn inclusion.  This generalizes the description of the generating anodynes and cofibrations for the Cisinski model structure on bisimplicial sets that models the homotopy theory of spaces.  
\end{example}
\begin{rem}
	The previous example is in some sense universal, and it will allow us to reduce certain questions about big corner tensors to binary ones. In particular, we will frequently use the observation that, in the situation of the example the map \(f_1(0) \square f_2(1) \to (f_1 \square^\lrcorner f_2)(0)\) is a pushout of the map \(f_1(0) \square f_2(0) \to f_1(1) \square f_2(0)\).
\end{rem}
\begin{lemma}\label{corneridentities}
	If for any \(i\in \{0,\dots,n\}\), the map \(f_i:A_i\to B_i\) is an identity map, then the corner tensor \(\square^\lrcorner(f_1,\dots,f_n)\) is an identity map.
\end{lemma}
\begin{proof}
	Assume for simplicity that \(i=1\), which is without loss of generality by reindexing.  Then the result follows by setting 
	\[
		U(s,t)=\int^{u_1,\dots,u_n} \left([1](u_1,s) \times [1](u_2\wedge\dots\wedge u_n,t)\right)\cdot \square(f_1(u_1),f_2(u_2),\dots, f_n(u_n)).
	\]
	Notice that by Yoneda reduction,
	\[
		\int^{s,t} [1](s\wedge t,x) \times [1](u_1,s) \times [1](u_2\wedge\dots\wedge u_n,t) = [1](u_1 \wedge u_2 \wedge \dots\wedge u_n,x),
	\]
	so we have that 
	\[
		\square^\lrcorner(f_1,\dots,f_n)(x)=\int^{s,t} [1](s\wedge t,x) U(s,t),
	\]
	which exhibits 
	\[
		\square^\lrcorner(f_1,\dots,f_n)(0)=U(0,1)\coprod_{U(0,0)} U(1,0),
	\]
	but \(U(0,1)=\square(B_1,\dots,B_n)\), which demonstrates that the map \(U(0,0) \to U(1,0)\) is the identity by cofinality. The map \(U(0,1) \to \square^\lrcorner(f_1,\dots,f_n)(0)\) is the pushout of an identity map, and the map \(U(0,1)\to \square^\lrcorner(f_1,\dots,f_n)(1)\) is also the identity, so it follows that the map \(\square^\lrcorner(f_1,\dots,f_n)\) must also be the identity.
\end{proof}

\begin{note}
	The next few lemmata involve very complicated diagrams.  As such, we will leave the objects appearing in these diagrams as implicit.
\end{note}
\begin{lemma}\label{cornertwist}
	Suppose again that \(\square\) preserves pushouts in its first argument, and suppose we have a coCartesian square in \(\mathcal{A}\)
	\begin{center}
		\begin{tikzpicture}
			\matrix (b) [matrix of math nodes, row sep=3em, column sep=5em]
			{
				\cdot & \cdot \\
				\cdot & \cdot \\
			};
			\path[->]
			(b-1-1) edge node[auto]{\(\scriptstyle{h}\)} (b-1-2) edge node[auto]{\(\scriptstyle{u}\)} (b-2-1)
			(b-2-1) edge node[auto]{\(\scriptstyle{g}\)} (b-2-2)
			(b-1-2) edge node[auto]{\(\scriptstyle{v}\)} (b-2-2);
		\end{tikzpicture},
	\end{center}
	and suppose that we have a commutative square
	\begin{center}
		\begin{tikzpicture}
			\matrix (b) [matrix of math nodes, row sep=3em, column sep=5em]
			{
				\cdot & \cdot \\
				\cdot & \cdot \\
			};
			\path[->]
			(b-1-1) edge node[auto]{\(\scriptstyle{g}\)} (b-1-2) edge node[auto]{\(\scriptstyle{p}\)} (b-2-1)
			(b-1-2) edge node[auto]{\(\scriptstyle{q}\)} (b-2-2)
			(b-2-1) edge node[auto]{\(\scriptstyle{\id}\)} (b-2-2);
		\end{tikzpicture}.
	\end{center}
	Given family of maps \(\mathbf{f}=(f_i \in \mathcal{A}_i^{[1]})_{i=2}^n\), let \(Q_{\bullet,\mathbf{f}}\) denote the evaluation of \(\square^\lrcorner(\bullet,\mathbf{f})\) at \(0\). Then \(g^\lrcorner: Q_{p,\mathbf{f}} \to Q_{q,\mathbf{f}}\) is a pushout of \(\square^\lrcorner(h,\mathbf{f})\).  
\end{lemma}
\begin{proof} The data allow us to construct a commutative cube 
	\begin{center}
		\begin{tikzpicture}
			\matrix (b) [matrix of math nodes, row sep=3em, column sep=5em]
			{
				\cdot &&& \cdot\\
				&\cdot & \cdot& \\
				&\cdot & \cdot& \\
				\cdot &&& \cdot\\
			};
			\path[->]
			(b-1-1) edge node[auto]{\(\scriptstyle{\id}\)} (b-1-4) edge node[auto]{\(\scriptstyle{qv}\)} (b-4-1)
			(b-1-4) edge node[auto]{\(\scriptstyle{qv}\)} (b-4-4)
			(b-2-2) edge node[auto,swap]{\(\scriptstyle{h}\)} (b-1-1) edge node[auto]{\(\scriptstyle{h}\)} (b-2-3) edge node[auto]{\(\scriptstyle{u}\)} (b-3-2)
			(b-2-3) edge node[auto]{\(\scriptstyle{\id}\)} (b-1-4) edge node[auto]{\(\scriptstyle{v}\)} (b-3-3)
			(b-3-2) edge node[auto]{\(\scriptstyle{g}\)} (b-3-3) edge node[auto]{\(\scriptstyle{p}\)} (b-4-1)
			(b-3-3) edge node[auto,swap]{\(\scriptstyle{q}\)} (b-4-4)
			(b-4-1) edge node[auto]{\(\scriptstyle{\id}\)} (b-4-4);
		\end{tikzpicture}.
	\end{center}
	Since the front and back faces of this cube are coCartesian, it gives a coCartesian square in \(\mathcal{A}^{[1]}\),
	\begin{center}
		\begin{tikzpicture}
			\matrix (b) [matrix of math nodes, row sep=2em, column sep=3em]
			{
				h & \id \\
				p & q \\
			};
			\path[->]
			(b-1-1) edge node[auto]{\(\scriptstyle{(h,\id)}\)} (b-1-2) edge node[auto,swap]{\(\scriptstyle{(u,qv)}\)} (b-2-1)
			(b-1-2) edge node[auto]{\(\scriptstyle{(v,qv)}\)} (b-2-2)
			(b-2-1) edge node[auto]{\(\scriptstyle{(g,\id)}\)} (b-2-2);
		\end{tikzpicture}.
	\end{center}
	Then applying \(\square^\lrcorner(\bullet,\mathbf{f})\), we have a commutative cube
	\begin{center}
		\begin{tikzpicture}
			\matrix (b) [matrix of math nodes, row sep=3em, column sep=5em]
			{
				\cdot &&& \cdot\\
				&Q_{h,\mathbf{f}} & Q_{\id,\mathbf{f}}& \\
				&Q_{p,\mathbf{f}} & Q_{q,\mathbf{f}}& \\
				\cdot &&& \cdot\\
			};
			\path[->]
			(b-1-1) edge node[auto]{\(\scriptstyle{\id^-}=\id\)} (b-1-4) edge node[auto]{\(\scriptstyle{qv^-}\)} (b-4-1)
			(b-1-4) edge node[auto,swap]{\(\scriptstyle{qv^-}\)} (b-4-4)
			(b-2-2) edge node[auto,swap]{\(\scriptstyle{\square^\lrcorner(h,\mathbf{f})}\)} (b-1-1) edge node[auto]{\(\scriptstyle{h^\lrcorner}\)} (b-2-3) edge node[auto]{\(\scriptstyle{u^\lrcorner}\)} (b-3-2)
			(b-2-3) edge node[auto]{\(\scriptstyle{\square^\lrcorner(\id,\mathbf{f})}\)} (b-1-4) edge node[auto]{\(\scriptstyle{v^\lrcorner}\)} (b-3-3)
			(b-3-2) edge node[auto]{\(\scriptstyle{g^\lrcorner}\)} (b-3-3) edge node[auto]{\(\scriptstyle{\square^\lrcorner(p,\mathbf{f})}\)} (b-4-1)
			(b-3-3) edge node[auto,swap]{\(\scriptstyle{\square^\lrcorner(q,\mathbf{f})}\)} (b-4-4)
			(b-4-1) edge node[auto]{\(\scriptstyle{\id^-}\)} (b-4-4);
		\end{tikzpicture}.
	\end{center}	
	its front and back faces remain pushouts, since \(\square^\lrcorner\) preserves all colimits in each argument preserved by \(\square\).  Then the map \(g^\lrcorner: Q_{p,\mathbf{f}} \to Q_{q,\mathbf{f}}\)  is a pushout of the map \(h^\lrcorner:Q_{h,\mathbf{f}} \to Q_{\id,\mathbf{f}}\), but by \Cref{corneridentities}, we see that \(\square^\lrcorner(\id,\mathbf{f})=\id\), so by commutativity, it follows that \(h^\lrcorner = \square^\lrcorner(h,\mathbf{f})\), and therefore, \(g^\lrcorner\) is a pushout of \(\square^\lrcorner(h,\mathbf{f})\).  
\end{proof}
Assume in the sequel that \(\square\) preserves connected colimits in each argument.
\begin{lemma}[\cite{oury}*{Lemma 3.10}]\label{cornertensorcell}
	Let \((\mathscr{J}_i)_{i=1}^n\) be a family of sets of morphisms of each \(\mathcal{A}_i\). Then 
	\[\square^\lrcorner(\mathscr{J}_1,\dots,\operatorname{Cell}(\mathscr{J}_k),\dots,\mathscr{J}_n) \subseteq \operatorname{Cell}(\square^\lrcorner(\mathscr{J}_1,\dots,\mathscr{J}_k,\dots,\mathscr{J}_n)).\]
\end{lemma}
\begin{proof}
	Without loss of generality, we may assume \(k=1\) by symmetry of the Cartesian product in \(\Cat\). 
	Let \(f_1\) be a map in \(\mathcal{A}_1\) that belongs to \(\operatorname{Cell}(\mathscr{J}_1)\).  Then this gives the data of a cocontinous diagram \(D:\alpha \to \mathcal{A}_1\) for an ordinal \(\alpha\) with colimit \(C\) and structure maps \(\phi_i: D(i)\to C\) such that \(f_1=\phi_0\).  Additionally, for any \(i <\alpha\), the maps \(g_i:D(i)\to D(i+1)\) are pushouts of maps belonging to \(\mathscr{J}_1\).
	Let \(D^+:\alpha^\triangleright \to \mathcal{A}_1\) be the extension of \(D\) sending \(\alpha\) to \(C\) and such that \(D^+(i\to \alpha)=\phi_i\) for \(i<\alpha\).  Define \(D^2: \alpha^\triangleright \to \mathcal{A}_1^{[1]}\) by the rule \(i\mapsto \phi_i\) for \(i<\alpha\) and \(\alpha\mapsto \id_C\) and sending \(i\to i+1\) to the commutative square \((g_i,\id_c):\phi_i\to \phi_i+1\) and sending the map \(i\to \alpha\) to the commutative square \((\phi_i,\id_C):\phi_i\to \id_C\).  Then the colimit of \(D^2\) is \(\id_C\) as pictured below.
	\begin{center}
		\begin{tikzpicture}
			\matrix (b) [matrix of math nodes, row sep=2em, column sep=2em, nodes={asymmetrical rectangle}]
			{
				D(0) & D(1) & \cdots & C \\
				C & C & \cdots & C \\
			};
			\path[->]
			(b-1-1) edge node[auto]{\(\scriptstyle{g_0}\)} (b-1-2) edge node[auto]{\(\scriptstyle{\phi_0}\)} (b-2-1)
			(b-1-2) edge node[auto]{\(\scriptstyle{\phi_1}\)} (b-2-2) edge node[auto]{\(\scriptstyle{g_1}\)} (b-1-3)
			(b-1-3) edge node[auto]{\(\scriptstyle{\phi_i}\)} (b-2-3) edge (b-1-4)
			(b-2-1) edge[-,double] (b-2-2)
			(b-2-2) edge[-,double] (b-2-3)
			(b-2-3) edge[-,double] (b-2-4)
			(b-1-4) edge[-,double] (b-2-4);
		\end{tikzpicture}.
	\end{center}
	Let \(f_j:X_i\to Y_j \in \mathscr{J}_j\) for each \(2\leq j\leq n\). Denote \(\square^\lrcorner(\bullet,f_2,\dots, f_n)\) by \(\square(\bullet,\mathbf{f})\) and the domain by by \(Q_{\bullet,\mathbf{f}}\).  Then since \(\square^\lrcorner\) preserves connected colimits argument by argument, it follows that 
	\[
		\colim\square^\lrcorner(D^2,\mathbf{f})=\square^\lrcorner(\id_C,\mathbf{f}),
	\]
	and so we have the structure maps of the colimiting cocone
	\[
		(\phi_i^\lrcorner,\id_C): \square^\lrcorner(\phi_i,\mathbf{f}) \to \square^\lrcorner(\id_C,\mathbf{f}).
	\]
	In particular, this demonstrates that the map
	\[
		(\phi_0^\lrcorner,\id_C):\square^\lrcorner(\phi_0,\mathbf{f}) \to \square^\lrcorner(\id_C,\mathbf{f})
	\]
	is the transfinite composite of the maps 
	\[
		(g_i^\lrcorner,\id_C):\square^\lrcorner(\phi_i,\mathbf{f}) \to \square^\lrcorner(\phi_{i+1},\mathbf{f})
	\]
	But by commutativity of the square
	\begin{center}
		\begin{tikzpicture}
			\matrix (b) [matrix of math nodes, row sep=2em, column sep=3em]
			{
				Q_{\phi_0,\mathbf{f}} & Q_{\id_C,\mathbf{f}} \\
				\square(C,\mathbf{Y}) & \square(C,\mathbf{Y}) \\
			};
			\path[->]
			(b-1-1) edge node[auto]{\(\scriptstyle{\phi_0^\lrcorner}\)} (b-1-2) edge node[auto]{\(\scriptstyle{\square^\lrcorner(\phi_0,\mathbf{f}})\)} (b-2-1)
			(b-1-2) edge node[auto]{\(\scriptstyle{\square^\lrcorner(\id_C,\mathbf{f})}\)} (b-2-2)
			(b-2-1) edge node[auto]{\(\scriptstyle{\id}\)} (b-2-2);
		\end{tikzpicture},
	\end{center}
	we see that 
	\[
		\square^\lrcorner(\phi_0,\mathbf{f}) = \square^\lrcorner(\id_C,\mathbf{f}) \circ \phi^\lrcorner_0,
	\] 
	but from \Cref{corneridentities}, we see \(\square^\lrcorner(\id_C,\mathbf{f})=\id\), so it immediately follows that \(\square^\lrcorner(\phi_0,\mathbf{f}) = \phi^\lrcorner_0\). This exhibits \(\square^\lrcorner(\phi_0,\mathbf{f})\) as a transfinite composite of the \(g^\lrcorner_i\).  
	Then it suffices to show that the \(g_i^\lrcorner\) are pushouts of maps belonging to 
	\[\square^\lrcorner(\mathscr{J}_1,,\dots ,\mathscr{J}_n).\]
	Since each \(g_i\) was a pushout of a morphism \(h_i \in \mathscr{J}_1\), and since we have a commutative square \((g_i,\id):\phi_i\to \phi_{i+1}\), we are exactly in the situation of \Cref{cornertwist}, which implies that each \(g_i^\lrcorner\) is a pushout of \(\square^\lrcorner(h_i,\mathbf{h})\), which proves the proposition.
\end{proof}
\begin{cor} 
	There is an inclusion
	\[\square^\lrcorner(\operatorname{Cell}(\mathscr{J}_1),\dots,\operatorname{Cell}(\mathscr{J}_n)) \subseteq \operatorname{Cell}(\square^\lrcorner(\mathscr{J}_1,\dots,\mathscr{J}_n).\]
\end{cor}
\begin{proof} 
	Apply the previous lemma \(n\) times, using the fact that \(\operatorname{Cell}\) is idempotent, since it is a closure operator.
\end{proof}

\begin{defn} Let \(\mathrm{Rex}_c\) denote the symmetric sub-multi-category of \(\Cat\) whose objects are the categories admitting all finite colimits and whose \(k\)-morphisms are the \(k\)-fold functors that preserve finite connected colimits.
\end{defn}
\begin{obs}\label{cornertensorfunctoriality} The corner tensor construction, sending a multimorphism 
	\[F:\mathcal{A}_1\times\dots\times \mathcal{A}_n \to \mathcal{D}\] to its corner tensor
	\[F^\lrcorner:\mathcal{A}_1^{[1]}\times\dots\times \mathcal{A}_n^{[1]} \to \mathcal{D}^{[1]}\] is a morphism of multicategories from \(\mathrm{Rex}_c\) to itself.  In particular, it is functorial with respect to the composition in \(\mathrm{Rex}_c\).
\end{obs}
\subsection{The regular Reedy structure of \(\Theta[\C]\)}\label{reedy}
In this section, we will prove some useful and interesting properties about regular Cartesian Reedy categories (see \Cref{cishelp}). In particular, we demonstrate that the class of monomorphisms is exactly the class of relative cell complexes for a set of maps \(\mathscr{M}\), which we will show coincides with the set of boundary maps for the regular skeletal Reedy category \(\Theta[\C]\).  

\begin{prop} 
	The category \(\Theta[\C]\) is a regular skeletal Reedy category whenever \(\C\) is a regular Cartesian Reedy category.  The dimension function of this regular Reedy category is given by 
	\[\dim[n](c_1,\dots,c_n)\defeq n+\dim_\C c_1 + \dots + \dim_\C c_n.\]
\end{prop}
\begin{proof} 
	It follows by \cite{bergner-rezk-reedy}*{Proposition 4.4} that \(\Theta[\C]\) is normal skeletal Reedy with the desired dimension function (see \Cref{normskelcat}). To prove that \(\Theta[\C]\) is regular, it suffices to show that any nondegenerate section 
	\[(\alpha,\mathbf{f}):[n](c_1,\dots,c_n) \to [m](d_1,\dots, d_m)\]
	is monic.  If \(n>m\), then the map \(\alpha\) is a simplicial degeneracy and therefore the map factors through \([m](c_{i_1},\dots,c_{i_m})\), which means that \((\alpha,\mathbf{f}))\) cannot be nondegenerate.  
	
	Then we have two cases, when \(n=m\) or \(n<m\), and therefore either \(\alpha\) is the identity or a simplicial face map.  If \(n=m\), then clearly \(\alpha, \mathbf{f}\) is monic, since each nondegenerate section \(c_i\to d_i\) must be monic by the regularity of \(\C\).
	
	If \(\alpha\) is a composite of outer face maps, then \((\alpha,\mathbf{f})\) lands in \([n](d_1,\dots,d_n)\) or \([n](d_{m-n},\dots, d_m)\), and then from the previous case together with the fact that those inclusions are monic.  Otherwise, by induction on dimension, we can assume that \(\alpha\) is the inclusion of a codimension \(1\) inner face map obtained by removing the \(k^\mathrm{th}\) vertex. Then \((\alpha,\mathbf{f})\) factors through the inclusion \[V[n](d_1,\dots,d_{k-1} \times d_k,\dots, d_m) \subseteq [m](d_1,\dots,d_m)\] by a map \(\operatorname{id}, \mathbf{g}\) where each \(g_i\) is a nondegenerate section.  Then each \(g_i\) must be monic since for \(i\neq k\), \(g_i\) is a nondegenerate section of a representable, which is monic, and for \(i=k\), \(g_i\) is a nondegenerate section \(c\to d_{k-1} \times d_k\), which is monic by the fact that \(\C\) is regular Cartesian.
\end{proof}
\begin{prop} The boundary \(\partial[n](c_1,\dots,c_n)\) can be computed using the corner-intertwiner (see \ref{horizontal})
	\[Q=\square^\lrcorner_n(\delta^n, \delta^{c_1}, \dots, \delta^{c_n})(0).\]
\end{prop}
\begin{proof}
	It is clear that \(Q\subseteq \partial ([n](c_1,\dots,c_n))\), as it is a union of representable subobjects.  Suppose conversely that \([m](d_1,\dots,d_m) \to [n](c_1,\dots,c_n)\) is a nondegenerate section with \(\dim [m](d_1,\dots,d_m) < \dim [n](c_1,\dots,c_n)\).  As it is nondegenerate, we see immediately that \(m\leq n\).  
	
	Suppose \(m<n\).  Then the map \([m]\to [n]\) factors through \(V_{\partial \Delta^n}(c_1,\dots,c_n)\), and since the inclusion map 
	\[
		V_{\partial \Delta^n}(c_1,\dots,c_n) \hookrightarrow [n](c_1,\dots,c_n)
	\]
	is monic, it follows that the map 
	\[
		[m](d_1,\dots,d_m)\to V_{\partial \Delta^n}(c_1,\dots,c_n)
	\]
	must also be monic, and ergo that \([m](d_1,\dots,d_m)\subseteq Q\).
	
	Otherwise, suppose \(m=n\).  By the strictness of the inequality \(\dim [m](d_1,\dots,d_m) < \dim [n](c_1,\dots,c_n)\), we see that there exists some \(k\) with \(1\leq k \leq n\) such that \(\dim d_k < \dim c_k\), as otherwise the dimensions would be equal, since \(\dim d_i \leq dim c_i\) for all \(1\leq i \leq n\) by nondegeneracy.  Then it follows immediately that \([m](d_1,\dots,d_m) \subset V[n](c_1,\dots,\partial c_k, \dots, c_n)\), which proves the proposition.
\end{proof}
\begin{cor} 
	We define the set
	\[\mathscr{M}=\{\partial t \to t \mid t \in \Ob(\Theta[\C])\}.\]  Then the class \(\operatorname{Cell}(\mathscr{M})\) is exactly the the class of monomorphisms of \(\cellset\).  
\end{cor}
\begin{proof}
	This is an immediate consequence of \cite{cisinski-book}*{Proposition 8.1.37} or \cite{bergner-rezk-reedy}*{4.4}.
\end{proof}
\begin{prop}
	The category \(\Theta[\C]\) is regular Cartesian Reedy when \(\C\) is.
\end{prop}
\begin{proof}
	We treat the case of binary products of representables.  The case of general finite products is similar, albeit more notation-heavy. We leave the details to the reader.

	By our calculations in \Cref{sec:intertwiner}, since all involved objects are sober, we see that for a section 
	\[
		(\alpha,\beta):[n](c^n_1,\dots,c^n_n) \to [m_1](c^{m_1}_1,\dots,c^{m_1}_{m_1})\times [m_2](c^{m_2}_1,\dots,c^{m_2}_{m_2})
	\]
	to be nondegenerate, the associated map of simplicial sets \([n]\to [m_1]\times [m_2]\) is monic. The map \((\alpha,\beta)\) therefore factors through the pullback of the labeling on \([m_1](c^{m_1}_1,\dots,c^{m_1}_{m_1})\times [m_2](c^{m_2}_1,\dots,c^{m_2}_{m_2})\) to \([n]\), denoted by \(([n],(\alpha,\beta)^\ast \Omega)\).  The map  
	\[
		([n],(\alpha,\beta)^\ast \Omega) \to [m_1](c^{m_1}_1,\dots,c^{m_1}_{m_1})\times [m_1](c^{m_2}_1,\dots,c^{m_2}_{m_2})
	\]
	is monic, so it will be enough to show that the map
	\[
		\iota:[n](c^n_1,\dots,c^n_n)\to ([n],(\alpha,\beta)^\ast \Omega)
	\]
	must also be monic.  The section \(\iota\) is certainly nondegenerate, since if it were not, the map \((\alpha,\beta)\) could not be either.
	
	We will directly calculate \((\alpha,\beta)^\ast \Omega\) in \ref{square2}, and it will be shown that the labeling, determined by its restriction to each edge \((e_i)_{i=1}^n\) of the spine of \([n]\) is given by the formula
	\[
		e_i^\ast (\alpha,\beta)^\ast \Omega = \prod_{k_\alpha=\alpha(i-1)+1}^{\alpha(i)} c^{m_1}_{k_\alpha} \times \prod_{k_\beta=\beta(i-1)+1}^{\beta(i)} c^{m_2}_{k_\beta}.
	\] 
	Since the section \(\iota\) is nondegenerate and the associated map on simplicial sets is an isomorphism,  we see that each of the sections 
	\[
		\iota_i:c^n_i \to \prod_{k_\alpha=\alpha(i-1)+1}^{\alpha(i)} c^{m_1}_{k_\alpha} \times \prod_{k_\beta=\beta(i-1)+1}^{\beta(i)} c^{m_2}_{k_\beta}
	\] 
	must be nondegenerate as well, and therefore since \(\C\) is regular Cartesian Reedy, it follows that each of the maps \(\iota_i\) is monic, which proves the first axiom, and that 
	\[
		\dim c^n_i \leq \sum_{k_\alpha=\alpha(i-1)+1}^{\alpha(i)} \dim c^{m_1}_{k_\alpha} + \sum_{k_\beta=\beta(i-1)+1}^{\beta(i)} \dim c^{m_2}_{k_\beta}.
	\]
	But it follows from this that 
	\begin{align*}
		\dim [n](c^n_1, \dots c^n_n) &= n + \sum_{i=1}^n \dim c^n_i\\
		&\leq n+ \sum_{i=1}^n \left(\sum_{k_\alpha=\alpha(i-1)+1}^{\alpha(i)} \dim c^{m_1}_{k_\alpha} + \sum_{k_\beta=\beta(i-1)+1}^{\beta(i)} \dim c^{m_2}_{k_\beta}\right)\\
		&\leq m_1 + \sum_{i=1}^{m_1} \dim c^{m_1}_i + m_2 + \sum_{j=1}^{m_2} \dim c^{m_2}_j\\
		&=\dim [m_1](c^m_1,\dots,c^{m_1}_{m_1}) + \dim [m_2](c^m_1,\dots,c^{m_2}_{m_2}),
	\end{align*}
	which proves the second axiom.
\end{proof}
\begin{rem}
	The proposition above gives us a way to construct new regular Cartesian Reedy categories from known examples by applying the functor \(\Theta[-]\), which we view as a kind of free-generation-under-suspension functor. Moreover, regular Cartesian Reedy categories are stable under certain kinds of filtered unions (at least filtered unions of fully faithful maps preserving the dimension grading). The list of such categories generated from the terminal category under the operation \(\Theta[-]\) is exactly the family of \(\Theta_n\) for \(0\leq n < \omega,\) and stabilizng under good-enough filtered unions, we also obtain the category \(\Theta=\Theta_\omega\).   These are ultimately the only examples we care about in this paper.
	
	Though we do not need it here, it can be also be shown that finite products of regular Cartesian Reedy categories are also regular Cartesian Reedy. Starting with the terminal category and taking finite products, \(\Theta[-]\), and good-enough filtered unions, we obtain a large supply of regular Cartesian Reedy categories.  We leave it as an open question as to whether or not these operations generate all examples. 
\end{rem}
\subsection{The anodyne theorem for horizontal inner anodynes}\label{horizontal}
In this section, following \cite{oury}*{3.4.4}, we will demonstrate that the horizontal inner anodynes are closed under corner products with monomorphisms.  As a corollary of the analysis in this section, we will demonstrate that \(\Theta[\C]\) is regular Cartesian Reedy.  We make no claim to originality.

\begin{defn} Given a simplicial set \(S\) define the functor 
	\[H_S: \overcat{\psh{\Delta}}{S} \times  \left(\ssetlab\right)_S \to \ssetlab\]
	by the rule
	\[\left(S^\prime \xrightarrow{f} S, (S,\Omega)\right) \mapsto (S^\prime,  \Omega\circ f),\]
	and we define the \emph{relative intertwiner over \(S\)}
	\[\square_S \defeq \square \circ H_S.\]
	
	Notice that when \(S=\Delta^n\), the fibre decomposes as \[\left(\ssetlab\right)_n\simeq \psh{\C}^n.\] So we can write 
	\[\square_n:\overcat{\psh{\Delta}}{\Delta^n} \times \psh{\C}^n \to \cellset.\]
\end{defn}

\begin{obs}\label{evaluationsquare}
	Given a labeled simplicial set \((S,\Omega)\), a map of simplicial sets \(Y\to S\), and an object \([t]=[n](c_1,\dots,c_n)\) of \(\Theta[\C]\), we can compute \(\square_S(f,\Omega)_t\) as follows: For any \(n\)-simplex \(s\in S_n\), let \((W_{s,i})_{i=1^n}\) be the family of \(\C\)-sets obtained by evaluation of \(\Omega\) on \(s\). A map \([t] \to Y\square \Omega\circ f\) is by definition a map \([t]\to (Y,\Omega\circ f)\).  Such a map is determined by giving an \(n\)-simplex \(y\in Y_n\) together with a family of maps 
	\[(c_i \xrightarrow{\zeta_i} W_{fy,i})_{i=1}^n.\] Then we can compute 
	\[\square_S(f,\Omega)_t\cong \coprod_{y\in Y_n} \prod_{i=1}^n W_{fy,i,c_i}.\]
\end{obs}

\begin{defn}\label{productintertwiner}
	Given a finite family of simplicial sets \(\mathbf{S}=(S_i)_{i=1}^n\), we define a functor:
	\[H_{\mathbf{S}}: \overcat{\psh{\Delta}}{\prod_{i=1}^n S_i} \times  \prod_{i=1}^n\left(\ssetlab\right)_{S_i} \to \ssetlab\] by the rule:
	\[\left(S \xrightarrow{\prod^n_{i=1} f_i} \prod_{i=1}^n S_i, (\Omega_i)_{i=1}^n\right)\mapsto \left(S,\prod_{i=1}^n (\Omega_i\circ f_i)\right).\]
	As in the previous definition, we define the \emph{relative multi-intertwiner} by the formula
	\[\square_{\mathbf{S}} \defeq \square \circ H_\mathbf{S}.\]
\end{defn}

\begin{rem} 
	Notice that given a finite family of labeled simplicial sets \((\mathbf{S},\mathbf{\Omega})=(S_i,\Omega_i)_{i=1}^n\) and a family \(\mathbf{f}=\left(S\xrightarrow{f_i} S_i\right)^n_{i=1},\) 
	\[
		H_{\mathbf{S}}(\mathbf{f},\mathbf{\Omega})\cong H_{S_1}(f_1, \Omega_1) \times^S \dots \times^S H_{S_n}(f_n,\Omega_n),
	\]
	where \(\times^S\) denotes the product in \(\left(\ssetlab\right)_S\).
\end{rem}

\begin{obs}
	Let \[(\mathbf{S},\mathbf{\Omega})=\left(\Delta^{m_i},\Omega_i\right)_{i=1}^n\] be a family of labeled simplices, and let \[\mathbf{f}=\left(\Delta^r\xrightarrow{f_i} \Delta^{m_i}\right)^n_{i=1},\] be a family of maps defining an \(r\)-simplex of the product.
	We may identify the \(\Omega_i\) with families of \(\C\)-presheaves \((X_{i,\ell})^{m_i}_{\ell=1}\), so we compute \(H_{\Delta^{m_i}}(f_i,\Omega_i)\) as the labeled simplex
	\[[r]\left(\left(\prod_{k=f_i(j-1)+1}^{f_i(j)} X_{i,k} \right)_{j=1}^r\right),\]
	and therefore, we can compute 
	\(H_{\mathbf{S}}(\mathbf{f},\mathbf{\Omega})\) as the labeled simplex
	\[ [r]\left(\left(\prod_{i=1}^n \left(\prod_{k=f_i(j-1)+1}^{f_i(j)} X_{i,k} \right)\right)_{j=1}^r\right).\]
\end{obs}
\begin{lemma}
	The relative intertwiner \(\square_S\) preserves colimits in the first variable.
\end{lemma}
\begin{proof} 
	Since colimits are computed objectwise in presheaves, it suffices to show that the functor \(\square_S(\bullet,\Omega)_t\) preserves colimits for all \([t]=[n](c_1,\dots,c_n) \in \Theta[\C]\) and all labels \(\Omega\) of \(S\).  Therefore, it suffices by \ref{evaluationsquare} to show this in the case where \(\C\) is the terminal category, since we may fix the family of objects \((c_1,\dots,c_n)\). For each \(s\in S_n\) let \((W_{s,i})_{i=1}^n\) be the evaluation of \(\Omega\) on \(s\).  Then given \(f:Y\to S\), we have a Cartesian square
	\begin{center}
		\begin{tikzpicture}
			\matrix (b) [matrix of math nodes, row sep=3em, column sep=3em]
			{
				\coprod_{y\in Y_n} \prod_{i=1}^n W_{fy,i} & Y_n\\
				\coprod_{s\in S_n} \prod_{i=1}^n W_{s,i} & S_n \\
			};
			\path[->]
			(b-1-1) edge (b-1-2) edge node[auto,swap]{\(\scriptstyle{\tau^*f}\)} (b-2-1)
			(b-2-1) edge node[auto]{\(\scriptstyle{\tau}\)} (b-2-2)
			(b-1-2) edge node[auto]{\(\scriptstyle{f}\)} (b-2-2);
		\end{tikzpicture},
	\end{center}
	exhibiting \(\coprod_{y\in Y_n} \prod_{i=1}^n W_{fy,i}\) as the pullback of \(f\) along \(\tau\), but by the universality of colimits in the category of sets, we are done.  
\end{proof}

\begin{note} 
	This is Oury's proof, but this statement can also be seen to follow immediately from \ref{pullbacksober}.
\end{note}

\begin{lemma}\label{conncolims}
	The relative intertwiner \(\square_n\) preserves connected colimits in each variable.  
\end{lemma}
\begin{proof}
	We saw from the previous lemma that it preserves colimits in the first variable, so representing \(Y\) as the colimit of its simplices, we immediately reduce to the case where \(Y\) is a simplex.  But we know in this case that any map \([p]\to [n]\) factors as a degeneracy followed by a face map.  In the case that \(f\) is a face map, we can compute the pullback of \(V[n](X_1,\dots,X_n)\) along \(f\) to be 
	\[V[p]\left(\prod_{i=f(0)+1}^{f(1)}X_i,\dots, \prod_{i=f(p-1)}^{f(p)}X_i\right).\]  By universality of colimits in \(\psh{\C}\), we see that it suffices to show that \[V[p]\left(\bullet,\dots,\bullet)=\square_p(\id_{\Delta^p},\bullet,\dots,\bullet\right)\] preserves connected colimits in each variable.  In the case where \(f\) is a degeneracy map, we can compute the pullback over \([p]\) to be \[V[p](\ast, \dots, X_1, \dots, \ast, \dots, X_n, \dots \ast),\] where we fill in the terminal object of \(\psh{\C}\) in each argument \(i\) where \(f(i-1)=f(i)\).  In this case again, it again suffices to show that \(V[p]\) preserves connected colimits in each variable, but this is precisely the content of \cite{rezk-theta-n-spaces}*{Proposition 4.5}, where the proof proceeds by first showing that if we set \(X_k=\varnothing\), then
	\[V[p+1+q](A_1,\dots, A_p, \varnothing, B_1,\dots ,B_q) \cong V[p](A_1,\dots,A_p) \coprod V[q](B_1,\dots,B_q),\]
	and then exhibiting the obvious parametric right adjoint
	\[\overcat{\left(V[p](A_1,\dots,A_p) \coprod V[q](B_1,\dots,B_q)\right)}{\psh{\Theta[\C]}}\to \psh{\C}.\]
\end{proof}
\begin{note} 
	This proof is substantially easier than Oury's proof, which relies on a long direct computation.
\end{note}

\begin{defn}
	Since the categories \(\psh{\C}\), \(\cellset\), \(\psh{\C}^n\), and \(\overcat{\psh{\Delta}}{\Delta^n}\) are all cocomplete (since they are all presheaf categories), and since the intertwiner preserves connected colimits argument-by-argument, we can use \ref{cornertensor} to define the functor
	\[\square^\lrcorner_n:\overcat{\psh{\Delta}}{\Delta^n}^{[1]} \times \underbrace{\psh{\C}^{[1]}\times \dots \times \psh{\C}^{[1]}}_{n\text{ times}} \to \cellset^{[1]},\] called the \emph{corner intertwiner}.

	More generally, for any finite family of simplices \(\left(\Delta^{m_i}\right)^n_{i=1}\), we can do the same trick and define the \emph{corner-multi-intertwiner}
	\[\square^\lrcorner_{m_1,\dots,m_n}: \overcat{\psh{\Delta}}{\Delta^{m_1}}^{[1]} \times \dots \times \overcat{\psh{\Delta}}{\Delta^{m_n}}^{[1]} \times \left(\psh{\C}^{[1]}\right)^{m_1} \times \dots \times \left(\psh{\C}^{[1]}\right)^{m_n} \to \cellset^{[1]}.\]
\end{defn}

Following \cite{oury}*{3.85}, we begin with the following observations:

\begin{obs}\label{square1}
	We saw from the definition of \(\square\) and the definition of products in \(\ssetlab\) that the diagram 
	\begin{center}
		\begin{tikzpicture}
			\matrix (b) [matrix of math nodes, row sep=3em, column sep=3em]
			{
				\ssetlab \times \ssetlab & \cellset\times \cellset\\
				\ssetlab 				 & \cellset               \\
			};
			\path[->]
			(b-1-1) edge node[auto]{\(\scriptstyle{\square \times \square}\)} (b-1-2) edge node[auto]{\(\scriptstyle{\times}\)} (b-2-1)
			(b-2-1) edge node[auto]{\(\scriptstyle{\square}\)} (b-2-2)
			(b-1-2) edge node[auto]{\(\scriptstyle{\times}\)} (b-2-2);
		\end{tikzpicture}
	\end{center}
	commutes.  We also computed that the diagram
	\begin{center}
		\begin{tikzpicture}
			\matrix (b) [matrix of math nodes, row sep=3em, column sep=4em]
			{
				\overcat{\psh{\Delta}}{\Delta^n} \times \psh{\C}^n \times \overcat{\psh{\Delta}}{\Delta^m} \times \psh{\C}^m & \ssetlab\times \ssetlab\\
				\overcat{\psh{\Delta}}{\Delta^n} \times \overcat{\psh{\Delta}}{\Delta^m} \times \psh{\C}^n \times \psh{\C}^m &\\
				\overcat{\psh{\Delta}}{\Delta^n\times \Delta^m} \times \psh{\C}^n \times \psh{\C}^m & \ssetlab\\
			};
			\path[->]
			(b-1-1) edge node[auto]{\(\scriptstyle{H_n \times H_m}\)} (b-1-2) edge node[auto]{\(\scriptstyle{\varsigma}\)} (b-2-1)
			(b-2-1) edge node[auto]{\(\scriptstyle{P\times \id \times \id}\)} (b-3-1)
			(b-1-2) edge node[auto]{\(\scriptstyle{\times}\)} (b-3-2)
			(b-3-1) edge node[auto]{\(\scriptstyle{H_{n,m}}\)} (b-3-2);
		\end{tikzpicture}
	\end{center}
	commutes as well where \(\varsigma\) permutes the factors and \(P\) is the functor sending a pair of simplicial sets \(f:S\to \Delta^n\) and \(g:S^\prime\to \Delta^m\) over \(\Delta^n\) and \(\Delta^m\) respectively to the simplicial set \[f\times g:S\times S^\prime \to \Delta^n\times \Delta^m\] over the product \(\Delta^n \times \Delta^m\).  Taking these two diagrams together, we see that the diagram 
	\begin{center}
		\begin{tikzpicture}
			\matrix (b) [matrix of math nodes, row sep=3em, column sep=4em]
			{
				\overcat{\psh{\Delta}}{\Delta^n} \times \psh{\C}^n \times \overcat{\psh{\Delta}}{\Delta^m} \times \psh{\C}^m & \cellset\times\cellset\\
				\overcat{\psh{\Delta}}{\Delta^n} \times \overcat{\psh{\Delta}}{\Delta^m} \times \psh{\C}^n \times \psh{\C}^m &\\
				\overcat{\psh{\Delta}}{\Delta^n\times \Delta^m} \times \psh{\C}^n \times \psh{\C}^m & \cellset\\
			};
			\path[->]
			(b-1-1) edge node[auto]{\(\scriptstyle{\square_n \times \square_m}\)} (b-1-2) edge node[auto]{\(\scriptstyle{\varsigma}\)} (b-2-1)
			(b-2-1) edge node[auto]{\(\scriptstyle{P\times \id \times \id}\)} (b-3-1)
			(b-1-2) edge node[auto]{\(\scriptstyle{\times}\)} (b-3-2)
			(b-3-1) edge node[auto]{\(\scriptstyle{\square_{n,m}}\)} (b-3-2);
		\end{tikzpicture}
	\end{center}
	also commutes.

Then by \ref{conncolims}, every functor appearing in this diagram preserves connected colimits in each argument, the intertwiners by the lemma, and the functors \(P\) and \(\times\), since they are products in presheaf categories and therefore preserve colimits in both arguments. Then by the functoriality of the corner tensor functor \ref{cornertensorfunctoriality}, we obtain a commutative diagram 
\begin{center}
	\begin{tikzpicture}
		\matrix (b) [matrix of math nodes, row sep=3em, column sep=4em]
		{
			\overcat{\psh{\Delta}}{\Delta^n}^{[1]} \times \left(\psh{\C}^{[1]}\right)^n \times \overcat{\psh{\Delta}}{\Delta^m}^{[1]} \times \left(\psh{\C}^{[1]}\right)^m & \cellset^{[1]}\times\cellset^{[1]}\\
			\overcat{\psh{\Delta}}{\Delta^n}^{[1]} \times \overcat{\psh{\Delta}}{\Delta^m}^{[1]} \times \left(\psh{\C}^{[1]}\right)^n \times \left(\psh{\C}^{[1]}\right)^m &\\
			\overcat{\psh{\Delta}}{\Delta^n\times \Delta^m}^{[1]} \times \left(\psh{\C}^{[1]}\right)^n \times \left(\psh{\C}^{[1]}\right)^m  & \cellset^{[1]}\\
		};
		\path[->]
		(b-1-1) edge node[auto]{\(\scriptstyle{\square^\lrcorner_n \times \square^\lrcorner_m}\)} (b-1-2) edge node[auto]{\(\scriptstyle{\varsigma}\)} (b-2-1)
		(b-2-1) edge node[auto]{\(\scriptstyle{P^\lrcorner \times \id \times \id}\)} (b-3-1)
		(b-1-2) edge node[auto]{\(\scriptstyle{\times^\lrcorner}\)} (b-3-2)
		(b-3-1) edge node[auto]{\(\scriptstyle{\square_{n,m}^\lrcorner}\)} (b-3-2);
	\end{tikzpicture}
\end{center}
also commutes, where \(P^\lrcorner=\times^\lrcorner\) is the corner product of simplicial sets.
\end{obs}

\begin{obs}\label{simplicialhorns}
	Consider the corner product of a simplicial inner horn inclusion with a simplicial boundary inclusion  
	\[\lambda^n_j\times^\lrcorner \delta^m: \Lambda^n_j \times \Delta^m \cup \Delta^n \times \partial \Delta^m \hookrightarrow \Delta^n\times \Delta^m.\]
	Then it is a standard fact of quasicategory theory that we can factor this map as a sequence
	\[\Lambda^n_j \times \Delta^m \cup \Delta^n \times \partial \Delta^m = X_0 \subseteq X_1 \subseteq \dots \to X_{k-1} \subseteq X_k=\Delta^n\times \Delta^m\]
	where each inclusion \(X_{i-1}\hookrightarrow X_i\) is the pushout of an inner horn inclusion \(\Lambda^{r_i}_{\ell_i} \to \Delta^{r_i}\) along an inclusion \(\Lambda^{r_i}_{\ell_i}\hookrightarrow X_{i-1}\).  By the construction of the sequence, each \([r_i]\to X_i\to \Delta^n\times \Delta^m\) is nondegenerate and does not factor through \(X_{i-1}\), so in particular, it does not factor through \(X_0\), and therefore the maps \(\alpha_i:\Delta^{r_i} \to \Delta^n\) and \(\beta_i:[r_i]\to \Delta^m\) do not factor through \(\Lambda^n_j\) or \(\partial \Delta^m\).  In particular, the image of \(\alpha_i\) is either \(\partial_j\Delta^n\) or all of \(\Delta^n\), and the image of \(\beta_i\) must be all of \(\Delta^m\), so all three maps \(\alpha_i\), \(\beta_i\), and \(\alpha_i \times \beta_i\) send the initial and terminal vertices of \(\Delta^{r_i}\) to the initial and terminal vertices of \(\Delta^n\), \(\Delta^m\), and \(\Delta^n\times\Delta^m\) respectively.
\end{obs}

\begin{obs}\label{square2}
	Let \((\alpha, \beta): \Delta^r \to \Delta^n\times \Delta^m\) be an injective map preserving initial and terminal elements.  Let \(\mathbf{A}=(A_i)^n_{i=1}\) and \(\mathbf{B}=(B_i)^m_{i=1}\) be objects of \(\psh{\C}^n\) and \(\psh{\C}^m\) respectively.  Let 
	\[K_{\alpha,\beta}:\psh{\C}^n \times \psh{\C}^m\to \psh{\C}^r\] be the functor defined by the rule
	\[(\mathbf{U},\mathbf{V})\mapsto \alpha^\ast \mathbf{U} \times \beta^\ast \mathbf{V},\]
	taking the product of the pullbacks to the fibre over \(\Delta^r\).  Then we have a diagram:
	\begin{center}
		\begin{tikzpicture}
			\matrix (b) [matrix of math nodes, row sep=3em, column sep=4em]
			{
				\overcat{\psh{\Delta}}{\Delta^r} \times \psh{\C}^n \times \psh{\C}^m & \overcat{\psh{\Delta}}{\Delta^r} \times \psh{\C}^r\\
				\overcat{\psh{\Delta}}{\Delta^n\times \Delta^m} \times \psh{\C}^n \times \psh{\C}^m & \ssetlab \\
			};
			\path[->]
			(b-1-1) edge node[auto]{\(\scriptstyle{\id \times K_{\alpha,\beta}}\)} (b-1-2) 
					edge node[auto]{\(\scriptstyle{((\alpha,\beta))\circ(-) \times \id \times \id}\)} (b-2-1)
			(b-2-1) edge node[auto]{\(\scriptstyle{H_{n,m}}\)} (b-2-2)
			(b-1-2) edge node[auto]{\(\scriptstyle{H_r}\)} (b-2-2);
		\end{tikzpicture}
	\end{center}
	To show that the diagram commutes, let \(p:X\to \Delta^r\) be a map.  Then evaluating \(H_r(p,K_{\alpha,\beta}(\mathbf{U},\mathbf{V})=H_r(p,\alpha^*\mathbf{U} \times \beta^*\mathbf{V})\) on a simplex \(x:\Delta^s\to X\) is 
	\begin{align*}
		(px)^\ast(\alpha^\ast\mathbf{U} \times \beta^\ast\mathbf{V}) &= (px)^\ast\alpha^\ast\mathbf{U} \times (px)^\ast\beta^\ast\mathbf{U}\\
		&=(\alpha px)^\ast\mathbf{U} \times (\beta px)^\ast\mathbf{V}\\
		&=H_{n,m}((\alpha px, \beta px),\mathbf{U},\mathbf{V})\\
		&=H_{n,m}\circ ((\alpha,\beta)\circ \times \id \times \id) (px,\mathbf{U},\mathbf{V}), 
	\end{align*}
	which demonstrates that the diagram commutes.
	Let \((t_i)_{i=1}^r\) such that \(t_i=\alpha(i)-\alpha(i-1) +\beta(i) - \beta(i-1)\).  Note that the sum of the \(t_i\) is exactly \(n+m\), since \(\alpha\) and \(\beta\) preserve initial and terminal objects.  We define a functor 
	\[\tau_i:\psh{\C}^n \times \psh{\C}^m \to \psh{\C}^{t_i}\] by the rule
	\[(\mathbf{A},\mathbf{B})\mapsto (A_{\alpha(i-1)+1}, \dots, A_{\alpha(i)}, B_{\beta(i-1)+1},\dots, B_{\beta(i)}).\]
	Then define \[\tau:\psh{\C}^n \times \psh{\C}^m \to \prod_{i=1}^r \psh{\C}^{t_i}.\]
	It is clear that \(\tau\) is a permutation of variables and therefore an isomorphism.
	Then let \[P_i: \psh{\C}^{t_i} \to \psh{\C}\] be the functor defined by the rule
	\[(X_1,\dots,X_{t_i}) \mapsto X_1 \times \dots \times X_{t_i}\]
	Then the \(P_i\) assemble to a map \((P_1, \dots, P_r)\) such that \[(P_1,\dots,P_r)\circ \tau = K_{\alpha,\beta}.\]
	Then the diagram
	\begin{center}
		\begin{tikzpicture}
			\matrix (b) [matrix of math nodes, row sep=3em, column sep=4em]
			{
				\overcat{\psh{\Delta}}{\Delta^r} \times \psh{\C}^n \times \psh{\C}^m & \overcat{\psh{\Delta}}{\Delta^r} \times \prod_{i=1}^r \psh{\C}^{t_i} & \overcat{\psh{\Delta}}{\Delta^r} \times \psh{\C}^r\\
				\overcat{\psh{\Delta}}{\Delta^n\times \Delta^m} \times \psh{\C}^n \times \psh{\C}^m & & \ssetlab \\
			};
			\path[->]
			(b-1-1) edge node[auto]{\(\scriptstyle{\id \times \tau}\)} (b-1-2) 
					edge node[auto]{\(\scriptstyle{((\alpha,\beta))\circ(-) \times \id \times \id}\)} (b-2-1)
			(b-2-1) edge node[auto]{\(\scriptstyle{H_{n,m}}\)} (b-2-3)
			(b-1-2) edge node[auto]{\(\scriptstyle{\id \times (P_i)_{i=1}^r}\)} (b-1-3) 
			(b-1-3)	edge node[auto]{\(\scriptstyle{H_r}\)} (b-2-3);
		\end{tikzpicture}
	\end{center}
	commutes, and therefore, composing the bottom horizontal and right vertical maps with \(\square\), we have another commutative diagram
	\begin{center}
		\begin{tikzpicture}
			\matrix (b) [matrix of math nodes, row sep=3em, column sep=4em]
			{
				\overcat{\psh{\Delta}}{\Delta^r}\times \psh{\C}^n \times \psh{\C}^m & \overcat{\psh{\Delta}}{\Delta^r} \times \prod_{i=1}^r \psh{\C}^{t_i} & \overcat{\psh{\Delta}}{\Delta^r} \times \psh{\C}^r\\
				\overcat{\psh{\Delta}}{\Delta^n\times \Delta^m} \times \psh{\C}^n \times \psh{\C}^m & & \cellset \\
			};
			\path[->]
			(b-1-1) edge node[auto]{\(\scriptstyle{\id \times \tau}\)} (b-1-2) 
					edge node[auto]{\(\scriptstyle{((\alpha,\beta))\circ(-) \times \id \times \id}\)} (b-2-1)
			(b-2-1) edge node[auto]{\(\scriptstyle{\square_{n,m}}\)} (b-2-3)
			(b-1-2) edge node[auto]{\(\scriptstyle{\id \times (P_i)_{i=1}^r}\)} (b-1-3) 
			(b-1-3)	edge node[auto]{\(\scriptstyle{\square_r}\)} (b-2-3);
		\end{tikzpicture}.
	\end{center}
	The bottom horizontal and lower right vertical maps preserve connected colimits, as we have seen.  The left vertical map preserves connected colimits because colimits are computed in the domain for comma categories.  The map \(\prod_{i=1}^r P_i\) preserves colimits in each argument because colimits are universal in toposes. 
	Then applying the corner tensor functor, we have the commutative diagram
	\begin{center}
		\begin{tikzpicture}
			\matrix (b) [matrix of math nodes, row sep=3em, column sep=2em]
			{
				\overcat{\psh{\Delta}}{\Delta^r}^{[1]} \times \left(\psh{\C}^{[1]}\right)^n \times \left(\psh{\C}^{[1]}\right)^m & \overcat{\psh{\Delta}}{\Delta^r}^{[1]} \times \prod_{i=1}^r \left(\psh{\C}^{[1]}\right)^{t_i}\\
				&\overcat{\psh{\Delta}}{\Delta^r} \times \left(\psh{\C}^{[1]}\right)^r\\
				\overcat{\psh{\Delta}}{\Delta^n\times \Delta^m}^{[1]} \times \left(\psh{\C}^{[1]}\right)^n\times \left(\psh{\C}^{[1]}\right)^m & \cellset^{[1]} \\
			};
			\path[->]
			(b-1-1) edge node[auto]{\(\scriptstyle{\id \times \tau}\)} (b-1-2) 
					edge node[auto]{\(\scriptstyle{((\alpha,\beta))\circ(-)^\lrcorner \times \id \times \id}\)} (b-3-1)
			(b-3-1) edge node[auto]{\(\scriptstyle{\square^\lrcorner_{n,m}}\)} (b-3-2)
			(b-1-2) edge node[auto]{\(\scriptstyle{\id \times (P_i^\lrcorner)_{i=1}^r}\)} (b-2-2) 
			(b-2-2)	edge node[auto]{\(\scriptstyle{\square^\lrcorner_r}\)} (b-3-2);
		\end{tikzpicture}.
	\end{center}
\end{obs}
\begin{lemma}\label{anodynelemma}
	Let \((\alpha,\beta): \Delta^r \to \Delta^n\times \Delta^m\) be a nondegenerate section with \(r\geq 2\) and such that \(\alpha\) and \(\beta\) preserve initial and terminal vertices. Let \(\mathbf{f}=\{f_i:\partial c_i \hookrightarrow c_i\}_{i=1}^n\) and \(\mathbf{g}=\{g_i:\partial d_i \hookrightarrow d_i\}_{i=1}^m\) be families of boundary inclusions for \(\C\).  Then for any inner horn inclusion \(\lambda^r_k: \Lambda^r_k\hookrightarrow \Delta^r\), the map
	\[
		\square^\lrcorner_{n,m}(\lambda^r_k, \mathbf{f},\mathbf{g})
	\]
	is horizontal inner anodyne.
\end{lemma}
\begin{proof}
	By \ref{square2}, we see that 
	\[
		\square^\lrcorner_{n,m}(\lambda^r_k,\mathbf{f},\mathbf{g}) \cong \square^\lrcorner_r(\lambda^r_k, (P_1^\lrcorner,\dots, P_r^\lrcorner) \circ \tau(\mathbf{f},\mathbf{g})),
	\]
	But the value of the argument in position \(1\leq j \leq r\) is 
	\[
		P_j^\lrcorner\circ \tau_j(\mathbf{f},\mathbf{g})=f_{\alpha(j-1)+1}\times^\lrcorner \dots \times^\lrcorner f_{\alpha(j)} \times^\lrcorner g_{\beta(j-1)+1}\times \dots^\lrcorner \times^\lrcorner g_{\beta(j)},
	\]
	which belongs to the class \(\operatorname{Cell}(\mathscr{B})\). That is, the map 
	\[
		\square^\lrcorner_r(\lambda^r_k, (P_1^\lrcorner,\dots, P_r^\lrcorner) \circ \tau(\mathbf{f},\mathbf{g}))
	\]
	belongs to 
	\[
		\square^\lrcorner_{r}(\lambda^r_k, \operatorname{Cell}(\mathscr{B}), \dots, \operatorname{Cell}(\mathscr{B})),
	\]
	where \(\mathscr{B}\) denotes the set of boundary inclusions for \(\C\).  By \Cref{cornertensorcell}, it follows therefore that this map belongs to 
	\[
		\operatorname{Cell}(\square^\lrcorner_{r}(\lambda^r_k, \mathscr{B}, \dots, \mathscr{B})).
	\]
	But the set of maps
	\[
		\square^\lrcorner_{r}(\lambda^r_k, \mathscr{B}, \dots, \mathscr{B})
	\]
	is a subset of the generating horizontal inner anodynes, and therefore the map
	\[
		\square^\lrcorner_{n,m}(\lambda^r_k,\mathbf{f},\mathbf{g})\cong \square^\lrcorner_r(\lambda^r_k, (P_1^\lrcorner,\dots, P_r^\lrcorner) \circ \tau(\mathbf{f},\mathbf{g}))
	\]
	is horizontal inner anodyne.
\end{proof}
Finally, we reach our destination.
\begin{thm}[Anodyne Theorem \cite{oury}*{3.88}]\label{anodynethm}
	The class of horizontal anodynes is closed under corner products with monomorphisms.  In particular, if we let \[\mathscr{J}=\{\square_n^\lrcorner(\lambda^n_k, \delta^{c_1},\dots,\delta^{c_n})| \text{ for } n\geq 2, 0<k<n\}.\] Then we have
	\[\mathscr{M}\times^\lrcorner \mathscr{J} \subseteq \operatorname{Cell}(\mathscr{J}).\]
\end{thm}
\begin{proof}
	Let \(f_0:\partial\Delta^n \hookrightarrow \Delta^n\), and let \(\mathbf{f}=(f_i)_{i=1}^n\) be a family of boundary inclusions in \(\psh{\C}\).  Let \(g_0:\Lambda^m_k \hookrightarrow \Delta^n\) be an inner horn inclusion, and let \(\mathbf{g}=(g_i)_{i=1}^n\) be a family of boundary inclusions in \(\psh{\C}\).  By \ref{square1}, we have 
	\[
		\square^\lrcorner_n(f_0,\dots,f_n) \times^\lrcorner \square^\lrcorner_m(g_0,\dots,g_m) \cong \square^\lrcorner_{n,m}(f_0 \times^\lrcorner g_0, f_1,\dots,f_n,g_1,\dots,g_n).
	\]
	By \ref{simplicialhorns}, we know that \(f_0\times^\lrcorner g_0\) can be factored as a finite sequence of pushouts of inner horn inclusions.  By \ref{cornertensorcell}, it follows that
	\[
		\square^\lrcorner_{n,m}(f_0 \times^\lrcorner g_0, f_1,\dots,f_n,g_1,\dots,g_n)
	\]
	is a finite composite of pushouts of maps
	\[
		\square^\lrcorner_{n,m}(h_i, f_1,\dots,f_n,g_1,\dots,g_n)
	\] 
	where \((h_i:\Lambda^{r_i}_{\ell_i} \to \Delta^{r_i})_{i=1}^k\) are inner horn inclusions and the implicit maps 
	\[
		(\alpha_i,\beta_i):\Delta^{r_i}\to \Delta^n\times \Delta^m
	\]
	are initial and terminal vertex preserving. 

	But by the previous lemma, we see that each of these maps is horizontal inner anodyne, so we are done.
\end{proof}

\subsection{Comparison with Rezk's complete \(\Theta[\C]\)-spaces}\label{rezkcomparison}
Since Rezk's complete Segal model structure on \(\Psh_\Delta(\Theta[\C])\) is Cartesian, since \(\ast \hookrightarrow E\) is one of the generators of the localization (see \cite{rezk-theta-n-spaces}), and since \(\Theta[\C]\) is regular skeletal Reedy, it follows by several results of Cisinski \cite{cisinski-book}*{Proposition 8.2.9, Theorem 3.4.36, Proposition 2.3.30} that Rezk's localizer for complete \(\Theta[\C]\)-Segal spaces is the simplicial completion of a localizer on \(\Theta[\C]\).  
\begin{obs}
	To show that Rezk's localizer is the simplicial completion of the localizer generated by the horizontal inner anodynes, it suffices to show the following two properties hold:
	\begin{enumerate}[leftmargin=3em,label=(\roman*{})] 
		\item The maps 
		\(\square_n^\lrcorner(\lambda^n_i,\delta^{c_1},\dots,\delta^{c_n})\times \Delta^0\) belong to the localizer for complete \(\Theta[\C]\)-Segal spaces for \(n\geq 2\) and \(0<i<n\).
		\item The Segal maps \(\operatorname{Se}[n](c_1,\dots,c_n): \operatorname{Sp}[n](c_1,\dots,c_n) \hookrightarrow [n](c_1,\dots,c_n)\) are horizontal inner anodyne.
	\end{enumerate}
\end{obs}
We will make use of the following lemma:
\begin{lemma}\label{precatproperty}
	For any inner horn inclusion \(\Lambda^n_k\hookrightarrow \Delta^n,\) and any presheaves \(X_1,\dots,X_n\) on \(\C\), the map \(V_{\Lambda^n_k}(X_1,\dots,X_n)\times \Delta^0 \hookrightarrow V[n](X_1,\dots,X_n)\times \Delta^0\) belongs to the localizer for complete Segal-\(\Theta[\C]\)-spaces.
\end{lemma}
\begin{proof}
	We will suppress the \(\times \Delta^0\) factor denoting discrete simplicial presheaves for the duration of this proof.  By \cite{rezk-theta-n-spaces}*{5.2}, we know that the maps \(\operatorname{Se}[n](X_1,\dots,X^n)\) are already weak equivalences.  Then we proceed following the argument of \cite{jtsegal}*{Lemma 3.5}.  Notice that trivial cofibrations have the right-cancellation property with respect to monomorphisms.  Then we show that since the class of trivial cofibrations contains the class of maps \(\operatorname{Se}[n](X_1,\dots,X_n)\), it also contains the class of maps 
	\[V_{\partial_0 \Delta^n \cup \partial_n \Delta^n}(X_1,\dots,X_n)\hookrightarrow V[n](X_1,\dots,X_n)\] by induction on \(n\).  Notice first that the map
	\[V_{\Lambda^2_1}(X_1,X_2)\hookrightarrow V[n](X_1,X_2)\]
	is automatically a trivial cofibration, since \(\Lambda^2_1=\operatorname{Sp}[2]\).
	For the case of \(n>2\), notice that by cancellation, it suffices to show that the maps
	\[V_{\operatorname{Sp}[n]}(X_1,\dots, X_n) \xrightarrow{i_n} V_{\partial_0\Delta^n \cup \operatorname{Sp}[n]}(X_1,\dots, X_n) \xrightarrow{j_n} V_{\partial_0 \Delta^n \cup \partial_n \Delta^n}(X_1,\dots,X_n)\]
	are trivial cofibrations.
	Then notice that 
	\[V_{\operatorname{Sp}[n]}(X_1,\dots,X_n) \xrightarrow{i_n} V_{\partial_0 \Delta^n\cup \operatorname{Sp}[n]}(X_1,\dots,X_n)\]
	is a pushout of the map 
	\[V_{\operatorname{Sp}[n-1]}(X_1,\dots,X_n) \hookrightarrow V_{\partial_0 \Delta^n}(X_1,\dots,X_n),\]
	and is therefore a trivial cofibration.  
	Notice that for \(d_0: [n-1] \to [n]\), \(d_0^{-1}(\operatorname{Sp}[n]) = \operatorname{Sp}[n-1]\) and \(d_0^{-1}(\partial_n\Delta^n) = \partial_{n-1}\Delta^{n-1}\).  Then the square
	\begin{center}
		\begin{tikzpicture}
			\matrix (b) [matrix of math nodes, row sep=3em, column sep=3em]
			{
				V_{\operatorname{Sp}[n-1]\cup \partial_{n-1}\Delta^{n-1}}(X_1,\dots,X_n) & V_{\operatorname{Sp}[n-1]\cup \partial_n\Delta^n}(X_1,\dots,X_n)\\
				V_{\partial_0\Delta^n}(X_1,\dots,X_n)				 &  V_{\partial_0\Delta^n\cup \partial_n \Delta^n}(X_1,\dots,X_n)           \\
			};
			\path[->]
			(b-1-1) edge (b-1-2) edge node[auto]{\(\scriptstyle{k_{n-1}}\)} (b-2-1)
			(b-2-1) edge node[auto]{\(\scriptstyle{d^\prime_0}\)} (b-2-2)
			(b-1-2) edge node[auto]{\(\scriptstyle{j_n}\)} (b-2-2);
		\end{tikzpicture}
	\end{center}
	is coCartesian, and \(k_{n-1}\) is a trivial cofibration by using the cancellation property with the map \(j_{n-1}\).
	Therefore, it follows that \(j_n\) is also a trivial cofibration. 

	We now prove the lemma:  By the cancellation property, it suffices to show that
	\[V_{\operatorname{Sp}[n]}(X_1,\dots,X_n) \hookrightarrow V_{\Lambda^n_k}(X_1,\dots,X_n)\]
	is a trivial cofibration for \(n\geq 2\) and \(0<i<n\).  The case \(n=2\) is obvious, so it suffices to show for the case \(n>2\).  Given a set \(S\subseteq [n]\), let \[\Lambda^n_S=\bigcup_{i\notin S} \partial_i \Delta^n.\]  
	We will show that for \(n>2\) and \(S\) a nonempty subset of \([1,\dots,n-1]\), the map 
	\[V_{\operatorname{Sp}[n]}(X_1,\dots,X_n) \hookrightarrow V_{\Lambda^n_S}(X_1,\dots,X_n)\]
	is a trivial cofibration.  We argue by induction on \(n\) and \(s=n-\operatorname{Card}(S)\).  If \(s=1\), \(\Lambda^n_S=\partial_0 \Delta^n \cup \partial_n \Delta^n\), in which case we are done by the previous argument.  If \(s>1\), let \(T=S\cup \{b\}\) for some \(b \in [1,\dots,n-1] \setminus S\).  Then by the inductive hypothesis,
	\[V_{\operatorname{Sp}[n]}(X_1,\dots,X_n) \hookrightarrow V_{\Lambda^n_T}(X_1,\dots,X_n)\]
	is a trivial cofibration.  Then it suffices to show that 
	\[V_{\Lambda^n_T}(X_1,\dots,X_n) \hookrightarrow V_{\Lambda^n_S}(X_1,\dots,X_n)\]
	is a trivial cofibration.   
	We see that the diagram
	\begin{center}
		\begin{tikzpicture}
			\matrix (b) [matrix of math nodes, row sep=3em, column sep=3em]
			{
				V_{\Lambda^n_T \cap \partial_b \Delta^n }(X_1,\dots,X_n) & V_{\Lambda^n_T}(X_1,\dots,X_n)\\
				V_{\partial_b\Delta^n}(X_1,\dots,X_n)				 &  V_{\Lambda^n_S}(X_1,\dots,X_n)           \\
			};
			\path[->]
			(b-1-1) edge (b-1-2) edge  (b-2-1)
			(b-2-1) edge  (b-2-2)
			(b-1-2) edge  (b-2-2);
		\end{tikzpicture}
	\end{center}
	is a pushout, and therefore, it suffices to show that 
	\[V_{\Lambda^n_T \cap \partial_b \Delta^n }(X_1,\dots,X_n) \hookrightarrow V_{\partial_b\Delta^n}(X_1,\dots,X_n)\]
	is a trivial cofibration.  But this is true by the inductive hypothesis on \(n\). Therefore, we are done. 
\end{proof}

\begin{prop} The map 
	\[\square_n^\lrcorner(\lambda^n_k,\delta^{c_1},\dots,\delta^{c_n})\times \Delta^0\] is a trivial cofibration for the model structure on complete \(\Theta[\C]\)-Segal spaces.   
\end{prop}
\begin{proof} 
	We again suppress the \(\times \Delta^0\) factor. Let \(Q=\square_n^\lrcorner(\lambda^n_k,\delta^{c_1},\dots,\delta^{c_n})\).  Evaluation of \(Q\) on \(0\) is the source and evaluation on \(1\) is the target. We must show that the monomorphism \(Q:Q(0)\hookrightarrow Q(1)\) is a trivial cofibration. Notice first that 
	\[\square_n(\Lambda^n_k, c_1,\dots,c_n)\hookrightarrow Q(0) \hookrightarrow Q(1) = [n](c_1,\dots,c_n)\] is a weak equivalence by the lemma.  Then by right-cancellation, it suffices to show that 
	\[\square_n(\Lambda^n_k, c_1,\dots,c_n)\hookrightarrow Q(0)\] is a trivial cofibration.
	Let 
	\[U(s,t)=\int^{u_0,\dots,u_n} \left([1](u_0,s) \times [1](u_1\wedge\dots\wedge u_n,t)\right)\cdot \square(\lambda^n_k(u_0),\delta^{c_1}(u_1),\dots, \delta^{c_n}(u_n)),\]
	where evaluation on \(u_i\in [1]\) denotes taking the source or target.  
	Then we see by coend reduction that
	\[\int^{s,t} [1](s\wedge t,x) \times [1](u_0,s) \times [1](u_1\wedge\dots\wedge u_n,t) = [1](u_0 \wedge u_1 \wedge \dots\wedge u_n,x),\]
	so by commutation of coends, we see that
	\[Q(x)=\int^{s,t} [1](s\wedge t,x) U(s,t),\]
	which proves that \[Q(0)=U(1,0) \coprod_{U(0,0)} U(0,1),\]
	but \(U(0,1)=\square_n(\Lambda^n_k, c_1,\dots,c_n)\), so the map 
	\[\square_n(\Lambda^n_k, c_1,\dots,c_n) \hookrightarrow Q(0)\] is a pushout of \(U(0,0)\to U(1,0)\), which we will show is a trivial cofibration.  Notice that in \(U(0,0)\), everything vanishes when \(u_0=1\), so we have that 
	\[U(0,0)\cong\int^{u_1\dots,u_n} [1](u_1\wedge\dots\wedge u_n,0)\cdot \square(\lambda^n_k(0),\delta^{c_1}(u_1),\dots, \delta^{c_n}(u_n)).\]
	Notice also that by cofinality, we have that 
	\[U(1,0)=\int^{u_0,\dots,u_n} \left([1](u_0,1) \times [1](u_1\wedge\dots\wedge u_n,0)\right)\cdot \square(\lambda^n_k(u_0),\delta^{c_1}(u_1),\dots, \delta^{c_n}(u_n)),\]
	is isomorphic to 
	\[\int^{u_1,\dots,u_n} [1](u_1\wedge\dots\wedge u_n,t)\cdot \square(\lambda^n_k(1),\delta^{c_1}(u_1),\dots, \delta^{c_n}(u_n)).\]
	Then the map \(U(0,0)\hookrightarrow U(1,0)\) is induced by the natural maps
	\[\square(\lambda^n_k(0),\delta^{c_1}(u_1),\dots,\delta^{c_n}(u_n)) \hookrightarrow \square(\lambda^n_k(1),\delta^{c_1}(u_1),\dots,\delta^{c_n}(u_n)).\]
	But \(\lambda^n_k(0)\hookrightarrow \lambda^n_k(1)\) is the inner horn inclusion \(\Lambda^n_k\hookrightarrow \Delta^n\), and therefore, by \Cref{precatproperty}, these are all trivial cofibrations.  But \(U(0,0)\) and \(U(0,1)\) are homotopy coends.
	
	To see this, notice that each of these objects can be computed as colimits over cubical diagrams with the terminal vertex removed.  Equipping these finite directed categories with the degree-raising Reedy structure, we see that a diagram is projectively cofibrant if and only if it is Reedy-cofibrant.  To see that the diagrams in question are Reedy-cofibrant, it suffices to notice that the latching object at any vertex is a union of subobjects, which implies that the latching map is monic at each vertex, and consequently that the diagram is Reedy-cofibrant. 
	
	Therefore, the map \(U(0,0)\hookrightarrow U(0,1)\) is a monic weak equivalence and therefore a trivial cofibration, which proves the proposition.
\end{proof}
This proves one direction of the theorem; now we prove the converse.
\begin{prop} The maps \[\operatorname{Se}[n](c_1,\dots,c_n):\operatorname{Sp}[n](c_1,\dots,c_n) \hookrightarrow [n](c_1,\dots,c_n)\] are horizontal inner anodyne.
\end{prop}
\begin{proof} 
	Since the map 
	\[\lambda^n: \Lambda^n=\partial_0 \Delta^n \cup \partial_n\Delta^n \hookrightarrow \Delta^n\]
	is inner anodyne, and since the empty maps \(e^{c_i}:\varnothing \hookrightarrow c_i\) are monic, it follows by \ref{cornertensorcell} that the corner-intertwiner 
	\[\square^\lrcorner_n(\lambda^n,e^{c_1},\dots,e^{c_n})\]
	is horizontal inner anodyne.  However, it is easy to see that this map is exactly \[\square_n(\Lambda^n,c_1,\dots,c_n) \hookrightarrow [n](c_1,\dots,c_n).\]  Therefore, it suffices to show that the map 
	\[\square_n(\operatorname{Sp}[n],c_1,\dots,c_n) \hookrightarrow \square_n(\Lambda^n,c_1,\dots,c_n)\] is a horizontal inner anodyne.  We will first show that the map 
	\[\square_n(\operatorname{Sp}[n]\cup \partial_0\Delta^n,c_1,\dots,c_n) \hookrightarrow \square_n(\Lambda^n,c_1,\dots,c_n)\] is horizontal inner anodyne.  
	To see this, we proceed by induction on \(n\).  This is immediate for \(n\leq 2\). Suppose \(n>2\).  Then the map 
	\[\square_n(\operatorname{Sp}[n]\cup \partial_0\Delta^n,c_1,\dots,c_n) \hookrightarrow \square_n(\Lambda^n,c_1,\dots,c_n)\] is horizontal inner anodyne, as it is a pushout of the map 
	\[\square_n(\operatorname{Sp}[n-1]\cup \partial_0\Delta^{n-1},c_1,\dots,c_n) \hookrightarrow \square_n(\partial_n \Delta^n,c_1,\dots,c_n),\]
	which is horizontal inner anodyne by the induction hypothesis.  Then it suffices to show that 
	\[\square_n(\operatorname{Sp}[n],c_1,\dots,c_n) \hookrightarrow \square_n(\operatorname{Sp}[n]\cup \partial_0\Delta^n,c_1,\dots,c_n)\] is horizontal inner anodyne.  Again, we proceed by induction on \(n\) and notice that this is clear for \(n\leq 2\), but we see immediately that \[\square_n(\operatorname{Sp}[n],c_1,\dots,c_n) \hookrightarrow \square_n(\operatorname{Sp}[n]\cup \partial_0\Delta^n,c_1,\dots,c_n)\] is a pushout of 
	\[\square_n(\operatorname{Sp}[n-1],c_1,\dots,c_n) \hookrightarrow \square_n(\partial_0 \Delta^n ,c_1,\dots,c_n),\]
	which is horizontal inner anodyne by the induction hypothesis, which concludes the proof.
\end{proof}
\begin{cor}
	The left Kan extension of the functor \[Y \times E^\bullet:\Theta[\C]\times \Delta \to \cellset\] (where \(Y\) is the Yoneda embedding) defined by the rule 
	\[[n](c_1,\dots,c_n)\times \Delta^m \mapsto [n](c_1,\dots,c_n) \times E^m\]
	induces a Quillen equivalence
	\[\Psh_\Delta(\Theta[\C])_{\mathrm{CSS}} \underset{\operatorname{Sing}_E}{\overset{\operatorname{Real}_E}{\rightleftarrows}} \cellset_{\mathrm{hJoyal}}\]
	between the model structure for complete \(\Theta[\C]\)-Segal spaces and the horizontal Joyal model structure,
	and the left Kan extension of the functor 
	\[d:\Theta[\C] \to \Psh_\Delta(\Theta[\C])\]
	defined by the rule
	\[[n](c_1,\dots,c_n) \mapsto [n](c_1,\dots,c_n) \times \Delta^0\]
	induces a Quillen equivalence
	\[\cellset_{\mathrm{hJoyal}} \underset{d^\ast}{\overset{d_!}{\rightleftarrows}} \Psh_\Delta(\Theta[\C])_{\mathrm{CSS}}.\]
	That is to say, the two model categories are Quillen bi-equivalent.
\end{cor}
\begin{proof} This is an immediate consequence of the previous proposition together with \cite{cisinski-book}*{Proposition 2.3.27}.
\end{proof}

\subsection{Recognition of horizontal Joyal fibrations}\label{admissible}
In this section, we will prove the anlogue of Joyal's pseudofibration theorem for the horizontal model structure on \(\Theta[\C]\).  We will need to set up some definitions.

\begin{defn}
	Recall from Definition \ref{productintertwiner}, we defined the functor \(H_\mathbf{S}\) for a finite family of simplicial sets \(\mathbf{S}\).  Consider the case of the functor \(H_{S,n}\), where the family is made up of two simplicial sets \(\Delta^n\) and some simplicial set \(S\).  Then we define the functor 
	\[H_{S,n}: \overcat{\psh{\Delta}}{(S\times \Delta^n)} \times  \left(\ssetlab\right)_{n} \to \ssetlab,\]
	to be the restriction of \(H_{S,n}\) to the terminal labeling of \(S\), the unique labeling of \(S\) that sends all edges of \(S\) to the terminal object of \(\psh{\C}\).   Composing \(\square\) with \(H_{S,n}\) is denoted by \(\square_{S,n}\).  For the remainder of this section, we also, by abuse of notation, define 
	\[\square_S:\overcat{\psh{\Delta}}{S} \to \cellset\]
	to be the composite of \(\square\) with the restriction of \(H_S\) to the terminal labeling of \(S\).  
\end{defn}

\begin{obs}\label{joyalsquare1}
	Observe that from \ref{conncolims} the functor 
	\[\square_{S,n}:\overcat{\psh{\Delta}}{(S\times\Delta^n)} \times\psh{\C}^n \to \cellset\]
	preserves connected colimits in each argument and therefore can be corner tensored. By the same argument as \ref{square1}, given \(h:Y\to S\) a map of simplicial sets, we can compute 
	\begin{align*}
		 h \times^\lrcorner \square^\lrcorner_n(f_0,\dots,f_n)  &= \square_S^\lrcorner(h)\times^\lrcorner \square^\lrcorner_n(f_0,\dots,f_n)\\
		&= \square^\lrcorner_{S,n}(h\times^\lrcorner f_0,f_1,\dots,f_n).
	\end{align*}
\end{obs}

In what follows, we will use a very nice observation of Danny Stevenson \cite{danny}.  Consider the case of the simplicial set \(E^1\) and the map \(e:\Delta^0 \hookrightarrow E^1\).  
\begin{lemma}[\cite{danny}*{Lemma 2.19}]\label{dannythm}
	The map of simplicial sets
	\[e\times^\lrcorner \delta^n: E^1\times \partial\Delta^n \cup \Delta^0 \times \Delta^n \hookrightarrow E\times \Delta^n,\]
	is inner anodyne for all \(n>0\).  
\end{lemma}
\begin{obs}\label{retractjoyal}
	It follows from the small object argument that we can factor \(e\times^\lrcorner \delta^n\) as a composite
	\[
	E^1 \times \partial \Delta^n \coprod_{\Delta^0 \times \partial \Delta^n} \Delta^0 \times \Delta^n \xrightarrow{\iota} E \xrightarrow{\varepsilon} E^1 \times \Delta^n,
	\] 
	where \(\iota\) is a relative cell complex for the inner horn inclusions and where \(\varepsilon\) is an inner fibration between quasicategories. Since \(e\times^\lrcorner \delta^n\) is inner anodyne, it follows that we have a lift \(E^1\times \Delta^n \xrightarrow{\eta} E\) by the lifting property that exhibits \(e\times^\lrcorner \delta^n\) as a retract of \(\iota\).  
\end{obs}
\begin{obs}\label{joyalsquare2}
	A simplex \((\alpha,\beta):\Delta^r \to E^1\times \Delta^n\) is determined by the destination of its vertices.  We label the vertices by \((a,i)\) and \((b,i)\) for \(0\leq i \leq n\), where \(e\) is the inclusion of \(a \in E^1\).  Given a simplex \((\alpha,\beta):\Delta^r\to E^1\times \Delta^n\), and a family of presheaves \(\mathbf{U}=(U_i)_{i=1}^n\) on \(\C\). 
	\[K_{\beta}:\psh{\C}^n \to \psh{\C}^r\]
	to be the functor sending
	\[\mathbf{U} \mapsto \beta^\ast \mathbf{U}.\]
	Similar to \ref{square2}, we have a diagram
	\begin{center}
		\begin{tikzpicture}
			\matrix (b) [matrix of math nodes, row sep=3em, column sep=4em]
			{
				\overcat{\psh{\Delta}}{\Delta^r} \times \psh{\C}^n & \overcat{\psh{\Delta}}{\Delta^r} \times \psh{\C}^r\\
				\overcat{\psh{\Delta}}{E^1\times \Delta^n} \times \psh{\C}^n & \ssetlab \\
			};
			\path[->]
			(b-1-1) edge node[auto]{\(\scriptstyle{\id \times K_{\beta}}\)} (b-1-2) 
					edge node[auto]{\(\scriptstyle{((\alpha,\beta))\circ(-) \times \id \times \id}\)} (b-2-1)
			(b-2-1) edge node[auto]{\(\scriptstyle{H_{E^1,n}}\)} (b-2-2)
			(b-1-2) edge node[auto]{\(\scriptstyle{H_r}\)} (b-2-2);
		\end{tikzpicture},
	\end{center}
	and the diagram commutes by a direct computation.  
	By abuse of notation, set \(\beta(-1)=0\) and \(\beta(r+1)=n\).  Then we define the family
	\[(t_i=\beta(i) - \beta(i-1))_{i=0}^{r+1}\] and let 
	\[\tau_i:\psh{\C}^n \to \psh{\C}^{t_i}\]
	be the map defined by sending 
	\[\mathbf{U}\mapsto (U_{\beta(i-1)+1},\dots,U_{\beta(i)}).\]
	Then this family of maps defines a map 
	\[\psh{\C}^n \to \prod_{i=1}^r\psh{\C}^{t_i}\]
	which is a permutation and therefore an isomorphism. Then for each \(0<i<r+1\) let
	\[P_i:\psh{\C}^{t_i} \to \psh{\C}\]
	be the map defined by the rule 
	\[(X_1,\dots,X_{t_i})\mapsto X_1\times \dots \times X_{t_i},\]
	and for \(i=0\) or \(i=r+1\), define 
	\[P_i:\psh{C}^{t_i} \to \ast\]
	to be the terminal functor.
	Then these \(P_i\) assemble to a map 
	\[(P_0,\dots,P_{r+1}): \prod_{i=0}^{r+1} \psh{\C}^{t_i} \to \psh{\C}^r,\]
	such that \((P_0,\dots,P_{r+1})\circ \tau = K_\beta\).
	Then we have a commutative diagram
	\begin{center}
		\begin{tikzpicture}
			\matrix (b) [matrix of math nodes, row sep=3em, column sep=4em]
			{
				\overcat{\psh{\Delta}}{\Delta^r} \times \psh{\C}^n & \overcat{\psh{\Delta}}{\Delta^r} \times \prod_{i=0}^{r+1} \psh{\C}^{t_i} & \overcat{\psh{\Delta}}{\Delta^r} \times \psh{\C}^r\\
				\overcat{\psh{\Delta}}{E^1\times \Delta^n} \times \psh{\C}^n & & \ssetlab \\
			};
			\path[->]
			(b-1-1) edge node[auto]{\(\scriptstyle{\id \times \tau}\)} (b-1-2) 
					edge node[auto]{\(\scriptstyle{((\alpha,\beta))\circ(-) \times \id}\)} (b-2-1)
			(b-2-1) edge node[auto]{\(\scriptstyle{H_{E^1,n}}\)} (b-2-3)
			(b-1-2) edge node[auto]{\(\scriptstyle{\id \times (P_i)_{i=0}^{r+1}}\)} (b-1-3) 
			(b-1-3)	edge node[auto]{\(\scriptstyle{H_r}\)} (b-2-3);
		\end{tikzpicture},
	\end{center}
	and therefore, composing the bottom horizontal and right vertical maps with \(\square\), we have another commutative diagram
	\begin{center}
		\begin{tikzpicture}
			\matrix (b) [matrix of math nodes, row sep=3em, column sep=4em]
			{
				\overcat{\psh{\Delta}}{\Delta^r} \times \psh{\C}^n & \overcat{\psh{\Delta}}{\Delta^r} \times \prod_{i=0}^{r+1} \psh{\C}^{t_i} & \overcat{\psh{\Delta}}{\Delta^r} \times \psh{\C}^r\\
				\overcat{\psh{\Delta}}{E^1\times \Delta^n} \times \psh{\C}^n & & \cellset \\
			};
			\path[->]
			(b-1-1) edge node[auto]{\(\scriptstyle{\id \times \tau}\)} (b-1-2) 
					edge node[auto]{\(\scriptstyle{((\alpha,\beta))\circ(-) \times \id}\)} (b-2-1)
			(b-2-1) edge node[auto]{\(\scriptstyle{\square_{E^1,n}}\)} (b-2-3)
			(b-1-2) edge node[auto]{\(\scriptstyle{\id \times (P_i)_{i=0}^{r+1}}\)} (b-1-3) 
			(b-1-3)	edge node[auto]{\(\scriptstyle{\square_r}\)} (b-2-3);
		\end{tikzpicture}.
	\end{center}
	We see that each of the arrows in this picture preserves connected colimits argument-by-argument, so applying the corner tensor functor, we obtain a commutative diagram
	\begin{center}
		\begin{tikzpicture}
			\matrix (b) [matrix of math nodes, row sep=3em, column sep=3.5em]
			{
				\overcat{\psh{\Delta}}{\Delta^r}^{[1]} \times (\psh{\C}^{[1]})^n & \overcat{\psh{\Delta}}{\Delta^r}^{[1]} \times \prod_{i=0}^{r+1} (\psh{\C}^{[1]})^{t_i} & \overcat{\psh{\Delta}}{\Delta^r}^{[1]} \times (\psh{\C}^{[1]})^r\\
				\overcat{\psh{\Delta}}{E^1\times \Delta^n}^{[1]} \times (\psh{\C}^{[1]})^{n} & & \cellset^{[1]} \\
			};
			\path[->]
			(b-1-1) edge node[auto]{\(\scriptstyle{\id \times \tau}\)} (b-1-2) 
					edge node[auto]{\(\scriptstyle{((\alpha,\beta))\circ(-)^\lrcorner \times \id}\)} (b-2-1)
			(b-2-1) edge node[auto]{\(\scriptstyle{\square^\lrcorner_{E^1,n}}\)} (b-2-3)
			(b-1-2) edge node[auto]{\(\scriptstyle{\id \times (P_i^\lrcorner)_{i=0}^{r+1}}\)} (b-1-3) 
			(b-1-3)	edge node[auto]{\(\scriptstyle{\square^\lrcorner_r}\)} (b-2-3);
		\end{tikzpicture}.
	\end{center}
\end{obs}
\begin{lemma}
	Let \((\alpha,\beta): \Delta^r \to E^1\times \Delta^n\) be a nondegenerate section with \(r\geq 2\), let \[\mathbf{f}=\{f_i:\partial c_i \hookrightarrow c_i\}_{i=1^n}\] be a family of boundary inclusions.  Then for any inner horn inclusion \(\lambda^r_k: \Lambda^r_k\hookrightarrow \Delta^r\), the map
	\[
		\square_{E^1,n}^\lrcorner(\lambda^r_k, \mathbf{f})
	\]
	is horizontal inner anodyne.
\end{lemma}
\begin{proof}
	The proof is practically identical to that of \Cref{anodynelemma} using \Cref{joyalsquare2} in place of \Cref{square2}.  
\end{proof}
\begin{thm}\label{joyalisothm}
	Set \(e:\Delta^0\to E^1\). Then for any boundary inclusion 
	\[\square_n^\lrcorner(\delta^n,\delta^{c_1},\dots,\delta^{c_n}),\]
	with \(n>0\), the map 
	\[e \times^\lrcorner \square_n^\lrcorner(\delta^n,\delta^{c_1},\dots,\delta^{c_n})\]
	is a horizontal inner anodyne.
\end{thm}
\begin{proof}
	By \ref{retractjoyal} we see that \(e\times^\lrcorner \delta^n\) can be factored as \(\varepsilon\circ \iota\) such that it is a retract of \(\iota\), which is a relative cell complex of inner horn inclusions.  By \Cref{cornertensorcell}, it follows that 
	\[
		\square^\lrcorner_{E^1,n}(\iota, \delta^{c_1},\dots,\delta^{c_n})
	\]
	is transfinite composite of pushouts of inner horn inclusions
	\[
		\square^\lrcorner_{E^1,n}(h_i, \delta^{c_1},\dots,\delta^{c_n}),
	\]
	where the \(h_i:\Lambda^{r_i}_{\ell_i}\to \Delta^{r_i}\) are inner horn inclusions in \(\overcat{\psh{\Delta}}{E^1 \times \Delta^n}\).  By the previous lemma, each of these maps is horizontal inner anodyne, and it follows therefore that the map
	\[
		\square^\lrcorner_{E^1,n}(\iota, \delta^{c_1},\dots,\delta^{c_n})
	\]
	is as well.  
	But the map
	\[
		\square^\lrcorner_{E^1,n}(e\times^\lrcorner \delta^n, \delta^{c_1},\dots,\delta^{c_n})
	\]
	is a retract of 
	\[
		\square^\lrcorner_{E^1,n}(\iota, \delta^{c_1},\dots,\delta^{c_n}),
	\]
	which we have just shown to be horizontal inner anodyne. Then from \ref{joyalsquare1}, we see that 
	\[
		\square^\lrcorner_{E^1,n}(e\times^\lrcorner \delta^n, \delta^{c_1},\dots,\delta^{c_n}) \cong e \times^\lrcorner \square^\lrcorner_n(\delta^n,\delta^{c_1},\dots,\delta^{c_n}),
	\]
	which proves the theorem.
\end{proof}
\begin{cor}
	The formal \(\C\)-quasicategories are the fibrant objects of the horizontal Joyal model structure.
\end{cor}
\begin{proof}
	By Cisinski's theorem (see \Cref{cisinskimaintheorem}), an object is fibrant in a Cisinski model structure if it has the right lifting property with respect to all of the anodyne maps generated by pushout products of the generating cofibrations with the inclusion of either endpoint of interval object as well as pushout-products of the generating anodynes with pushout-product powers of the boundary inclusion into the interval object.  By \Cref{anodynethm}, we see that an object has the right lifting property with respect to the second set of maps if and only if it has the right lifting property with respect to the horizontal inner anodynes, since such maps all belong to the saturated class generated only by the horizontal inner horn inclusions, which are the generating anodynes in this situation. 
	
	Since the formal \(\C\)-quasicategories by definition have the right lifting property with respect to all inner anodynes, which are closed under pushout-products with arbitrary monomorphisms, and since the maps \(e\times^\lrcorner  \square_n^\lrcorner(\delta^n,\delta^{c_1},\dots,\delta^{c_n})\) are inner anodyne for all \(n>0\),  it suffices to show that any formal \(\C\)-quasicategory has the right lifting property with respect to the single map \(e\times^\lrcorner \delta^0\), but this map is isomorphic to the map \(e:\Delta^0 \hookrightarrow E^1\), and such a lift always exists by choosing the lift through the retraction \(E^1\to \Delta^0\).
\end{proof}
\begin{cor}
	The fibrations between fibrant objects in the horizontal Joyal model structure are horizontal inner fibrations having the right lifting property with respect to the map \(e:\Delta^0\hookrightarrow E^1\).
\end{cor}
\begin{proof} 
	A fibration between fibrant objects must have the right lifting property with respect to all horizontal inner anodynes and all maps of the form \(e\times^\lrcorner  \square_n^\lrcorner(\delta^n,\delta^{c_1},\dots,\delta^{c_n})\).  Since every inner fibration has the right lifting property with respect to all of those maps for \(n>0\), it follows that an inner fibration between fibrant objects need only have the right lifting property with respect to the case where \(n=0\), which is exactly the map \(e\).
\end{proof}

\section{The Coherent Nerve, Horizontal case}
In \cite{htt}, Lurie reintroduces an important adjunction \[\psh{\Delta}  \underset{\mathfrak{N}_\Delta}{\overset{\mathfrak{C}_\Delta}{\rightleftarrows}} \Cat_{\psh{\Delta}},\] coming originally from work of Cordier and Porter, where the left adjoint is called the \emph{coherent realization} and the right adjoint is called the \emph{coherent nerve}.  One of the significant early theorems in \cite{htt} demonstrates that this adjunction is in fact a Quillen equivalence between the Joyal model structure on the one hand and the Bergner model structure on the other.

We find that it is useful to instead consider this adjunction as one of the form \[\psh{\Theta[\ast]}  \underset{\mathfrak{N}_\Delta}{\overset{\mathfrak{C}_\Delta}{\rightleftarrows}} \Cat_{\Psh_\Delta(\ast)},\] where we have obvious isomorphisms \(\Theta[\ast]\cong \Delta\) and \(\Psh_\Delta(\ast)\cong \psh{\Delta}\).  This is suggestive of a generalization to a new case where we replace \(\ast\) with a small regular Cartesian Reedy category \(\C\). We will develop this adjunction throughout the current chapter, and we will demonstrate that an analogous Quillen equivalence indeed holds.
\subsection{The coherent realization for \(\Theta[\C]\)}
The goal of this section is to show that for any small regular Cartesian Reedy category \(\C\), we can construct a new adjunction
\[\cellset \underset{\mathfrak{N}}{\overset{\mathfrak{C}}{\rightleftarrows}} \Cat_{{\spsh}_{\operatorname{inj}}}\]
generalizing the coherent nerve and realization.  We will also give a useful computation of \(\mathfrak{C}\) in some special cases.

\begin{defn} 
	The category of \(\C\)-precategories is the full subcategory of \(\psh{\Delta \times \C}\) spanned by the \emph{precategory objects}, namely those simplicial presheaves \(F\) on \(\C\) such that \(F_0\) is a constant presheaf on \(\C\).  We denote this category by \(\mathbf{PCat}(\C)\). 
\end{defn}

\begin{defn}
	The functor \(k: \Delta \times \C\to \Theta[\C]\) defined by the rule 
	\[
		([n],c) \mapsto [n](c,\dots,c).
	\]
	Induces a colimit-preserving functor \(k^\ast: \cellset \to \psh{\Delta\times\C}\) that lands in \(\mathbf{PCat}(\C)\) called the \emph{associated precategory}.
\end{defn}

\begin{defn}\label{pointwisedefn}
	The \emph{pointwise realization} functor \(\mathfrak{C}_{\Delta,\bullet}: \psh{\Delta\times\C} \to \Cat_{\psh{\Delta}}^{\C^\op}\) defined by the rule 
	\[
		\mathfrak{C}_{\Delta,\bullet}(X)_c \defeq \mathfrak{C}_{\Delta}(X_c)
	\]
		restricts to colimit preserving functor \(\mathfrak{C}_{\Delta,\bullet}: \mathbf{PCat}(\C) \to \Cat_{\spsh}\) in the obvious way, which by abuse of notation, we also refer to as the \emph{pointwise realization}.
\end{defn}

\begin{defn}
	The \emph{coherent realization} \(\mathfrak{C}_{\Theta[\C]}:\cellset \to \Cat_{\spsh},\) also denoted by \(\mathfrak{C}\) by abuse of notation when \(\C\) is fixed is defined as the composite: 
	\[
		\cellset \xrightarrow{k^\ast} \mathbf{PCat}(\C) \xrightarrow{\mathfrak{C}_{\Delta,\bullet}} \Cat_{\spsh}.
	\] 
\end{defn}
\begin{note}
	It is immediate from the cocontinuity of each functor in this composite that the functor \(\mathfrak{C}_{\Theta[\C]}\) is cocontinuous and therefore determined on representables. In what follows, we will give an explicit computation of its values on representables and in fact more generally on cellular sets of the form \(V[n](A_1,\dots,A_n)\).
\end{note}
\begin{rem}
	It is easy to see that if we substitute the terminal category for \(\C\), this specializes precisely to the usual coherent realization \(\mathfrak{C}_{\Delta} = \mathfrak{C}_{\Theta[\ast]}\).  
\end{rem}

We will extensively abuse notation in what follows by identifying a simplicial set with its associated constant simplicial presheaf on \(\C\) and identifying a presheaf on \(\C\) with its associated discrete simplicial presheaf.
\begin{defn}
	We define a construction on objects \[Q:\Delta\int\psh{\C}\to \Cat_{\spsh}.\] Suppose \([n](X_1,\dots, X_n)\) is any object of \(\Delta\int\psh{\C}\). Then we define \(Q([n](X_1,\dots,X_n))\) as follows:
	\begin{itemize}
		\item The objects are the vertices \(\{0,\dots,n\}\)
		\item The Hom-object
		      \[\Hom(i,j)=
			      \begin{cases}
				      \varnothing                                                & \text{ for } i>j \\
				      c\Delta^0                                                & \text{ for } i=j \\
				      X_{i+1} \times \Delta^1 \times X_{i+2}\times \dots \times \Delta^1 \times X_j & \text{ for } i<j
			      \end{cases}
		      \]
		\item The associative composition law, \(\Hom(i,j)\times \Hom(j,k)\to \Hom(i,k)\) which is the inclusion on the bottom face with respect to \(j\):
		      \begin{align*}
			      X_{i+1}\times \Delta^1\times\dots\times\Delta^1\times X_j \times & \{1\}\times X_{j+1} \times \Delta^1\times\dots\times\Delta^1\times X_k    \\
			                                                                       & \downarrow                                                                \\
			      X_{i+1}\times \Delta^1\times\dots\times\Delta^1\times X_j\times  & \Delta^1\times X_{j+1} \times \Delta^1\times\dots\times\Delta^1\times X_k
		      \end{align*}
	\end{itemize}
\end{defn}
\begin{prop} The construction \(Q\) is functorial.
\end{prop}
\begin{proof}
	Recall that a map \[[n](X_1,\dots, X_n) \to [m](Y_1,\dots Y_m)\] in \(\Delta \int \psh{\C}\) is given by a pair \((\gamma,	\mathbf{f})\), where \(\gamma:[n]\to [m]\) is a map of simplices together with a family of maps \[\mathbf{f}=\left(f_i: 	X_i \to \prod_{j=\gamma(i-1)}^{\gamma(i)}Y_j\right)_{i=1}^n.\]
	If for \(0<i\leq n\), we have \(\gamma(i-1)=\gamma(i)\), we can see easily that \(\gamma\) factors through the 	codegeneracy map \[[n](X_1,\dots,X_n)\to [n-1](X_1,\dots,\psh{X_i},\dots,X_n).\]  Applying this factorization repeatedly, 	we factor \((\gamma,\mathbf{f})\) as a codegeneracy followed by a map \((\gamma^\prime,f^\prime)\) such that \(\gamma^\prime\) is the inclusion of a coface \([n^\prime]\hookrightarrow [m]\).

	Since \(\Delta\int \psh{\C}\) is fibred over \(\Delta\), we may take the Cartesian lift of \(\gamma^\prime\), which is 	the map \[\overline{\gamma^\prime}=(\gamma^\prime,\mathbf{id}):[n']\left(\prod_{j=\gamma^\prime(0)}^{\gamma^\prime(1)}Y_j,	 \dots, \prod_{j=\gamma^\prime(n^\prime-1)}^{\gamma^\prime(n^\prime)} Y_j \right)\hookrightarrow [m](Y_1,\dots,Y_m).\]  	By Cartesianness, we have a unique factorization of \((\gamma^\prime,\mathbf{f^\prime})\) by this map, yielding a map \[	(\id,\overline{f^\prime}):[n'](X^\prime_1,\dots,X^\prime_{n^\prime})\to [n']\left(\prod_{j=\gamma^\prime(0)}^	{\gamma^\prime(1)}Y_j, \dots, \prod_{j=\gamma^\prime(n^\prime-1)}^{\gamma^\prime(n^\prime)} Y_j \right).\]

	Then to prove the proposition, we need to show functoriality in three cases:
	\begin{itemize}
		\item If the map \((\gamma,\mathbf{f})\) is a codegeneracy of codimension \(1\), suppose \(\gamma=\sigma^i:[n+1]\to 	[n]\) for \(0\leq i\leq n\).  Then \(Q((\sigma^i,\mathbf{id}))_{ab}:\Hom(a,b)\to \Hom(\sigma^i(a),\sigma^i(b))\) is 	defined on the homs as follows:
		      \[Q(\sigma^i)_{ab}=
			      \begin{cases}
				      \id\times \operatorname{min}\circ \tau_{X_{i}} \times \id & \text{ if } a <  i \leq b \\
				      \id                                                       & \text{ otherwise }
			      \end{cases}\]
		      where \(\operatorname{min}:\Delta^1\times \Delta^1\to \Delta^1\) is induced by the map of posets sending \((x,y)	\mapsto \operatorname{min}(x,y)\) and \(\tau_{X_{i}}\to \ast\) is the terminal map.  Specifically, in the case 	where \(a< i \leq b\), the map is given by the composite:
		      \begin{align*}
			      X_{a+1}\times \Delta^1\times\dots\times\Delta^1\times & X_{i} \times \Delta^1 \times 	\dots\times\Delta^1\times X_{b} \\
			                                                            & \downarrow                                          	         \\
			      X_{a+1}\times \Delta^1\times\dots\times\Delta^1\times & \ast \times \Delta^1 \times 	\dots\times\Delta^1\times X_{b}  \\
			                                                            & \parallel                                           	         \\
			      X_{a+1}\times \Delta^1\times\dots\times\Delta^1       & \times \Delta^1 \times \dots\times\Delta^1\times X_	{b}       \\
			                                                            & \downarrow                                          	         \\
			      X_{a+1}\times \Delta^1\times\dots\times               & \Delta^1 \times \dots\times\Delta^1\times X_{b}     	         \\
		      \end{align*}
		      If \(i=0\) or \(i=n\), we consider \(\Delta^0=\ast\) to be \(\{0\}\).
		\item If the map \((\gamma,\mathbf{f})\) is a pure coface of codimension \(1\), we have two subcases: If it is an 	outer coface, the map is just the obvious inclusion.  If it is an inner coface, it has a term that looks like \(X_{i}	\times X_{i+1}\), and this is included in all of the \(\Hom\) objects as \(X_{i}\times \{0\} \times X_{i+1}\)
		\item If the map \((\gamma,\mathbf{f})\) is such that \(\gamma=\id\), since each of the \(\Hom\) objects is given as 	a product of the \(X_i\) with \(\Delta^1\) of the same length, just map them by \(f_a\times\Delta^1\times \dots 	\times\Delta^1 \times f_{b-1}\) using the functoriality of the Cartesian product.
	\end{itemize}
	It is an easy exercise to see that this assignment is functorial and completely analogous to the un-enriched case.
\end{proof}

\begin{note}
	We specify the following abuse of notation: Given an ordered family of sets \(S_1,\dots, S_n\), we denote by
	\[
		[n](S_1,\dots,S_n)
	\]
	the nerve of the free category on the directed graph specified as follows:
	\begin{itemize}
		\item The set of vertices is \(\{0,\dots,n\}\), and
		\item The set of arrows from the vertex \(i-1\) to \(i\) is \(S_i\).
	\end{itemize}
\end{note}

\begin{prop}\label{pointwise}
	The functor \(Q:\Delta\int\psh{\C}\to \Cat_{\spsh}\) factors as the composite
	\[
		\Delta\int\psh{\C} \xrightarrow{V} \cellset \xrightarrow{\mathfrak{C}} \Cat_{\spsh}
	\]
\end{prop}
\begin{proof} 
	Since for any \([m](X_1,\dots,X_m) \in \Delta\int\psh{\C}\), the enriched categories \(Q([m](X_1,\dots,X_m))\) and \(\mathfrak{C}(V[m](X_1,\dots,X_m))\) have object sets in natural bijection with the set of vertices of \([m]\) it suffices to produce, for each pair of vertices \(i, j \in [m]\) a natural isomorphism
	\[
		\mathfrak{C}(V[m](X_1,\dots,X_m))(i, j) \to Q([m](X_1,\dots,X_m))(i,j),
	\]
	which amounts to the data of a natural (in \(c\)) isomorphism of simplicial sets
	\[
		\mathfrak{C}(V[m](X_1,\dots,X_m))(x, y)_c \to Q([m](X_1,\dots,X_m))(x,y)_c.
	\]
	First, observe that specifying a simplex
	\[
		\Delta^n \to Q([m](X_1,\dots,X_m))(i, j)_{c}
	\]
	is equivalent to specifying a morphism
	\begin{align*}
		(\Delta^n, c) &\to X_{i+1} \times \Delta^1 \times \dots \times \Delta^1 \times X_j\\ 
		&=(\Delta^1)^{j-i-1} \times \prod_{k=i+1}^j X_k
	\intertext{which is itself equivalent to specifying a simplex}
		\Delta^n &\to (\Delta^1)^{j-i-1} \times \prod_{k=i+1}^j X_k(c),
	\end{align*}
	so in particular, we may make the identification
	\begin{align*}
		Q([m](X_1,\dots,X_m))(i, j)_{c} &\cong \mathfrak{C}_\Delta([m](X_1(c),\dots,X_m(c)))(i, j),
	\intertext{and therefore it follows that we have an isomorphism, natural in \(c\):}
		Q([m](X_1,\dots,X_m))_{c} &\cong \mathfrak{C}_\Delta([m](X_1(c),\dots,X_m(c))).
	\end{align*}
	It therefore suffices to demonstrate a natural isomorphism of simplicial sets
	\[
		k^\ast(V[m](X_1,\dots,X_m))_c\cong [m](X_1(c),\dots,X_m(c)),
	\]
	but it can be observed that specifying a simplex
	\[
		\Delta^n \to k^\ast(V[m](X_1,\dots,X_m))_c
	\]
	corresponds to specifying a map
	\[
		[n](c,\dots,c) \to V[m](X_1,\dots,X_m),
	\]
	which in turn is given by the data of a map
	\[
		\gamma:[n]\to [m]
	\]
	together with maps
	\[
		c\to \prod_{k=\gamma(0)+1}^{\gamma(n)} X_k,
	\]
	or equivalently, a family
	\[
		(f_{\gamma(0)+1}, \dots, f_{\gamma(n)}) \in \prod_{k=\gamma(0)+1}^{\gamma(n)} X_k(c),
	\]
	which specifies a unique simplex
	\[
		\Delta^n \to [m](X_1(c),\dots,X_m(c)),
	\]
	which proves the claim.
\end{proof}

\subsection{The Bergner-Lurie model structure on \(\Cat_{\spsh_{\mathrm{inj}}}\)}
In this section, we cite important results from \cite{htt}*{Appendix A.3} concerning the generalized Bergner model structure for categories enriched in an excellent monoidal model category.  

\begin{defn}
	Let \(\mathbf{S}\) be a monoidal model category, and let \(X\) be an \(\mathbf{S}\)-enriched category.  Then we define the \emph{homotopy category} \(\mathbf{h}X\) to be the ordinary category underlying the the associated \(\mathbf{h}\mathbf{S}\)-enriched category also denoted by \(\mathbf{h}X\).  
\end{defn}

\begin{defn}
	Let \(\mathbf{S}\) be a monoidal model category, and let \(f:X\to Y\) be an \(\mathbf{S}\)-enriched functor of \(\mathbf{S}\)-enriched categories.  
	
	\begin{itemize}
		\item We say that \(f\) is \emph{weakly fully faithful} if for every pair of objects \(x,x^\prime\) of \(X\), the component map 
		\[
			f_{x,x^\prime}: X(x,x^\prime) \to Y(f(x),f(x^\prime))
		\]
		is a weak equivalence of \(\mathbf{S}\).
		\item We say that \(f\) is \emph{weakly essentially surjective} if the induced functor on homotopy categories \(\mathbf{h}X\to \mathbf{h}Y\) is essentially surjective, that is, if for every object \(y\) of \(\mathbf{h}Y\) there exists an object \(x\) of \(\mathbf{h}X\) and an isomorphism \(y\cong f(x)\) in \(\mathbf{h}Y\).
		\item We say that \(f\) is an \emph{\(\mathbf{S}\)-enriched weak equivalence} if it is weakly fully faithful and weakly essentially surjective.
	\end{itemize}
\end{defn}
\begin{defn}
	Let \(\mathbf{S}\) be a monoidal model category. 

	We say that an \(\mathbf{S}\)-enriched category \(X\) is \emph{locally fibrant} if for every pair of objects \(x,x^\prime\) of \(X\), the object of morphisms \(X(x,x^\prime)\) is a fibrant object of \(\mathbf{S}\).
	
	We say that an \(\mathbf{S}\)-enriched functor of \(\mathbf{S}\)-enriched categories \(f:X\to Y\) be an  is a \emph{local fibration} if the following two conditions hold:
	\begin{enumerate}
		\item  For every pair of objects \(x,x^\prime\) of \(X\), the component map 
		\[
			f_{x,x^\prime}: X(x,x^\prime) \to Y(f(x),f(x^\prime))
		\]
		is a fibration.
		\item The induced functor on homotopy categories \(\mathbf{h}X\to \mathbf{h}Y\) is an isofibration of ordinary categories.
	\end{enumerate}
\end{defn}
\begin{defn}
	Let \(\mathbf{S}\) be a monoidal model category.  We will define the \emph{categorical suspension} functor \(\mathbf{2}:\mathbf{S} \to \Cat_{\mathbf{S}}\).  Given an object \(S\) of \(\mathbf{S}\), we define \(\mathbf{2}(S)\) as follows:  
	\begin{itemize}
		\item The set of objects of \(\mathbf{2}(S)\) is precisely the set \(\{0,1\}\).
		\item The object of morphisms is defined by 
		\[
			\mathbf{2}(S)(i,j)=
			\begin{cases}
				\mathbf{1}_\mathbf{S} & \text{if}\quad i=j\\
				S & \text{if}\quad i<j\\
				\varnothing & \text{if} \quad i>j
			\end{cases},
		\]  
	\end{itemize}
	where \(\mathbf{1}_{\mathbf{S}}\) is the unit object of \(\mathbf{S}\). The extension of the definition to morphisms is the obvious one.
	We also define the \(\mathbf{S}\)-enriched category \({[0]}_{\mathbf{S}}\) to be the enriched category with one object whose object of endomorphisms is exactly \(\mathbf{1}_{\mathbf{S}}\).  
\end{defn}
\begin{prop}[\cite{htt}*{Proposition A.3.2.4}]
	Suppose \(\mathbf{S}\) is a monoidal combinatorial model category in which all objects are cofibrant.  Then there exists a left-proper combinatorial model structure on \(\Cat_{\mathbf{S}}\) with weak equivalences the \(\mathbf{S}\)-enriched weak equivalences as defined above and with cofibrations the weakly saturated class generated by the set
	\[
		\{\varnothing \hookrightarrow {[0]}_{\mathbf{S}}\}\cup \{\mathbf{2}(f)\mid f \enskip \text{is a generating cofibration of} \enskip \mathbf{S}\} 
	\]
	This model structure is called the Lurie-Bergner model structure on \(\mathbf{S}\)-enriched categories.
\end{prop}
\begin{defn}[\cite{htt}*{Definition A.3.2.16}]
	A model category \(\mathbf{S}\) equipped with a monoidal product \(\otimes\) is called \emph{excellent} if the following conditions hold:
	\begin{enumerate}[label=(A\arabic*{})]
		\item The model category \(\mathbf{S}\) is combinatorial.
		\item Every monomorphism of \(\mathbf{S}\) is a cofibration, and cofibrations are stable under products.
		\item The collection of weak equivalences of \(\mathbf{S}\) is stable under filtered colimits
		\item The monoidal structure of \(\mathbf{S}\) is compatible with the model structure.  That is, the tensor product is a left-Quillen bifunctor.
		\item The model category \(\mathbf{S}\) satisfies the invertibility hypothesis.
	\end{enumerate}
\end{defn}
We quickly make use of the following lemma:
\begin{lemma}[\cite{htt}*{Lemma A.3.2.20}]\label{excellentlemma}
	Let \(\mathbf{S}\) be an excellent monoidal model category, and let \(\mathbf{S}^\prime\) be a monoidal model category satisfying axioms (A1-A4).  Then if there exists a monoidal left-Quillen functor \(\mathbf{S}\to \mathbf{S}^\prime,\) it follows that \(\mathbf{S}^\prime\) is also excellent.
\end{lemma}
\begin{cor}
	For any small category \(\C\), the model category \(\spsh_{\mathrm{inj}}\) of simplicial presheaves on \(\C\) with the injective model structure and the Cartesian product is excellent.
\end{cor}
\begin{proof}
	It is clear that axioms (A1), (A2), and (A4) are satisfied. Axiom (A3) is also satisfied by recalling that the injective model structure is regular by \Cref{injregular} and therefore closed under filtered colimits by \Cref{filteredcolims}. It therefore suffices to demonstrate axiom (A5).  
	
	If we take \(p^\ast: \psh{\Delta} \to \spsh\) to be the functor induced by the projection \(p:\C\times \Delta \to \Delta\), it is clearly monoidal, as it preserves all limits.  Moreover, it is also clear that it sends monomorphisms to monomorphisms and weak equivalences to weak equivalences. Since \(\psh{\Delta}\) is excellent, it follows therefore from \Cref{excellentlemma} that \(\spsh_{\mathrm{inj}}\) is also excellent.
\end{proof}
\begin{cor}
	If \(\setS\) is a set of morphisms of \(\spsh\) such that the left-Bousfield localization of \(\spsh_{\mathrm{inj}}\) at \(\setS\) is again a Cartesian-monoidal model category, then the model category \(\spsh_{\setS}\) obtained from this localization is also excellent.
\end{cor}
\begin{proof}
	As before, axioms (A1) and (A2) are obviously satisfied.  Axiom (A4) is satisfied by hypothesis, and axiom (A3) follows from the fact that any localizer containing a regular localizer is again regular by \Cref{regularinclusion}.  It again suffices to demonstrate axiom (A5).

	Notice now that the identity functor is a monoidal left-Quillen functor \(\spsh_{\mathrm{inj}} \to \spsh_{\setS}\).  Then axiom (A5) again follows from \Cref{excellentlemma}.
\end{proof}
\begin{thm}[\cite{htt}*{Theorem A.3.2.24}]
	Let \(\mathbf{S}\) be an excellent monoidal model category.  Then the following two results hold:
	\begin{enumerate}
		\item An \(\mathbf{S}\)-enriched category \(X\) is Bergner-Lurie fibrant if and only if it is locally fibrant.
		\item An \(\mathbf{S}\)-enriched functor \(f:X\to Y\) is a fibration for the Bergner-Lurie model structure if and only if it is a local fibration.
	\end{enumerate}
\end{thm}
Tying this all together, we obtain the following characterization of the Bergner-Lurie model structure:
\begin{cor}
	For any small category \(\C\) and any set \(\setS\) of morphisms of \(\spsh\), there exists a left-proper combinatorial model structure on \(\Cat_{\spsh}\) characterized by the following classes of maps
	\begin{itemize}
		\item[(C)] 	The cofibrations are exactly the weakly saturated class generated by the set of maps
					\[
						\{\varnothing \hookrightarrow {[0]}_{\spsh}\}\cup \{\mathbf{2}(f)\mid f \enskip \text{is a generating cofibration of} \enskip \mathbf{\spsh}\} 
					\]
		\item[(W)]  The weak equivalences are exactly the \(\spsh_\setS\)-enriched weak equivalences.
		\item[(F)]  The fibrant objects are the \(\spsh\)-enriched categories whose \(\Hom\)-objects are \(\setS\)-local injectively-fibrant simplicial presheaves on \(\C\), and the fibrations with fibrant target are exactly the local fibrations. 
	\end{itemize}
\end{cor}
\begin{note}
	We denote the Bergner-Lurie model structure with respect to a set of maps \(\setS\) of \(\spsh\) by \(\Cat_{\spsh_{\setS}}\).  In the special case when \(\setS\) is empty, this reduces to the case where \(\spsh\) is equipped with the injective model structure, and we denote its associated Bergner-Lurie model category by \(\Cat_{\spsh_{\mathrm{inj}}}\).
\end{note}
\begin{rem}
	The absence of hypotheses on \(\C\) in this section is what leads us to believe that there may be a way to drop the hypothesis that \(\C\) is regular Cartesian Reedy, but this goes beyond the scope of this paper.
\end{rem}

\subsection{Necklaces and the coherent realization}
Necklaces were introduced by Dugger and Spivak in \cite{ds1} in order to understand the mapping objects \(\mathfrak{C}_\Delta(X)(x,y)\).  They prove a useful theorem that allows one to compute the coherent realization up to homotopy as a simplicially-enriched category whose hom-objects are the nerves of ordinary categories.  We will demonstrate here how their theory generalizes to our setting. We begin by recalling the definition of a necklace:

\begin{defn}
	A \emph{necklace} is a bi-pointed simplicial set \((T,(\alpha,\omega))\) of the form
	\[
		\Delta^{m_1} \vee \dots \vee \Delta^{m_n},
	\]
	with specified vertices 
	\[(\alpha,\omega):\Delta^0\coprod \Delta^0 \xrightarrow{\{0\}\coprod \{m_n\}} \Delta^{m_1} \coprod \Delta^{m_n}\xrightarrow{\iota_1\coprod \iota_n} T.\]
	By abuse of notation, we will simply refer to necklaces by the name of the simplicial set, suppressing the distinguished vertices \((\alpha,\omega)\).
	We define the category \(\Nec\) to be the full subcategory of bi-pointed simplicial sets spanned by the necklaces.
\end{defn}

Dugger and Spivak construct a functor
\[
	\mathfrak{C}^{\Nec}_\Delta: \psh{\Delta}\to \Cat_{\psh{\Delta}}
\]
whose evaluation on a simplicial set \(X\) is given as follows:
\begin{itemize}
	\item The set of objects objects of \(\mathfrak{C}^{\Nec}_\Delta(X)\) is \(X_0\).
	\item Given any two vertices \(x,x^\prime\in X_0\), the simplicial set of morphisms from \(x\) to \(x^\prime\) is given by the formula 
	\[
		\mathfrak{C}^{\Nec}_\Delta(X)(x,x^\prime)\defeq N\overcat{\Nec}{X_{x,x^\prime}},
	\]
	where \(\overcat{\Nec}{X_{x,x^\prime}}\) denotes the slice over the bi-pointed simplicial set \(X_{x,x^\prime}\) in the category of bi-pointed simplicial sets.  
	\item The composition map is induced by concatenation of necklaces.  That is, given a pair of necklaces \(T\to X_{x,x^\prime}\) and \(T^\prime\to X_{x^\prime,x^{\prime\prime}}\), their composite is given by the necklace 
	\[
		T \bigvee_{\omega_T,\alpha_{T^\prime}} T^\prime \to X_{x,x^{\prime\prime}}.
	\]
\end{itemize}

In order to compare \(\mathfrak{C}^{\Nec}_\Delta\) and \(\mathfrak{C}_\Delta\), Dugger and Spivak introduce an auxiliary functor \(\mathfrak{C}^{\Hoc}_\Delta\) that admits specified natural transformations to both.  This leads to the main theorem:

\begin{thm}[\cite{ds1}*{Theorem 5.2}]\label{dugspivthm}
	There is a specified natural zig-zag of weak equivalences of functors valued in simplicially-enriched categories:
	\[
		\mathfrak{C}^{\Nec}_\Delta \leftarrow \mathfrak{C}^{\Hoc}_\Delta \to \mathfrak{C}_\Delta.
	\]
\end{thm}

The functors \(\mathfrak{C}_\Delta^{\Nec}\) and \(\mathfrak{C}_\Delta^{\Hoc}\), much like the functor \(\mathfrak{C}_\Delta\), send simplicial sets to simplicially enriched categories with set of objects equal to the set of \(0\)-simplices.  Ergo, they induce useful functors when applied pointwise to precategories.
\begin{defn}
	As in \Cref{pointwisedefn},  the \emph{pointwise necklace realization}
	\[
		\mathfrak{C}^{\Nec}_{\Delta,\bullet}:\mathbf{PCat}(\C) \to \Cat_{\spsh}
	\] 
	is defined by the rule
	\[
		\mathfrak{C}^{\Nec}_{\Delta,\bullet}(X)_c\defeq \mathfrak{C}^{\Nec}_{\Delta}(X_c).
	\]
	The \emph{pointwise homotopy colimit realization} \(\mathfrak{C}^{\Hoc}\) is defined similarly.  
\end{defn}

\begin{defn}
	The \emph{necklace realization} functor \(\mathfrak{C}^{\Nec}:\cellset \to \Cat_{\spsh}\) is the composite
	\[
		\cellset \xrightarrow{k^\ast} \mathbf{PCat}(\C) \xrightarrow{\mathfrak{C}^{\Nec}_{\Delta,\bullet}} \Cat_{\spsh}.
	\]
	The \emph{homotopy colimit realization} functor \(\mathfrak{C}^{\Hoc}\) is defined similarly.
\end{defn}

From these definitions and \Cref{dugspivthm} above, we deduce the following useful corollary.

\begin{cor}\label{necthm}
	There is a specified natural zig-zag of weak equivalences of functors valued in \(\Cat_{\spsh}\)
	\[
		\mathfrak{C}^{\Nec} \leftarrow \mathfrak{C}^{\Hoc} \to \mathfrak{C}
	\]
\end{cor}

\subsection{Gadgets}
To prove the equivalence between quasicategories and simplicially-enriched categories, Dugger and Spivak make use of another intermediate construction called a category of Gadgets.  These are subcategories of bi-pointed simplicial sets that generalize necklaces while still retaining many of their useful properties. We begin by recalling the definition of a category of gadgets.

\begin{defn}[\cite{ds1}]
	A \emph{Category of Gadgets} is a subcategory \(\mathcal{G}\) of the category \(\psh{\Delta}_{\ast,\ast}\) satisfying the following properties:
	\begin{itemize}
		\item The category \(\mathcal{G}\) contains \(\Nec\).
		\item For all \(G\in \mathcal{G}\) and all necklaces \(T\), there is an equality 
		\[
			\mathcal{G}(T,G) = \psh{\Delta}(T,G).
		\]
		\item For any \(G \in \mathcal{G}\), the simplicial set \(\mathfrak{C}(G)(\alpha,\omega)\) is contractible.
	\end{itemize}
	The category \(\mathcal{G}\) is moreover said to be \emph{closed under wedges} if
	\begin{itemize}
		\item For any \(G, G^\prime\) in \(\mathcal{G}\), the wedge \(G\vee G^\prime\) is as well.
	\end{itemize}
	For any bi-pointed simplicial set \(X_{x,x^\prime}\) and any category of gadgets \(\mathcal{G}\), we functorially define a simplicial set 
	\[
		\mathfrak{C}^{\mathcal{G}}_\Delta(X)(x,x^\prime)\defeq N\overcat{\mathcal{G}}{X_{x,x^\prime}}.
	\]
	Moreover, if \(\mathcal{G}\) is closed under wedges, the collection of simplicial sets 
	\[
		(\mathfrak{C}^{\mathcal{G}}_\Delta(X)(x,x^\prime))_{x,x^\prime \in X}
	\]
	assembles to a simplicially enriched category \(\mathfrak{C}^{\mathcal{G}}_\Delta(X)\) with composition induced by concatenation of gadgets.
\end{defn}

Dugger and Spivak then prove the following useful proposition:

\begin{prop}[\cite{ds1}*{Proposition 5.5}]
	For any category of gadgets \(\mathcal{G}\), the natural map
	\[\mathfrak{C}^{\Nec}_\Delta(X)(x,x^\prime)\to \mathfrak{C}^{\mathcal{G}}_\Delta(X)(x,x^\prime)\]
	induced by the inclusion \(\Nec \hookrightarrow \mathcal{G}\) is a weak homotopy equivalence.  Moreover, if \(\mathcal{G}\) is closed under wedges, the natural transformation of functors valued in simplicially-enriched categories
	\[\mathfrak{C}^{\Nec}_\Delta\to \mathfrak{C}^{\mathcal{G}}_\Delta\]
	is a weak equivalence.
\end{prop}

As in the previous section, given a category \(\mathcal{G}\) of gadgets, we can extend the realization functor to a functor
\[
	\mathfrak{C}^\mathcal{G}:\cellset_{\ast,\ast} \to \spsh,
\]
and when \(\mathcal{G}\) is closed under wedges these specify a functor
\[
	\mathfrak{C}^\mathcal{G}:\cellset \to \Cat_{\spsh},
\]
from which we obtain the following corollary:
\begin{cor}\label{gadgetlemma}
	The inclusion \(\Nec\hookrightarrow \mathcal{G}\) induces a natural equivalence of functors valued in \(\spsh\)
	\[\mathfrak{C}^{\Nec}(X)(x,x^\prime)\to \mathfrak{C}^{\mathcal{G}}(X)(x,x^\prime),\]
	and when \(\mathcal{G}\) is closed under wedges, these assemble to a natural equivalence of functors valued in \(\Cat_{\spsh}\):
	\[\mathfrak{C}^{\Nec}\to \mathfrak{C}^{\mathcal{G}}.\]
\end{cor}

\subsection{Quillen functoriality}
In this section, we show that the adjunction \[\cellset_{\mathrm{hJoyal}} \underset{\mathfrak{N}}{\overset{\mathfrak{C}}{\rightleftarrows}} \Cat_{\spsh_{\mathrm{inj}}}\] is a Quillen pair. We will extensively use the characterization of \(\mathfrak{C}\) given in \Cref{pointwise}.  We begin with the following observation:

\begin{prop}
	For any \(n>0\), let \(K\subseteq \{1,\dots,n-1\}\) and define
	\[\Lambda^n_K=\bigcup_{i\notin K} \partial_i \Delta^n,\]
	and let
	\[\lambda^n_K:\Lambda^n_K\hookrightarrow \Delta^n\]
	denote the inclusion map.  Then
	\[\mathfrak{C}(\square^\lrcorner_n(\lambda^n_K,\delta^{c_1},\dots,\delta^{c_n}))(i,j)\]
	is an isomorphism whenever \(i\neq 0\) or \(j\neq n\).  Moreover, the map
	\[\mathfrak{C}(\square^\lrcorner_n(\lambda^n_K,\delta^{c_1},\dots,\delta^{c_n}))(0,n)\]
	is exactly
	\[\delta^{c_1}\times^\lrcorner h^1_K \times^\lrcorner \dots \times^\lrcorner h^{n-1}_K \times^\lrcorner \delta^{c_n},\]
	where
	\[
		h^k_K =
		\begin{cases}
			\lambda^1_1 \text{ if } k\in K \\
			\delta^1 \text{ otherwise}
		\end{cases}.
	\]
\end{prop}
\begin{proof}
	Let \(X\) denote the domain of \(\square^\lrcorner_n(\lambda^n_K,\delta^{c_1},\dots,\delta^{c_n})\).  If \(f:T\to [n](c_1,\dots,c_n)_{i,j}\) is a bi-pointed map from a necklace \(T\), with \(i\neq 0\), then \(f\) factors through the inclusion of the subobject \([n-1](c_2,\dots,c_n)\subseteq V_{\Lambda^n_K}(c_1,\dots,c_n)\), so \(\mathfrak{C}(X)(i,j)=\mathfrak{C}([n](c_1,\dots,c_n)\).  The case where \(j\neq n\) follows by symmetry.

	The second part comes from the observation that when \(K=\{1,\dots,n-1\}\),
	\[\mathfrak{C}(V_{\Lambda^n_K}(c_1,\dots,c_n))(0,n)=\bigcup_{i=1}^{n-1} c_1\times \Gamma^1_i \times \dots \times \Gamma^{n-1}_i \times c_n,\]
	where
	\[
		\Gamma^\ell_i=
		\begin{cases}
			\Lambda^1_1 \text{ for } \ell=i \\
			\Delta^1 \text{ otherwise}
		\end{cases}.
	\]  To see this, notice that \(\Lambda^n_K\) is the union of the two outer faces, and attaching them along their common face gives a colimit in \(\Cat_{\spsh}\) where \(\mathfrak{C}(V_{\Lambda^n_K}(c_1,\dots,c_n)(0,n)\) is freely generated by compositions
	\[\mathfrak{C}([n-1](c_1,\dots,c_{n-1}))(0,\ell)\times \{1\} \times \mathfrak{C}([n-1](c_2,\dots,c_{n}))(\ell,n).\]

	For when \(K\) is otherwise, each additional inner face gives the factor
	\[\mathfrak{C}([n-1](c_1,\dots,c_{n-1}))(0,\ell)\times \{0\} \times \mathfrak{C}([n-1](c_2,\dots,c_{n}))(\ell,n),\]
	so in general,
	\[\mathfrak{C}(V_{\Lambda^n_K}(c_1,\dots,c_n))(0,n)=\bigcup_{i=1}^{n-1} c_1\times \Gamma^1_{i,K} \times \dots \times \Gamma^{n-1}_{i,K} \times c_n,\]
	where
	\[
		\Gamma^\ell_{i,K}=
		\begin{cases}
			\partial\Delta^1 \text{ for } \ell=i \text{ and } i\in K \\
			\Lambda^1_1 \text{ for } \ell=i \text{ and } i\notin K   \\
			\Delta^1 \text{ otherwise}
		\end{cases}.
	\]
	Each factor
	\[V[n](c_1,\dots,\partial c_j, \dots, c_n)\]
	contributes
	\[\mathfrak{C}(V[n](c_1,\dots,\partial c_j, \dots, c_n))(0,n)=c_1\times \Delta^1\times\dots \times \Delta^1 \times \partial c_j \times \Delta^1 \times \dots \times \Delta^1 \times c_n,\]
	and taking the union of all of the factors gives exactly the domain of the inclusion
	\[\delta^{c_1}\times^\lrcorner h^1_K \times^\lrcorner \dots \times^\lrcorner h^{n-1}_K \times^\lrcorner \delta^{c_n}.\]
\end{proof}
\begin{prop}\label{quillen1}
	The functor \(\mathfrak{C}\) sends monomorphisms to cofibrations and horizontal inner anodynes to trivial cofibrations.
\end{prop}
\begin{proof}
	Let \[\mathbf{2}:\spsh \to \Cat_{\spsh}\] be the functor sending a simplicial presheaf \(X\) to the enriched category with objects \(\{0,1\}\) with \(\mathbf{2}(X)(0,0)=\mathbf{2}(X)(1,1)=\ast\), \(\mathbf{2}(X)(1,0)=\varnothing\), and \(\mathbf{2}(X)(0,1)=X\).

	When \(K=\varnothing\), \(\lambda^n_K=\delta^n\), so the lemma tells us that \[\mathfrak{C}(\square^\lrcorner_n(\delta^n,\delta^{c_1},\dots,\delta^{c_n})\] is a pushout of the map \[\mathbf{2}(\delta^{c_1}\times^\lrcorner \delta^1 \times^\lrcorner \dots \times^\lrcorner \delta^1 \times^\lrcorner \delta^{c_n}),\] which is a cofibration, which proves the claim.

	Similarly, when \(K\) is a singleton, \(\lambda^n_K=\lambda^n_k\) is the inclusion of an inner horn, so \[\mathfrak{C}(\square^\lrcorner_n(\lambda^n_k,\delta^{c_1},\dots,\delta^{c_n})\] is the pushout of the map \[\mathbf{2}(\delta^{c_1}\times^\lrcorner h^1_k \times^\lrcorner \dots \times^\lrcorner h^{n-1}_k \times^\lrcorner \delta^{c_n}),\] where \(h^k_k=\lambda^1_1\).  This is a corner map where one factor is a trivial cofibration (because it is Kan anodyne), and therefore its image under \(\mathbf{2}\) is a trivial cofibration.  Since the pushout of a trivial cofibration is a trivial cofibration, we are done.
\end{proof}

\begin{cor}
	The coherent nerve of a fibrant \(\spsh_{\mathrm{inj}}\)-enriched category is a formal \(\C\)-quasicategory.
\end{cor}

\begin{lemma}
	The object \(\mathfrak{C}(E^n)\) is weakly contractible for all \(n\).
\end{lemma}
\begin{proof}
	We notice immediately that \(\mathfrak{C}(E^n)(i,j)_\bullet\) is a constant simplicial presheaf for all \(i,j\), so it suffices to show that \(\mathfrak{C}(E^n)(i,j)_\ast\) is contractible for all \(i,j\), but then it follows immediately from the classical case.
\end{proof}

\begin{prop}
	The coherent nerve sends fibrations between fibrant \(\spsh\)-enriched categories to fibrations for the horizontal Joyal model structure.
\end{prop}
\begin{proof}
	Given a fibration between two fibrant \(\spsh\)-enriched categories, \(p:\mathcal{D}\to \mathcal{D}^\prime\), we see immediately that the coherent nerve takes this fibration to a horizontal inner fibration between formal \(\C\)-quasicategories by \Cref{quillen1}.  To show that it is a fibration for the horizontal Joyal model structure, it suffices by \Cref{isofibrations} to show that it has the right lifting property with respect to the inclusion \(e:\Delta^0 \hookrightarrow E^1\). By \Cref{quillen1}, we see that \(\mathfrak{C}\) takes the monomorphism \(e\) to a cofibration, and by the previous lemma, we see that \(\mathfrak{C}(e)\) is a weak equivalence.  It follows that \(\mathfrak{N}(p)\) is a fibration for the horizontal Joyal model structure.
\end{proof}

\begin{cor}\label{horizquillen}
	The adjunction 
	\[\cellset_{\mathrm{hJoyal}} \underset{\mathfrak{N}}{\overset{\mathfrak{C}}{\rightleftarrows}} \Cat_{\spsh_{\mathrm{inj}}}\]
	is a Quillen pair.
\end{cor}
\begin{proof}
	If \(\mathfrak{C}\) takes cofibrations to cofibrations, and \(\mathfrak{N}\) takes fibrations between fibrant objects to fibrations between fibrant objects, then the adjunction is a Quillen pair, but this is exactly what we proved in this section.
\end{proof}

The following lemma will be useful for later.
\begin{lemma}\label{precatanodyne}
	If \(A\hookrightarrow B\) is an inner anodyne map of simplicial sets, the induced map
	\[k_!(A\times c) \hookrightarrow k_!(B\times c)\] is horizontal inner anodyne. Consequently, if \(X\) is a horizontal \(\C\)-quasicategory, each simplicial set \(k^\ast(X)_c\) is a quasicategory. 
\end{lemma}
\begin{proof}
	This follows directly by \Cref{precatproperty} applied to the maps
	\[V_{\Lambda^n_i}(c,\dots,c) \hookrightarrow V_{\Delta^n}(c,\dots,c),\] as these maps are exactly the maps
	\[k_!(\Lambda^n_i\times c) \hookrightarrow k_!(\Delta^n \times c).\]
\end{proof}

\subsection{The \(\Hom\) by cosimplicial resolutions for \(\C\)-precategories}
In this section, we show that if \(X\) is a \(\C\)-precategory such that \(X_c\) is a quasicategory for all \(c \in \C\), the simplicial presheaves \(\mathfrak{C}_{\Delta,\bullet}(X)(x,x^\prime)\) can be computed by resolutions.  To show this, we will make use of some helpful results in \cite{ds2}.

We define the following four cosimplicial bi-pointed simplicial sets:
\[C^\bullet_{\mathrm{cyl}}\defeq \colim \left( \partial\Delta^1 \leftarrow \Delta^\bullet \times \partial\Delta^1 \hookrightarrow \Delta^\bullet \times \Delta^1\right)\]
\[C^\bullet_{E}\defeq \colim \left( \partial\Delta^1 \leftarrow \Delta^\bullet \times \partial\Delta^1 \hookrightarrow \Delta^\bullet \times E^1\right)\]
\[C^\bullet_{R}\defeq \colim \left( \ast \leftarrow \Delta^\bullet \xrightarrow{\partial^{\bullet+1}} \Delta^{\bullet+1}\right)\]
\[C^\bullet_{L}\defeq \colim \left( \ast \leftarrow \Delta^\bullet \xrightarrow{\partial^{0}} \Delta^{\bullet+1}\right).\]

These cosimplicial bi-pointed simplicial sets fit in a natural diagram
\begin{center}
	\begin{tikzpicture}
		\matrix (b) [matrix of math nodes, row sep=3em, column sep=3em]
		{
			C^\bullet_R &                             &                \\
			               & C^\bullet_{\mathrm{cyl}} & C^\bullet_E \\
			C^\bullet_L &                             &                \\
		};
		\path[->]
		(b-1-1) edge (b-2-2)
		(b-3-1) edge (b-2-2)
		(b-2-2) edge (b-2-3);
	\end{tikzpicture}.
\end{center}
As these are cosimplicial objects in a cocomplete category, they induce adjunctions admit right adjoints
\[
	\operatorname{Map}^{(-)}: \psh{\Delta}_{\ast,\ast} \to \psh{\Delta},
\]
where
\[
	\operatorname{Map}^{(-)}_X(x,x^\prime)_n\defeq \Hom(C^n_{(-)},X_{x,x^\prime})
\]
\begin{lemma}[\cite{ds2}*{Proposition 9.4}]
	Each of the cosimplicial bi-pointed simplicial sets above is a Reedy-cofibrant cosimplicial resolution of \(\Delta^1\) with respect to the bi-pointed Joyal model structure. 
\end{lemma}
The following corollary is immediate from the fact that quasicategories are the fibrant objects of the Joyal model structure together with the general theory of cosimplicial resolutions:
\begin{cor}
	If \(X_{x,x^\prime}\) is a bi-pointed quasicategory, then there is a natural isomorphism in the homotopy category
	\[
		\operatorname{Map}^{(-)}_X(x,x^\prime)\cong \operatorname{hMap}_X(x,x^\prime),
	\]
	where the \((-)\) on the lefthand side can take any of the values \(\mathrm{cyl},\) \(E\), \(R\) or \(L\), and where the righthand side denotes the homotopy function complex of maps of bi-pointed simplicial sets
	\[
		\Delta^1 \to X_{x,x^\prime}.
	\]
\end{cor}
Since the choice of resolution doesn't matter, we abuse notation and let \(\Map\) and \(C^\bullet\) denote whichever choice of resolution and adjoint that is convenient.
The following theorem is a key result in \cite{ds2}.  
\begin{thm}[\cite{ds2}*{Corollary 5.3}]
	There is a zig-zag, natural in \(X_{x,x^\prime}\)
	\[
		\mathfrak{C}_\Delta(X)(x,x^\prime) \leftrightsquigarrow \Map_X(x,x^\prime),
	\]
	which becomes a zig-zag of natural weak homotopy equivalences upon restriction to bi-pointed quasi-categories.
\end{thm}
From the naturality of this result, we obtain the following corollary:
\begin{cor}
	If \(X_{x,x^\prime}\) is a bi-pointed \(\C\)-precategory such that each \(X_c\) is a quasicategory, the component of the natural zig-zag of maps of simplicial presheaves on \(\C\) at \(X_{x,x^\prime}\) is a zig-zag of weak equivalences in \(\spsh\):
	\[
		\mathfrak{C}_{\Delta,\bullet}(X)(x,x^\prime)_c \leftrightsquigarrow \Map_{X_c}(x,x^\prime)
	\]
\end{cor}
Combining this corollary with \Cref{precatanodyne}, we obtain the following:
\begin{cor}
	Upon restriction to bi-pointed formal \(\C\)-quasicategories, we have a natural zig-zag of weak equivalences in \(\spsh\)
	\[
		\mathfrak{C}(X)(x,x^\prime)_c \leftrightsquigarrow \Map_{k^*(X)_c}(x,x^\prime).
	\]
\end{cor}
Unwinding the definitions, we note the following corollary:
\begin{cor}\label{mapspacecor}
	If \(X_{x,x^\prime}\) is a bi-pointed formal \(\C\)-quasicategory, we have zig-zags of weak equivalences in \(\spsh\)
	\begin{align*}
		\mathfrak{C}(X)(x,x^\prime)_c \leftrightsquigarrow \Hom(k_!(C^\bullet_R\times c), X_{x,x^\prime}),
		\intertext{and}
		\mathfrak{C}(X)(x,x^\prime)_c \leftrightsquigarrow \Hom(k_!(C^\bullet_L\times c), X_{x,x^\prime}).
	\end{align*}
\end{cor}
We also make note of one more useful fact:
\begin{lemma}\label{resolutionlemma}
	The cosimplicial bi-pointed \(\C\)-cellular sets \(k_!(C^\bullet_R\times c)\) and \(k_!(C^\bullet_L\times c)\) are Reedy-cofibrant cosimplicial resolutions of \([1](c)=k_!(\Delta^1\times c)\) in \(\cellset_{\ast,\ast}\).  
\end{lemma}
\begin{proof}
	From \cite{ds2}*{Proposition 9.4}, we see that the maps
	\[C^n_R\to \Delta^1\]
	and
	\[C^n_L\to \Delta^1\]
	are retracts of the inner anodynes inclusions 
	\[\Delta^1\hookrightarrow C^n_R\]
	and
	\[\Delta^1\hookrightarrow C^n_L\] respectively.  It then follows immediately from \Cref{precatanodyne} that the maps
	\[k_!(C^n_R\times c) \to k_!(\Delta^1\times c)=[1](c)\]
	and 
	\[k_!(C^n_L\times c) \to k_!(\Delta^1\times c)=[1](c)\]
	are retracts of horizontal inner anodynes and ergo equivalences.  Reedy cofibrancy is clear.
\end{proof}

\subsection{The \(\Hom\) by cosimplicial resolution}
For every object \([1](c)\) in \(\Theta[\C]\), we introduce four canonical cosimplicial resolutions, which we can use to define simplicial presheaves that represent the mapping space between two vertices of a \(\Theta[\C]\)-set.

First, we define the functor
\[(\bullet)^{\triangleright}(c),\text{ resp. }(\bullet)^{\triangleleft}(c): \Delta \to \Theta[\C],\]
which sends
\[[n]\mapsto [n+1](\ast,\dots,\ast,c), \text{ resp. } [n]\mapsto [n+1](c,\ast,\dots,\ast).\]
We see immediately that there are natural embeddings
\[([n])^{\triangleright,c} \hookrightarrow \Delta^n \times [1](c)  \hookleftarrow ([n])^{\triangleleft,c}\]
where each map embeds along the respective outer shuffle. We also have an obvious natural embedding
\[\mathscr{H}(\Delta^n)\times [1](c) \hookrightarrow E^n\times [1](c).\]

Then we define the following four cosimplicial objects:
\[C^\bullet_{\mathrm{cyl}}(c)\defeq \colim \left( V[1](\varnothing) \leftarrow \Delta^\bullet \times V[1](\varnothing) \hookrightarrow \Delta^\bullet\times [1](c)\right)\]
\[C^\bullet_{E}(c)\defeq \colim \left( V[1](\varnothing) \leftarrow E^\bullet \times V[1](\varnothing) \hookrightarrow E^\bullet\times [1](c)\right)\]
\[C^\bullet_{R}(c)\defeq \colim \left( \ast \leftarrow \Delta^\bullet \hookrightarrow (\bullet)^\triangleright(c)\right)\]
\[C^\bullet_{L}(c)\defeq \colim \left( \ast \leftarrow \Delta^\bullet \hookrightarrow (\bullet)^\triangleleft(c)\right).\]

These cosimplicial objects fit in a natural diagram
\begin{center}
	\begin{tikzpicture}
		\matrix (b) [matrix of math nodes, row sep=3em, column sep=3em]
		{
			C^\bullet_R(c) &                             &                \\
			               & C^\bullet_{\mathrm{cyl}}(c) & C^\bullet_E(c) \\
			C^\bullet_L(c) &                             &                \\
		};
		\path[->]
		(b-1-1) edge (b-2-2)
		(b-3-1) edge (b-2-2)
		(b-2-2) edge (b-2-3);
	\end{tikzpicture}.
\end{center}
induced by the inclusions we described above.
\begin{prop}
	Each of the cosimplicial objects described above is a Reedy-cofibrant object of bipointed presheaves on \(\Theta[\C]\), and each is objectwise horizontal-Joyal equivalent to the constant cosimplicial object \([1](c)\).  That is to say, each of these cosimplicial objects is a cosimplicial resolution of \([1](c)\) in \(\cellset_{\ast,\ast}\).
\end{prop}
\begin{proof}
	That they are Reedy-cofibrant is obvious, and it is also obvious that \(C^\bullet_E(c)\) is objectwise horizontal-Joyal equivalent to \([1](c)\). 	To treat the case of \(C^\bullet_L(c)\), we follow the proof of \cite{ds2}*{Lemma 9.3}. We begin by defining a filtration on \(C^n_L(c)\) as follows: 
	\begin{align*}
		X_0 &= [0,1](c), \\
		X_1 &= \bigcup_{1< i <\leq n} [0,1,i](c,\ast), \\
		X_2 &= \bigcup_{1< i < j \leq n} [0,1,i,j](c,\ast,\ast),  
	\end{align*}
	and so on.  We see that \(X_{n-1}=C^n_L(c)\), so it suffices to show that each inclusion \(X_{k-1} \subseteq X_{k}\) is obtained by horizontal inner anodyne attachments.  Notice first of all that 
	\[h^\ell_1:\square_\ell(\Lambda^\ell_1,c,\ast,\dots,\ast)\hookrightarrow \square_\ell(\Delta^\ell,c,\ast,\dots,\ast)\] is horizontal inner anodyne, since it is exactly \[\square^\lrcorner_k(\lambda^\ell_1,e^{c},e^\ast,\dots,e^\ast),\] where \(e^x\) denotes the empty map \(\varnothing \to x\).
	Then choose \(1< i_1 < \dots < i_k \leq n\) in \(X_k\).  Then \(X_{k-1}\) contains all of the faces \([0,1,i_1,\dots, \hat{i}_{k_j},\dots,i_k](c,\ast,\dots,\ast)\), but it also contains the degenerate face \([1,i_1,\dots,i_k](\ast,\dots,\ast)\), so therefore the map \(X_{k-1} \subseteq X_k\) is obtained by pushing out along copies of \(h^{k+1}_1\), and therefore it is horizontal inner anodyne.  Therefore, the composite of the filtration \([1](c)=X_0 \subseteq \dots \subseteq X_{n-1}\) is horizontal inner anodyne, and by \(3\)-for-\(2\), we see that the retraction \(C^n_L(c) \to [1](c)\) is a horizontal Joyal equivalence.  The proof for \(C^\bullet_R\) follows by symmetry. 
	The proof for \(C^\bullet_{\mathrm{cyl}}(c)\) is similar and can be obtained by applying the scheme for \(C^\bullet_L(c)\) to the proof of that case in \cite{ds2}*{Proposition 9.4}. We omit it because we do not make use of that particular resolution.
\end{proof}

\begin{defn}
	Given a bipointed \(\C\)-cellular set \(X_{x,y}\) we define the \emph{mapping object} from \(x\) to \(y\) to the simplicial presheaf obtained by taking homotopy function complexes
	\[\Map_X(x,y)_c\defeq h\cellset_{\ast,\ast}([1](c), X).\]
\end{defn}
\begin{rem}
	It is a general fact of abstract homotopy theory that if \(X\) is a fibrant object, we can compute the homotopy function complex with any Reedy-cofibrant cosimplicial resolution of \([1](c)\).  Any two Reedy-cofibrant cosimplicial resolutions are related by a natural zig-zag of weak equivalences, so any one will do.
\end{rem}

\begin{prop}\label{mapspacecomparison}
  For any formal \(\C\)-quasicategory \(X\) and any pair of vertices \(x,y\), there is a natural zig-zag of weak equivalences between \(\Map_X(x,y)\) and \(\mathfrak{C}(X)(x,y)\).
\end{prop}
\begin{proof}
	By \Cref{resolutionlemma}, we see that we can equivalently compute \(\Map_X(x,y)\) using the cosimplicial resolutions \(k_!(C^\bullet_R\times c)\) of \([1](c)\), from which the result follows by \Cref{mapspacecor}.
\end{proof}

\begin{lemma}\label{dssquare1}
	For any bi-pointed formal \(\C\)-quasicategory \(X_{x,y}\), the functor
	\[\overcat{\Delta}{\cellset_{\ast,\ast}(C^\bullet_R(c), X_{x,y})} \to \overcat{\Delta}{\cellset_{\ast,\ast}(k_!(C^\bullet_R\times c), X_{x,y})}\]
	induced by the map of cosimplicial resolutions of \([1](c)\)
	\[k_!(C^\bullet_R\times c) \to C^\bullet_R(c)\]
	induces a weak homotopy equivalence on nerves.  Moreover, as we have a natural isomorphism
	\[\overcat{\Delta}{\cellset_{\ast,\ast}(k_!(C^\bullet_R\times c), X_{x,y})}\cong \overcat{\Delta}{\psh{\Delta}_{\ast,\ast}(C^\bullet_R, k^\ast(X_{x,y})_c)},\]
	we have a commutative diagram in which the specified maps are weak homotopy equivalences
	\begin{center}
		\begin{tikzpicture}
			\matrix (b) [matrix of math nodes, row sep=3em,
				column sep=2em]
			{
				\coliml\limits_{[n],C^n_R\to k^\ast(X_{x,y})_c}\!\mathfrak{C}_\Delta(C^n_R)(\alpha,\omega)     & \coliml\limits_{[n],C^n_R(c)\to X_{x,y}}\!\mathfrak{C}_\Delta(C^n_R)(\alpha,\omega)                   \\
				\hocoliml\limits_{[n],C^n_R\to k^\ast(X_{x,y})_c}\!\mathfrak{C}_\Delta(C^n_R)(\alpha,\omega)     & \hocoliml\limits_{[n],C^n_R(c)\to X_{x,y}}\!\mathfrak{C}_\Delta(C^n_R)(\alpha,\omega)                   \\
				\hocoliml\limits_{[n],C^n_R\to k^\ast(X_{x,y})_c} \ast     & \hocoliml\limits_{[n],C^n_R(c)\to X_{x,y}} \ast   \\};
			\path[->]
			(b-1-2) edge (b-1-1)
			(b-3-2) edge node[auto]{\(\scriptstyle{\sim}\)} (b-3-1)
			(b-2-2) edge node[auto]{\(\scriptstyle{\sim}\)} (b-2-1) edge node[auto]{\(\scriptstyle{\sim}\)} (b-3-2) edge (b-1-2)
			(b-2-1) edge node[auto]{\(\scriptstyle{\sim}\)} (b-3-1) edge (b-1-1);
		\end{tikzpicture},
	\end{center}
\end{lemma}
\begin{proof}
	Any weak homotopy equivalence of simplicial sets induces a weak homotopy equivalence on nerves of their categories of elements.  As the map \(k_!(C^\bullet_R\times c)\to C^\bullet_R(c)\) is a map between Reedy-cofibrant resolutions of \([1](c)\), it follows that the induced map 
	\[
		\cellset_{\ast,\ast}(C^\bullet_R(c), X_{x,y}) \to \cellset_{\ast,\ast}(k_!(C^\bullet_R\times c), X_{x,y})
	\]
	is a weak homotopy equivalence whenever \(X_{x,y}\) is fibrant.  Taken together, these observations prove the first claim.  To obtain the diagram in the statement, notice that the vertical maps are equivalences as each of the simplicial sets \(\mathfrak{C}_\Delta(C^n_R)(\alpha,\omega)\) is weakly contractible and that the bottom horizontal map is an equivalence, since each of these homotopy colimits is naturally weakly equivalent to the nerve of the diagram category.  The fact that the middle horizontal map is a weak homotopy equivalence follows by \(3\)-for-\(2\).
\end{proof}
\begin{note}
	In what follows, we will make use of a special category of gadgets denoted by \(\mathcal{Y}\).  It is the full subcategory of \(\psh{\Delta}_{\ast,\ast}\) spanned by those bi-pointed simplicial sets \(Y_{\alpha,\omega}\) such that \(\mathfrak{C}_\Delta(Y)(\alpha,\omega)\) is contractible.  The properties of this category of gadgets are spelled out in \cite{ds2}*{Section 5}.
\end{note}
\begin{lemma}\label{dssquare2}
	Given a \(\spsh\)-enriched category \(\mathcal{D}\) and any pair of objects \(x,y\in \mathcal{D}\), there is a natural commutative diagram
	\begin{center}
		\begin{tikzpicture}
			\matrix (b) [matrix of math nodes, row sep=3em,
				column sep=2em]
			{
				\mathfrak{C}(\mathfrak{N}\mathcal{D})(x,y)_c     & \mathcal{D}(x,y)_c                  \\
				\coliml\limits_{T\in \overcat{\Nec}{k^\ast(\mathfrak{N}\mathcal{D}_{x,y})_c}}\!\mathfrak{C}_\Delta(T)(\alpha,\omega) & \coliml\limits_{Y\in \overcat{\mathcal{Y}}{k^\ast(\mathfrak{N}\mathcal{D}_{x,y})_c}}\!\mathfrak{C}_\Delta(Y)(\alpha,\omega)  \\};
			\path[->]
			(b-1-1) edge (b-1-2)
			(b-2-1) edge node[auto]{\(\scriptstyle{\cong}\)} (b-1-1) edge  (b-2-2)
			(b-2-2) edge (b-1-2);
		\end{tikzpicture},
	\end{center}
	in which the specified map is an isomorphism.
\end{lemma}
\begin{proof}
Recall that \(\mathfrak{N}=k_\ast \circ \mathfrak{N}_{\Delta,\bullet}\), so given a map \[f:Y\to k^\ast(\mathfrak{N}\mathcal{D}_{x,y})_c,\] we obtain the map \(\mathfrak{C}(Y)\to \mathcal{D}_c\) as the transpose of the composite of \(f\) with the whiskered counit of the adjunction \(k^\ast k_\ast \to \id\).   In particular, for each such map, we obtain a universal factorization \[\mathfrak{C}_\Delta(Y) \to \mathfrak{C}\mathfrak{N}\mathcal{D}_c\to \mathcal{D}_c.\] This determines the righthand vertical map after taking colimits.
Upon restriction to \[\overcat{\Nec}{k^\ast(\mathfrak{N}\mathcal{D}_{x,y})_c},\] we see that this factorization gives rise to the lefthand vertical map, which is an isomorphism by \cite{ds1}*{Proposition 4.3}.  
\end{proof}
\begin{thm}\label{counitthm}
  For any fibrant \(\spsh_{\mathrm{inj}}\)-enriched category \(\mathcal{D}\), the counit map
	\[\epsilon_\mathcal{D}:\mathfrak{C}(\mathfrak{N}\mathcal{D})\to \mathcal{D}\]
	is a weak equivalence of \(\spsh_{\mathrm{inj}}\)-enriched categories.
\end{thm}
\begin{proof}
	Let \(C^\bullet=C^\bullet_R\).  Then consider the following commutative diagram:
	\begin{center}
		\begin{tikzpicture}
			\matrix (b) [matrix of math nodes, row sep=3em,
				column sep=2em]
			{
				\mathfrak{C}(\mathfrak{N}\mathcal{D})(x,y)_c                                                                & \mathcal{D}(x,y)_c                                                                &                                                                                            \\
				\coliml\limits_{T\to k^\ast(\mathfrak{N}\mathcal{D})_c}\! \mathfrak{C}_\Delta(T)(\alpha,\omega)   & \coliml\limits_{Y\to k^\ast(\mathfrak{N}\mathcal{D})_c}\!\mathfrak{C}_\Delta(Y)(\alpha,\omega)   & \coliml\limits_{[n],C^n(c)\to \mathfrak{N}\mathcal{D}}\!\mathfrak{C}_\Delta(C^n)(\alpha,\omega)   \\
				\hocoliml\limits_{T\to k^\ast(\mathfrak{N}\mathcal{D})_c}\!\mathfrak{C}_\Delta(T)(\alpha,\omega) & \hocoliml\limits_{Y\to k^\ast(\mathfrak{N}\mathcal{D})_c}\!\mathfrak{C}_\Delta(Y)(\alpha,\omega) & \hocoliml\limits_{[n],C^n(c)\to \mathfrak{N}\mathcal{D}}\!\mathfrak{C}_\Delta(C^n)(\alpha,\omega) \\
				\hocoliml\limits_{T\to k^\ast(\mathfrak{N}\mathcal{D})_c}\!\ast &
				\hocoliml\limits_{Y\to k^\ast(\mathfrak{N}\mathcal{D})_c}\!\ast &
				\hocoliml\limits_{[n],C^n(c)\to \mathfrak{N}\mathcal{D}}\!\ast   \\
			};
			\path[->]
			(b-1-1) edge (b-1-2)
			(b-2-1) edge node[auto]{\(\scriptstyle{\cong}\)} (b-1-1) edge (b-2-2)
			(b-2-2) edge (b-1-2)
			(b-2-3) edge (b-2-2) edge (b-1-2)
			(b-3-1) edge node[auto]{\(\scriptstyle{\sim}\)} (b-2-1) edge (b-3-2) edge node[auto]{\(\scriptstyle{\sim}\)} (b-4-1)
			(b-3-2) edge (b-2-2) edge node[auto]{\(\scriptstyle{\sim}\)} (b-4-2)
			(b-3-3) edge (b-2-3) edge (b-3-2) edge node[auto]{\(\scriptstyle{\sim}\)} (b-4-3)
			(b-4-1) edge node[auto]{\(\scriptstyle{\sim}\)} (b-4-2)
			(b-4-3) edge node[auto,swap]{\(\scriptstyle{\sim}\)} (b-4-2);
		\end{tikzpicture}.
	\end{center}
	This diagram is obtained by composing the diagram from \cite{ds2}*{Proposition 5.2} reproduced below with the diagrams appearing in \Cref{dssquare1} and \Cref{dssquare2}
	\begin{center}
		\begin{tikzpicture}
			\matrix (b) [matrix of math nodes, row sep=3em,
				column sep=2em]
			{
				\coliml\limits_{T\to k^\ast(\mathfrak{N}\mathcal{D})_c}\! \mathfrak{C}_\Delta(T)(\alpha,\omega)   & \coliml\limits_{Y\to k^\ast(\mathfrak{N}\mathcal{D})_c}\!\mathfrak{C}_\Delta(Y)(\alpha,\omega)   & \coliml\limits_{[n],C^n\to k^\ast(\mathfrak{N}\mathcal{D})_c}\!\mathfrak{C}_\Delta(C^n)(\alpha,\omega)   \\
				\hocoliml\limits_{T\to k^\ast(\mathfrak{N}\mathcal{D})_c}\!\mathfrak{C}_\Delta(T)(\alpha,\omega) & \hocoliml\limits_{Y\to k^\ast(\mathfrak{N}\mathcal{D})_c}\!\mathfrak{C}_\Delta(Y)(\alpha,\omega) & \hocoliml\limits_{[n],C^n\to k^\ast(\mathfrak{N}\mathcal{D})_c}\!\mathfrak{C}_\Delta(C^n)(\alpha,\omega) \\
				\hocoliml\limits_{T\to k^\ast(\mathfrak{N}\mathcal{D})_c}\!\ast &
				\hocoliml\limits_{Y\to k^\ast(\mathfrak{N}\mathcal{D})_c}\!\ast &
				\hocoliml\limits_{[n],C^n\to k^\ast(\mathfrak{N}\mathcal{D})_c}\!\ast   \\
			};
			\path[->]
			(b-1-1) edge (b-1-2)
			(b-1-3) edge (b-1-2)
			(b-2-1) edge node[auto]{\(\scriptstyle{\sim}\)} (b-1-1) edge (b-2-2) edge node[auto]{\(\scriptstyle{\sim}\)} (b-3-1)
			(b-2-2) edge (b-1-2) edge node[auto]{\(\scriptstyle{\sim}\)} (b-3-2)
			(b-2-3) edge (b-1-3) edge (b-2-2) edge node[auto]{\(\scriptstyle{\sim}\)} (b-3-3)
			(b-3-1) edge node[auto]{\(\scriptstyle{\sim}\)} (b-3-2)
			(b-3-3) edge node[auto,swap]{\(\scriptstyle{\sim}\)} (b-3-2);
		\end{tikzpicture}.
	\end{center}
	The colimits in the left column are indexed over \(\overcat{\Nec}{k^\ast(\mathfrak{N}\mathcal{D}_{x,y})_c},\) the colimits in the middle column are indxed by \(\overcat{\mathcal{Y}}{k^\ast(\mathfrak{N}\mathcal{D}_{x,y})_c},\) and the colimits in the righthand column are indexed by \(\overcat{\Delta}{\psh{\Delta}_{\ast,\ast}(C^\bullet_R, k^\ast(X_{x,y})_c)}\).  
	The two horizontal maps in the bottom row are weak equivalences by \cite{ds2}*{Proposition 5.2}, as are the downward-oriented vertical maps, using the fact that \(k^\ast(\mathfrak{N}\mathcal{D})_c\) is a quasicategory. The indicated upward-oriented vertical map is a weak equivalence by \cite{ds1}*{Theorem 5.2}.  By \(3\)-for-\(2\), it follows that the horizontal maps in the third row are all weak equivalences.  We therefore reduce the composite diagram to a smaller diagram:
	\begin{center}
		\begin{tikzpicture}
			\matrix (b) [matrix of math nodes, row sep=3em,
				column sep=2em]
			{
				\mathfrak{C}(\mathfrak{N}\mathcal{D})(x,y)_c     & \mathcal{D}(x,y)_c     &               \\
				\hocoliml\limits_{T\to k^\ast(\mathfrak{N}\mathcal{D})_c}\!\mathfrak{C}_\Delta(T)(\alpha,\omega) & \hocoliml\limits_{Y\to k^\ast(\mathfrak{N}\mathcal{D})_c}\!\mathfrak{C}_\Delta(Y)(\alpha,\omega) & \hocoliml\limits_{[n],C^n(c)\to \mathfrak{N}\mathcal{D}}\!\mathfrak{C}_\Delta(C^n)(\alpha,\omega) \\};
			\path[->]
			(b-1-1) edge (b-1-2)
			(b-2-1) edge node[auto]{\(\scriptstyle{\sim}\)} (b-1-1) edge node[auto]{\(\scriptstyle{\sim}\)} (b-2-2)
			(b-2-2) edge (b-1-2)
			(b-2-3) edge node[auto,swap]{\(\scriptstyle{\gamma_c}\)} (b-1-2) edge node[auto,swap]{\(\scriptstyle{\sim}\)} (b-2-2);
		\end{tikzpicture},
	\end{center}
	By \(3\)-for-\(2\), we can see that it suffices to show that the map
	\[\gamma_c: \hocoliml_{[n],C^n(c)\to \mathfrak{N}\mathcal{D}} \mathfrak{C}_\Delta(C^n)(\alpha,\omega) \to \coliml_{[n],C^n(c)\to \mathfrak{N}\mathcal{D}} \mathfrak{C}_\Delta(\C^n)(\alpha,\omega)\to \mathcal{D}(x,y),c\]
	is a weak equivalence.
	To do this, notice that
	\begin{align*}
		\cellset_{\ast,\ast}(C^n(c), \mathfrak{N}\mathcal{D}_{x,y}) & \cong \overcat{\partial[1]}{\Cat_{\spsh}}(\mathfrak{C}(C^n(c))_{\alpha,\omega},\mathcal{D}_{x,y}) \\
		& \cong \spsh(\mathfrak{C}(C^n(c))(\alpha,\omega), \mathcal{D}(x,y)).
	\end{align*}

	As in the proof of \cite{ds2}*{Proposition 5.8}, we define the cosimplicial simplicial set \(Q^\bullet = \mathfrak{C}_\Delta (C^\bullet)(\alpha,\omega)\), which is obviously isomorphic to the \(Q^\bullet\) defined in \cite{ds2}. Moreover, by direct computation, we see that
	\[\mathfrak{C}(C^\bullet(c))(\alpha,\omega)\cong Q^\bullet \times c.\]
	Then we see immediately that
	\[\spsh(\mathfrak{C}(C^n(c))(\alpha,\omega), \mathcal{D}(x,y)) \cong \spsh(Q^n \times c, \mathcal{D}(x,y))\cong \psh{\Delta}(Q^n, \mathcal{D}(x,y)_c),\]
	so \(\gamma_c\) is precisely the map obtained by composing
	\[\left(\hocoliml_{[n],Q^n\to \mathcal{D}(x,y)_c} Q^n\right) \to \left(\coliml_{[n],Q^n\to \mathcal{D}(x,y)_c} Q^n\right) \to \mathcal{D}(x,y)_c.\]
	The result then follows immediately by application of \cite{ds2}*{Lemma 5.9}.
\end{proof}

\begin{rem}
	This result is even stronger than it first appears, because it implies that the counit map is a weak equivalence for fibrant categories enriched in any Cartesian-closed left-Bousfield localization of \(\spsh_{inj}\).  It reduces proving comparison theorems for such localizations to showing that \(\mathfrak{C}\) is a left-Quillen functor (something we already know for the horizontal Joyal model structure by \Cref{horizquillen}) and reflects weak equivalences.
\end{rem}

\subsection{The horizontal comparison theorem}\label{horizcomparison}
Dugger and Spivak introduce a definition of a Dwyer-Kan equivalence as a stepping stone to proving the comparison theorem.  They use the definition of DK-equivalence as an intermediate step to proving that \(\mathfrak{C}_\Delta\) is homotopy-conservative.  We give an analogous definition as follows:

\begin{defn}
  A map \(f:X\to Y\) of presheaves on \(\Theta[\C]\) is called a \emph{horizontal Dwyer-Kan equivalence} if the following two properties hold:
	\begin{itemize}
		\item The induced map
		      \[f_*:\operatorname{Ho}(\cellset_{\mathrm{hJoyal}}(\ast,X) \to \operatorname{Ho}(\cellset_{\mathrm{hJoyal}}(\ast,Y)\]
		      is bijective, and
		\item For any two vertices \(x,x^\prime\in X_0\), the induced map
		      \[\Map_X(x,x^\prime) \to \Map_Y(f(x),f(x^\prime))\]
		      is a weak equivalence of simplicial presheaves on \(\C\).
	\end{itemize}
\end{defn}

\begin{prop} A map \(f:X\to Y\) of presheaves on \(\Theta[\C]\) is a horizontal weak equivalence if and only if it is a horizontal Dwyer-Kan equivalence.
\end{prop}
\begin{proof}
	It is clear that any horizontal Joyal equivalence is automatically horizontally Dwyer-Kan since our constructions are all homotopy-invariant, so we prove that all horizontal Dwyer-Kan equivalences are horizontal Joyal equivalences.  We notice immediately that if \(X\) and \(Y\) are fibrant, the horizontal Dwyer-Kan condition implies that the associated map \(\mathcal{Q}(f):\mathcal{Q}(X)\to \mathcal{Q}(Y)\) between complete \(\Theta[\C]\)-Segal spaces is an equivalence, where
	\[\mathcal{Q}:\cellset \to \psh{\Theta[\C]\times \Delta}\]
	is defined by the rule
	\[\mathcal{Q}(X)_{[n](c_1,\dots,c_n),m} \defeq\Hom([n](c_1,\dots,c_n)\times E^m,X).\]
	Since \(\mathcal{Q}\) is the right adjoint of a Quillen equivalence by \Cref{rezkcomparison}, a map \(f\) between fibrant objects is a weak equivalence if and only if its image under \(\mathcal{Q}\) is. Since Dwyer-Kan equivalences between complete \(\Theta[\C]\)-Segal spaces are exactly the weak equivalences by \cite{rezk-theta-n-spaces}, the claim holds for \(X\) and \(Y\) fibrant.

	In general, given a horizontal Dwyer-Kan equivalence \(f:X\to Y\) where \(X\) and \(Y\) are no longer assumed to be fibrant, we can take a fibrant replacement \(\tilde{Y}\) of \(Y\) such that \(Y\to Y^\prime\) is a trivial cofibration for the horizontal Joyal model structure.  Then we can also factor \(X\to Y\to \tilde{Y}\) into a trivial Joyal cofibration \(X\to \tilde{X}\) followed by a fibration \(\tilde{X}\to \tilde{Y}\).  But notice now that the condition of being horizontally DK-equivalent is homotopy invariant, so the map \(\tilde{X}\to \tilde{Y}\) is also a horizontal DK-equivalence.  Since \(\tilde{Y}\) is fibrant and \(\tilde{X}\to \tilde{Y}\) is a horizontal Joyal fibration, this is a horizontal Joyal equivalence.  Then by \(3\)-for-\(2\)  we see that \(f\) is also a horizontal Joyal equivalence, which concludes the proof.
\end{proof}

\begin{prop}\label{conservativity}
	A map \(f:X\to Y\) of presheaves on \(\Theta[\C]\) is a horizontal Joyal equivalence if and only if \(\mathfrak{C}(f)\) is a weak equivalence of \(\Psh_\Delta (\C)_{\mathrm{inj}}\)-enriched categories.
\end{prop}
\begin{proof}
	We only need to check one direction, since the other direction is immediate by the fact that \(\mathfrak{C}\) is left-Quillen.  Assume \(f:X\to Y\) has the property that \(\mathfrak{C}(f)\) is an equivalence. Then as in the previous proposition, we can reduce to the case where \(X\) and \(Y\) are fibrant, but in this case, we know from \Cref{mapspacecomparison} that \(\mathfrak{C}(X)(x,x^\prime)\) is connected by a natural zig-zag of weak equivalences to \(\Map_X(x,y)\), so if the map  \(\mathfrak{C}(X)(x,x^\prime)\to \mathfrak{C}(Y)(f(x),f(x^\prime))\) is a weak equivalence, it follows that the map \(\Map_X(x,x^\prime)\to \Map_Y(f(x),f(x^\prime))\) is also a weak equivalence.

	Then it suffices to show that when \(\mathfrak{C}(f)\) is an equivalence, the induced map on sets of homotopy classes
	\[[\ast,X]_{E^1} \to [\ast,Y]_{E^1}\]
	is a bijection.  Notice that
	\[[\ast,X]_{E^1} \cong \pi_0 \cellset(E^\bullet,X),\]
	and since \(E^n=\mathscr{H}\operatorname{cosk}_0 \Delta^n\). By abuse of notation, we also denote the simplicial set \(\operatorname{cosk}_0 \Delta^n\) by \(E^n\).  We noticed earlier that \(\mathscr{H}\) has a right adjoint, which we now denote by \(\mathscr{N}\). Using this, we can rewrite the question as asking for the induced map to give a bijection
	\[\pi_0 \psh{\Delta}(E^n,\mathscr{N}X) \to \pi_0 \psh{\Delta}(E^n,\mathscr{N}X),\]
	which is the same as giving a bijection
	\[[\Delta^0,\mathscr{N}X]_{E^1}\to [\Delta^0,\mathscr{N}X]_{E^1}\].

	Notice also that the data classifying an equivalence in \(\mathfrak{C}(X)\) all factor through the simplicial category \(\mathfrak{C}(X)_{\ast_\C}\) obtained by evaluating each of the \(\Hom\) objects at the terminal object \(\ast_{\C}\) of \(\C\).  It is an easy exercise to see that
	\[\mathfrak{C}(X)_{\ast_\C} \cong \mathfrak{C}_\Delta (\mathscr{N}X),\]
	but since \(\mathscr{N}X\) is quite clearly a quasicategory, the claim follows immediately from the ordinary case.  This implies that \(f\) is a horizontal Dwyer-Kan equivalence, and therefore by the previous proposition, a horizontal Joyal equivalence, which concludes the proof.
\end{proof}

\begin{thm}\label{mainthm1}
	The Quillen pair
	\[\cellset_{\mathrm{hJoyal}} \underset{\mathfrak{N}}{\overset{\mathfrak{C}}{\rightleftarrows}} \Cat_{\spsh_{\mathrm{inj}}}\]
	is a Quillen equivalence.
\end{thm}
\begin{proof}
	As we have proven that the derived counit is always an equivalence in \Cref{counitthm}, all we have left to show is that the derived unit transformation
	\[X\to \mathfrak{N}\mathfrak{C}(X) \to \mathfrak{N}\mathcal{D}\]
	is a weak equivalence for all presheaves \(X\) on \(\Theta[\C]\), where \(\mathfrak{C}(X) \to \mathcal{D}\) is a weak equivalence and \(\mathcal{D}\) is fibrant.  However, by the previous proposition, we see that it suffices to show that the map
	\[\mathfrak{C}(X)\to \mathfrak{C}\mathfrak{N}\mathfrak{C}(X) \to \mathfrak{C}\mathfrak{N}\mathcal{D}\]
	is a weak equivalence.  We obtain a naturality square from the counit
	\begin{center}
		\begin{tikzpicture}
			\matrix (b) [matrix of math nodes, row sep=3em, column sep=2em]
			{
				\mathfrak{C}\mathfrak{N}\mathfrak{C}(X) & \mathfrak{C}\mathfrak{N}\mathcal{D} \\
				\mathfrak{C}(X)                                               & \mathcal{D}                                               \\
			};
			\path[->]
			(b-1-1) edge (b-1-2) edge (b-2-1)
			(b-1-2) edge node[auto]{\(\scriptstyle{\sim}\)} (b-2-2)
			(b-2-1) edge node[auto]{\(\scriptstyle{\sim}\)} (b-2-2);
		\end{tikzpicture},
	\end{center}
	in which the indicated arrows are equivalences (for the bottom horizontal, this was by choice, and for the righthand vertical, it comes from \Cref{counitthm}).  But if we precompose with the unit map \(\mathfrak{C}\eta_X:\mathfrak{C}(X)\to \mathfrak{C}\mathfrak{N}\mathfrak{C}(X)\), the lefthand arrow becomes the identity by the triangle identities, which proves the claim by applying \(3\)-for-\(2\) to the commutative diagram
	\begin{center}
		\begin{tikzpicture}
			\matrix (b) [matrix of math nodes, row sep=3em, column sep=2em]
			{
				\mathfrak{C}(X) &\mathfrak{C}\mathfrak{N}\mathfrak{C}(X) & \mathfrak{C}\mathfrak{N}\mathcal{D} \\
				&\mathfrak{C}(X)                                               & \mathcal{D}                                               \\
			};
			\path[->]
			(b-1-1) edge (b-1-2) edge[-,double] (b-2-2)
			(b-1-2) edge (b-1-3) edge (b-2-2)
			(b-1-3) edge node[auto]{\(\scriptstyle{\sim}\)} (b-2-3)
			(b-2-2) edge node[auto]{\(\scriptstyle{\sim}\)} (b-2-3);
		\end{tikzpicture}.
	\end{center}
\end{proof}

\section{The Coherent Nerve, Local case}
\subsection{The \((\C,\setS)\)-enriched model structure}\label{rezkvert}
While our presentation of the horizontal Joyal model structure comes mainly from David Oury's thesis \cite{oury}, what follows is independent, making use of the resolution technology we developed in the previous section to give a simple and satisfying story. Suppose \(\M=(\C,\setS)\) is a Cartesian presentation in the sense of Rezk, where \(\setS\) is a set of monomorphisms of \(\spsh\) such that the left-Bousfield localization of \(\spsh_\mathrm{inj}\) at \(\setS\) is a Cartesian model category.  

Recall that we had a number of functorial cosimplicial objects \[C_{(-)}^\bullet(\bullet):\Delta \times \C \to \cellset_{\ast,\ast},\] such that \(C_{(-)}^\bullet(\bullet)\) was a cosimplicial resolution of \([1](\bullet)\), which is a Reedy cofibrant diagram \(\C\to \cellset_{\ast,\ast}\).  Since \(\cellset_{\ast,\ast}\) is cocomplete, \(C_{(-)}^\bullet(\bullet)\) extends to a cocontinuous functor \[\Sigma:\spsh\to \cellset_{\ast,\ast}.\]  

\begin{prop} The functor \(\Sigma_{(-)}\) is left-Quillen when \(\cellset_{\ast,\ast}\) is equipped with the horizontal Joyal model structure.
\end{prop}
\begin{proof} 
  It clearly preserves cofibrations, so it suffices to show that its right adjoint preserves fibrations between fibrant objects.  However, this is clear, since the right adjoint sends a bi-pointed formal \(\C\) quasicategory \(X_{x,y}\) to \(\Map_X(x,y)\), which we saw sends horizontal Joyal fibrations to injective fibrations of simplicial presheaves on \(\C\).
\end{proof}
\begin{cor}
  The functor \(\Sigma_{(-)}\) is independent up-to-homotopy of choice of resolution \(C_{(-)}^\bullet(\bullet)\).
\end{cor}
\begin{proof}
  Since simplicial presheaves are always canonically the homotopy-colimit of their representables, and since left-Quillen functors send homotopy-colimits to homotopy-colimits, it suffices to show that \(\Sigma_{(-)}(\Delta^n\times c)\) is independent up-to-homotopy.  But this is clear since all \(C_{(-)}^\bullet(\bullet)\) are connected by natural zig-zags of natural weak equivalences, since they are all cosimplicial resolutions of the same functor \([1](\bullet)\).
\end{proof}
We can therefore, without any worry, denote \(\Sigma_{(-)}\) simply by \(\Sigma\).  Then we define the following model structure:
\begin{defn}
  If \(\M=(\C,\setS)\) is a Cartesian presentation, we define the model category \(\cellset_{\M}\) to be the left-Bousfield localization of \(\cellset_\mathrm{hJoyal}\) at the set \(\Sigma(\setS)\), where we call the fibrant objects \emph{\(\M\)-enriched quasicategories} or simply \emph{\(\M\)-quasicategories}.
\end{defn}

\begin{prop} Let \(\mathscr{B}\) denote the set of simplicial boundary inclusions.  Then a formal \(\C\)-quasicategory is an \(\M\)-quasicategory if and only if it has the right lifting property with respect to \[\Sigma(\mathscr{B}\times^\lrcorner \setS).\]
\end{prop}
\begin{proof}
	Let \(E:\psh{\Delta} \to \cellset\) be the left Kan extension of the cosimplicial object \(E^\bullet\).  We know by the construction of the left-Bousfield localization of Cisinski model categories that a formal \(\C\)-quasicategory \(X\) is \(\Sigma(\setS)\)-local if and only if it has the right-lifting property with respect to \(E(\mathscr{B})\times^\lrcorner \Sigma(\setS)\), but this occurs only when for every \(s:A\to B\) in \(\setS\), the map \(X^{\Sigma(f)}:X^{\Sigma(B)} \to X^{\Sigma(A)}\) has the right lifting property with respect to \(E(\mathscr{B})\).  
	
	By we claim that by adjunction, this happens if and only if \(\mathscr{G}(X^{\Sigma(f)})\) is a trivial fibration, where \(\mathscr{G}\) is the functor sending a formal \(\C\)-quasicategory \(X\) to the Kan core of the underlying quasicategory \(\mathscr{N}(X)\).  To see this, notice that since all of the maps \(f \in \setS\) are monic, and since \(\Sigma\) preserves cofibrations, it follows that the maps \(X^{\Sigma(f)}\) are all horizontal Joyal fibrations because the horizontal Joyal model structure is Cartesian.  Then, since \(\mathscr{N}\) sends horizontal Joyal fibrations between formal \(\C\)-quasicategories to Joyal fibrations of quasicategories, and since the Kan core of a Joyal fibration between quasicategories is a Kan fibration of Kan complexes, it suffices to show that the Kan fibration \(\mathscr{G}(X^f)\) is a trivial fibration as we claimed.
  
	Notice that \(\mathscr{G}(X^{\bullet})\) is exactly the simplicial mapping space \(\operatorname{hMap}^\Delta(\bullet,X)\).  For any two vertices \(x,y\) of \(X\), it suffices therefore to show that \(\operatorname{hMap}^\Delta_{\ast,\ast}(\Sigma(f),X_{x,y})\) is a trivial fibration, since a Kan fibration is a trivial fibration if and only if it has contractible fibres.  But \(\operatorname{hMap}_{\ast,\ast}(\Sigma(f),X_{x,y})\) is exactly \(\operatorname{hMap}^\Delta(f,\Map_X(x,y))\), which is a trivial fibration if and only if \(\Map_X(x,y)\) has the right lifting property with respect to \(b \times^\lrcorner f\) where \(b\) is a simplicial boundary inclusion, which proves the proposition.
\end{proof}
\begin{cor} A formal \(\C\)-quasicategory \(X\) is an \(\M\)-quasicategory if and only if \(\Map_X(x,y)\) is \(\setS\)-local for all pairs of vertices \(x,y\) in \(X\).  
\end{cor}

\begin{thm} For any Cartesian presentation \(\M=(\C,\setS)\), the model category \(\cellset_{\M}\) is Cartesian-closed.  
\end{thm}
\begin{proof}
  This is exactly \cite{rezk-theta-n-spaces}*{Proposition 8.5}.
\end{proof}

In what follows, let \(\Sigma\) be \(\Sigma_R\).  

\begin{prop} The pair \(\cellset_\M \underset{\mathfrak{N}}{\overset{\mathfrak{C}}{\rightleftarrows}} \Cat_{\spsh_{\setS}}\) is a Quillen pair.
\end{prop}
\begin{proof}
  It suffices to show that \(\mathfrak{N}\) preserves fibrant objects by the properties of the left-Bousfield localization.  Since the coherent nerve of any fibrant \(\spsh\)-enriched category \(\mathcal{D}\) is already a formal \(\C\)-quasicategory, it suffices to show that \(\mathfrak{N}\mathcal{D}\) has the right-lifting property with respect to \(\Sigma(\mathscr{B}\times^\lrcorner \setS)\).  This will be true so long as the maps \(\mathfrak{C}(\Sigma(\mathscr{B}\times^\lrcorner \setS))\) are all weak equivalences.  To see this, let \(\mathbf{2}(A)\) for any simplicial presheaf \(A\) on \(\C\) denote the \(\spsh\)-enriched category whose objects are \(\{0,1\}\) and where 
  \[
    \mathbf{2}(A)(x,y)=
    \begin{cases}
      \ast \text{ if } x=y\\
      A \text{ if } x<y\\
      \varnothing \text{ otherwise}
    \end{cases}. 
  \]
  For all \(n \geq 0\) and \(c \in \C\), there is a natural weak equivalence 
  \[\mathfrak{C}(\Sigma(\Delta^n \times c))(\alpha,\omega) \cong Q^n \times c \xrightarrow{\sim} \Delta^n \times c \cong \mathbf{2}(c\times \Delta^n)(0,1).\] 
  Following \cite{htt}*{Proposition 2.2.2.7}, We define a realization 
  \[\realiz{\bullet}_{Q}:\spsh\to \spsh\] 
  by left Kan extension of the functor \(\Delta^n\times c \mapsto Q^n \times c\) along the Yoneda embedding. Let \(\mathcal{A}\) denote the class of simplicial presheaves \(A\) on \(\C\) such that the map
  \[\realiz{A}_Q \to A\]
  is an injective equivalence.  This class is closed under filtered colimits, since injective weak equivalences are closed under filtered colimits, so it suffices to consider the case where \(A\) has finitely many nondegenerate representable cells \([n] \times c\).  Since \(\Delta\) and \(\C\) are regular Cartesian Reedy, so is their product by \cite{cisinski-book}*{8.2.7}, and the boundary of a representable cell is given by the formula 
  \[\partial(\Delta^n\times c)=\partial\Delta^n \times c \cup \Delta^n \times \partial c.\]
  We work by induction on Reedy dimension and number of cells.
  If \(A=\varnothing\), we are done, since the map in question is the identity.
  Otherwise, suppose
  \[A=A^\prime \coprod_{\partial(\Delta^n\times c)}  \Delta^n\times c.\]
  This is a homotopy pushout since \(\partial(\Delta^n \times c) \to \Delta^n\times c\) is an injective cofibration. Similarly, 
  \[\realiz{A}_Q=\realiz{A^\prime} \coprod_{\realiz{\partial(\Delta^n\times c)}_Q}  \realiz{\Delta^n\times c}\]
  is also a homotopy-pushout since \(\realiz{\bullet}_Q\) preserves monomorphisms.  Then we see that the map 
  \[\realiz{\Delta^n \times c}_Q=Q^n \times c \to \Delta^n\times c\] is already a weak equivalence since \(Q^n\to \Delta^n\) is a weak equivalence and the injective model structure is Cartesian.  The map 
  \[\realiz{\partial(\Delta^n \times c)}_Q \to \partial(\Delta^n \times c)\]
  is a weak equivalence by the induction hypothesis, since the Reedy dimension of \(\partial(\Delta^n \times c)\) is less than the dimension of \(\Delta^n \times c\). Finally, we see that 
  \[\realiz{A^\prime}_Q \to A^\prime\]
  is a weak equivalence since \(A^\prime\) has one fewer nondegenerate cell than \(A\) and is therefore also covered in the induction hypothesis.

  Therefore, the natural map
  \[\mathfrak{C}(\Sigma(A))\cong \mathbf{2}(\realiz{A}_Q) \xrightarrow{\sim} \mathbf{2}(A)\]
  is a weak equivalence in \(\Cat_{\Psh_{\Delta}(C)_{\mathrm{inj}}}\) for all simplicial presheaves \(A\) on \(\C\).
  From this, it follows that since \(\mathbf{2}(b\times^\lrcorner f)\) is an \(\M\)-equivalence for any \(f\in \setS\), and since we have a natural equivalence of arrows \[\mathfrak{C}(\Sigma(b\times^\lrcorner f))\xrightarrow{\sim} \mathbf{2}(b\times^\lrcorner f),\] then by \(3\)-for-\(2\), \(\mathfrak{C}(\Sigma(b\times^\lrcorner f))\) is a weak equivalence, which proves left-Quillen functoriality.
\end{proof}
\begin{thm}\label{maintheorem2}
  The Quillen pair \(\cellset_\M \underset{\mathfrak{N}}{\overset{\mathfrak{C}}{\rightleftarrows}} \Cat_{\Psh_{\Delta}(\C)_{\setS}}\) is a Quillen equivalence.
\end{thm}
\begin{proof} It suffices to show that \(\mathfrak{C}\) is homotopy-conservative, so let \(f:X\to Y\) be a map in \(\cellset\) such that \(\mathfrak{C}(f)\) is an equivalence in \(\Cat_{\spsh_{\M}}\). Using the same argument as in \Cref{horizcomparison}, we reduce to the case where \(f:X\to Y\) is a map between \(\M\)-quasicategories.  Since \(\M\)-quasicategories are also formal \(\C\)-quasicategories, we can apply \Cref{mapspacecomparison} to obtain a natural zig-zag of weak equivalences between \(\Map_X(x,y)\) and \(\mathfrak{C}(X)(x,y)\) for any pair of vertices \(x,y\) of \(X\).  By \(3\)-for-\(2\) and since \[\mathfrak{C}(X)(x,y) \to \mathfrak{C}(Y)(fx,fy)\] was assumed to be an \(\M\)-equivalence, we see that the map \(\Map_X(x,y)\to \Map_Y(fx,fy)\) must also be an \(\M\)-equivalence.  In fact, since both \(\Map_X(x,y)\) and \(\Map_Y(fx,fy)\) are local, this map is actually an equivalence for \(\Cat_{\spsh_{\mathrm{inj}}}\).  The argument showing that \(f\) is bijective on iso-components is the same as in the proof of \Cref{conservativity} by passing to the underlying quasicategory.  Therefore, it follows that \(f\) is a horizontal Dwyer-Kan equivalence, which concludes the proof.
\end{proof}
\subsection{The Yoneda embedding and Yoneda's lemma}
We need the following easy lemma:
\begin{lemma} There is a natural isomorphism \(\mathfrak{C}(X^\op)\cong \mathfrak{C}(X)^\op\).  
\end{lemma} 
\begin{proof}
	It suffices to check on representables, and this is left as an easy exercise to the reader.
\end{proof}
Before we give a construction of the Yoneda embedding and a proof of Yoneda's lemma for \(\M\)-quasicategories, we need two lemmas from \cite{htt}.   We fix a Cartesian presentation \(\mathcal{M}=(\C,\setS)\) for the remainder of this section.

\begin{prop}\cite{htt}*{4.2.4.4}\label{lurieprop1}
	Let \(X\in \cellset\) be a cellular \(\C\)-set, \(\mathcal{D}\) a small \(\spsh\)-enriched category, and let \(\mathfrak{C}(X)\to \mathcal{D}\) be an equivalence of \(\spsh_{\setS}\)-enriched categories. Suppose \(\mathbf{A}\) is a \(\spsh_{\setS}\)-enriched model category, and let \(\mathcal{U}\) be \(\mathcal{D}\)-chunk (see \cite{htt}*{A.3.4.9} for the definition).  Then the induced map
	\[\mathfrak{N}((\mathcal{U}^\mathcal{D})^\circ)\to \mathfrak{N}(\mathcal{U}^\circ)^X\] is an equivalence of \(\M\)-quasicategories.
\end{prop}
\begin{proof} Although we have altered the statement slightly, the only result used in the proof in \cite{htt} that doesn't hold for all excellent monoidal model categories is \cite{htt}*{2.2.5.1}, but the analogue of this is exactly \Cref{maintheorem2}.
\end{proof}
\begin{prop}\cite{htt}*{4.2.4.7}\label{lurieprop2}
	Let \(\mathcal{I}\) be a fibrant \(\spsh_{\setS}\)-enriched category, \(X\) an object of \(\cellset\), and \(p:\mathfrak{N}\to X\) be any map.  Then we can find the following:
	\begin{itemize}
		\item A fibrant \(\spsh_{\setS}\)-enriched category \(\mathcal{D}\).
		\item An enriched functor \(P:\mathcal{I}\to \mathcal{D}\).
		\item A map \(j:X\to \mathfrak{N}(\mathcal{D})\) that is a weak equivalence in \(\cellset_\mathcal{M}\).
		\item An equivalence between \(j\circ p\) and \(\mathfrak{N}(P)\) as objects of the \(\M\)-quasicategory \(\mathfrak{N}(\mathcal{D})^{\mathfrak{N}(\mathcal{I})}\).
	\end{itemize}
\end{prop}
\begin{proof} No change to the proof of \cite{htt}*{4.2.4.7} is needed.
\end{proof}

We begin by constructing the Yoneda embedding:
\begin{defn}
	Let \(X\) be a \(\C\)-cellular set, and let \(\Phi:\mathfrak{C}(X) \xrightarrow{\sim} \mathcal{D}\) be an \(\M\)-enriched fibrant replacement.  Since \(\mathcal{D}\) is fibrant, the functor
	\[\Hom_\mathcal{D}: \mathcal{D}\times \mathcal{D}^\op \to \spsh\]
	factors through the full subcategory \(\spsh^\circ\).  By \cite{htt}*{Corollary A.3.4.14}, this gives rise up to homotopy to a universal map
	\[J_\mathcal{D}:\mathcal{D}\to \left(\spsh^{\mathcal{D}^\op}\right)_\mathrm{proj}^\circ\]
	that is fully faithful up to homotopy, since it is homotopic to the enriched Yoneda embedding. 
	Then we have a map
	\[\mathfrak{C}(X\times X^\op) \xrightarrow{\mathfrak{C}(p_1) \times \mathfrak{C}(p_2)} \mathfrak{C}(X) \times \mathfrak{C}(X)^\op \xrightarrow{\Phi \times \Phi^\op} \mathcal{D}\times \mathcal{D}^\op \xrightarrow{\Hom_{\mathcal{D}}} \spsh_\setS^\circ\]
	which yields by adjunction
	\[X \to \mathfrak{N}(\spsh_\setS^\circ)^{X^\op}.\]
	We denote \(\mathfrak{N}(\spsh_\setS^\circ)\) simply by \(\M\), and finally, we obtain \emph{the Yoneda embedding} for \(\M\)-quasicategories:
	\[j:X\to \M^{X^\op}.\]
\end{defn}

Let \(\Pre(X)\) denote the large \(\M\)-quasicategory \(\M^{X^\op}\), and let \(\M^+\) denote the coherent nerve of the huge enriched category of not-necessarily-small \(\setS\)-local injectively fibrant simplicial presheaves on \(\C\).  

\begin{defn} We say that a functor \(F:X^\op \to \M\) is \emph{representable} if the object it classifies belongs to the essential image of the Yoneda embedding \(j:X\to \Pre(X)\).  If \(x:\ast\to X\) is a vertex of \(X\), we denote the associated representable functor by \(h_x\).  
\end{defn}

\begin{prop}[Yoneda embedding]\cite{htt}*{5.1.3.1}
	The Yoneda embedding is fully faithful.
\end{prop}
\begin{proof} 
	First, let \(\Phi:\mathfrak{C}(X^\op)\xrightarrow{\sim} \mathcal{D}\) be a fibrant replacement.  We have an equivalence by \Cref{lurieprop1} 
	\[j^{\prime\prime}: \mathfrak{N}\left(\left((\spsh_\setS)^\mathcal{D}_\mathrm{proj}\right)^\circ\right) \xrightarrow{\sim} \M^{\mathfrak{N}(\mathcal{D})}, \]
	and since \(\M\) is fibrant and \(\Psi:X^\op \to \mathfrak{N}\mathcal{D}\), the adjunct of \(\Phi\) is an equivalence between cofibrant objects, it follows that the induced map \(\M^{\mathfrak{N}\mathcal{D}}\to \Pre(X)\) is an equivalence between fibrant objects, since the model structure is Cartesian.  Since \(X\) is cofibrant, it follows that there is a map \(h:X\to \M^{\mathfrak{N}\mathcal{D}}\) such that the composite
	\[X \xrightarrow{h} \M^{\mathfrak{N}\mathcal{D}} \xrightarrow{\M^{\Psi}} \Pre(X)\]
	is equivalent to \(j\), and again, since \(\mathfrak{N}\left(\left((\spsh_\setS)^\mathcal{D}_\mathrm{proj}\right)^\circ\right)\) is fibrant and \(j^{\prime\prime}\) is a weak equivalence between fibrant objects, we can find a map 
	\[j^\prime:X\to \mathfrak{N}\left(\left((\spsh_\setS)^\mathcal{D}_\mathrm{proj}\right)^\circ\right)\]
	such that the composite 
	\[X\xrightarrow{j^\prime} \mathfrak{N}\left(\left((\spsh_\setS)^\mathcal{D}_\mathrm{proj}\right)^\circ\right) \xrightarrow{j^{\prime\prime}} \M^{\mathfrak{N}\mathcal{D}} \xrightarrow{\sim} \Pre(X)\]
	is equivalent to \(j\).
	It suffices to show that the map \(j^\prime\) is fully faithful.  Let 
	\[J:\mathfrak{C}(X) \to \left((\spsh_\setS)^\mathcal{D}_\mathrm{proj}\right)^\circ\]
	be the adjunct of \(j^\prime\).  Then we see that \(J\) is equivalent to the composite \(J_{\mathcal{D^\op}}\circ \Phi^\op\), where \(J_{\mathcal{D^\op}}\) is fully-faithful and \(\Phi^\op\) is an equivalence, which concludes the proof.	
\end{proof}

\begin{prop}[Yoneda's Lemma]\cite{htt}*{5.5.2.1}
	Let \(X\) be a small \(\C\)-cellular set, and let \(f: X^\op \to \M\) be an object of \(\Pre(X)\).  Then let \(F:\Pre(X)^\op \to \M^+\) be the functor represented by \(f\).  Then the composite \[X^\op\xrightarrow{j^\op_X} \Pre(X)^\op \xrightarrow{F} \M^+\] is equivalent to \(f\).
\end{prop}
\begin{proof}
	By \Cref{lurieprop2}, we can choose a small fibrant \(\M\)-enriched category \(\mathcal{D}\) and an equivalence \(\Phi:X^\op\to \mathfrak{N}(\mathcal{D})\) such that \(f\sim \mathfrak{N}(f^\prime) \circ \Phi\) for some \(f^\prime:\mathcal{D} \to \spsh_\setS^\circ\).  Without loss of generality, we can assume that \(f^\prime\) is a projectively cofibrant diagram.  Using \Cref{lurieprop1}, we have an equivalence of \(\M\)-quasicategories 
	\[\Psi:\mathfrak{N}\left(\left((\spsh_\setS)_\mathrm{proj}^\mathcal{D}\right)^\circ\right) \xrightarrow{\sim} \Pre(X).\] 
	We observe that \(F\circ \Psi^\op\) can be identified with the coherent nerve of the the map 
	\[G: \left(\left((\spsh_\setS)_\mathrm{proj}^\mathcal{D}\right)^\circ\right)^\op \to \left(\spsh^{+}_\setS\right)^\circ\] represented by \(f^\prime\).  The Yoneda embedding factors through \(\Psi\) by the adjunct of the composite 
	\[j^\prime:\mathfrak{C}(X)\xrightarrow{\Phi^\op} \mathcal{D}^\op \hookrightarrow \left((\spsh_\setS)_\mathrm{proj}^\mathcal{D}\right)^\circ,\]
	so it follows that \(F\circ j^\op\) can be identified with the adjunct of 
	\[\mathfrak{C}(X)^\op \xrightarrow{(j^\prime)^\op} \left(\left((\spsh_\setS)_\mathrm{proj}^\mathcal{D}\right)^\circ\right)^\op \xrightarrow{G} \left(\spsh^{+}_\setS\right)^\circ.\]
	This composite is equal to the \(f^\prime\) we started with, so its coherent nerve is equivalent to \(f\).
\end{proof}

\subsection{Weighted limits and colimits}
We fix a Cartesian presentation \(\M=(\C,\setS)\).  Our presentation here follows the one given in \cite{nlabwlim}.

\begin{defn}
	Let \(D\) be a small \(\C\)-cellular set, that is, an object of \(\cellset\), and suppose we have a diagram \(f:D\to \M\) and a \emph{weight} \(W:D\to \M\).  If \(X\) is an \(\M\)-quasicategory and \(f:D\to X\) is a diagram, we define the map \(h_f: X^\op \to M^D\) to be the adjunct of the composite \[D\to X\xrightarrow{j} \M^{X^\op},\] and we define \(h^W:\M^D \to \M^+\) to be the map co-representing \(W\) as an object of \(\M^D\).  Then if the composite \(h^W\circ h_f:X^\op \to \M^+\) is representable, we call its representing object \emph{the limit of \(f\) weighted by \(W\)}, denoted  by \(\limw^W f\).

	Dually, coweights for colimits are diagrams \(D^\op \to \M\), and their coweighting functors are their associated corepresentable functors \(\M^{D^\op}\to \M^+\).  Given a diagram \(f:D\to X\), the \emph{colimit of \(f\) weighted by \(W\)} is defined to be an object of \(X\) that corepresents the weighted limit of \(f^\op\), which is denoted by \(\colim^W f\).  
\end{defn}

\begin{prop} The large \(\M\)-quasicategory \(\M\) has all small weighted limits and colimits.
\end{prop}
\begin{proof}  Given \(f:D\to \M\) and a weight \(W:D\to \M\), it is straightforward to see that a representing object for the weighted limit is exactly \[\Map_{\M^D}(W,f),\] by unwinding the definitions.  The existence of weighted colimits is subtler, but follows from general facts about weighted homotopy colimits in enriched model categories.
\end{proof}

For now, we don't have much to say for this particular application.  If other definitions are proposed for weighted limits and colimits, they should be equivalent to these.  

\subsection{Examples}
The only examples we really care about are the cases where \(\C=\Theta_n\) for \(0\leq n\leq \omega\) and where \(\setS\) is the set of generating anodynes for the model structure on weak \(n\)-categories.  We invite the reader to consider other applications.  We expect that a simple application would be to consider the left-Bousfield localization of spaces at homology equivalences, but we aren't certain if this is a Cartesian model structure.

Also note that our definitions of weighted limits and colimits do not work for computing lax and oplax weighted limits and colimits in weak \(\omega\)-categories (taking \(\C\) to be \(\Theta=\Theta_\omega\) and \(\setS\) to be the union of all of the generating weak equivalences for the Rezk model structure).  The problem is that the lax Gray tensor product has not yet been shown to be homotopy-invariant (in particular, a left-Quillen bifunctor), so the function complexes of lax and oplax natural transformations are themselves not yet known to be homotopy invariant.  This is why we encourage people in the future to refer to weighted limits and colimits in these cases as \emph{weighted pseudolimits} and \emph{weighted pseudocolimits} respectively.

\appendix
\section{Appendix: Recollections on Cisinski Theory}
We recall in the appendix a number of important technical facts from Cisinski theory, which comprises a large body of results on the construction of extremely tame model structures on presheaf categories.  This is the big machine backing up \Cref{modelstrucdefn} as well as our reduction of \Cref{rezkcomparison} to checking properties of generating anodynes.
\subsection{Cisinski model structures and localizers}
In what follows, we will work with a fixed small category \(\mathcal{A}\). These are stated in more generality in \cite{cisinski-book}*{Chapter 1}, but we will specialize to the case of a Cartesian cylinder functor, that is, a cylinder functor determined by taking the Cartesian product with an interval object.
\begin{defn}
	A \emph{separating interval object} of \(\psh{\mathcal{A}}\) is an object \(I\) together with two monic arrows \(\partial_0,\partial_1:\ast \to I\) such that the pullback \(\ast \times_{I} \ast=\varnothing\).  We call the induced map \(\delta^I:\ast \coprod \ast \xrightarrow{(\partial_0,\partial_1)} I\) the \emph{boundary map}
\end{defn}
\begin{defn}
	A cellular model \(\mathscr{M}\) for \(\psh{\mathcal{A}}\) is a small set of monomorphisms such that \(\operatorname{llp}(\operatorname{rlp}(\mathscr{M}))\) is exactly the class of monomorphisms of \(\psh{\mathcal{A}}\).   
\end{defn}
\begin{prop}[\cite{cisinski-book}*{Proposition 1.2.27}]
	Every category of presheaves on a small category \(\mathcal{A}\) admits a cellular model in which the target of each map is a quotient of a representable.
\end{prop}
\begin{defn}
	A \emph{class of anodynes} \(\mathbf{An}\) for a separating interval \(I\) is a class of monomorphisms that satisfies the following properties:
	\begin{itemize}
		\item \(\mathbf{An}\) is generated by a set \(S\), that is, there exists a set of monomorphisms \(S\) such that \(\mathbf{An}=\operatorname{llp}(\operatorname{rlp}(S))\).
		\item For any monomorphism \(g\), the corner maps \(\partial_i \times^\lrcorner g\in \mathbf{An}\) for \(i\in \{0,1\}\).  
		\item For any map \(f\in \mathbf{An}\), the map \(\delta^I \times^\lrcorner f \in \mathbf{An}\).  
	\end{itemize} 
\end{defn}
\begin{prop}[\cite{cisinski-book}*{Proposition 1.3.13}] Given any set \(S\) of monomorphisms and any separating interval object \(I\), there exists a smallest class of anodynes \(\mathbf{An}_I(S)\) for \(I\).  In particular, this class is generated by the set of maps \(\Lambda_I(S,\mathscr{M})\) where \(\mathscr{M}\) is a cellular model for \(\mathcal{A}\) defined as follows:
	\begin{itemize}
		\item We define the set \(\Lambda_I^0(S,\mathscr{M})=S \cup \partial_0 \times^\lrcorner \mathscr{M} \cup \partial_1 \times^\lrcorner \mathscr{M}\)
		\item Then we define for any set of maps \(T\) the set \(\Lambda_I(T)=\delta^I\times^\lrcorner T\).
		\item Then we define \(\Lambda_I(S,\mathscr{M})=\bigcup_i^\infty \Lambda_I^i(\Lambda_I^0(S,\mathscr{M}))\).  
	\end{itemize}
\end{prop}
\begin{thm}[\cite{cisinski-book}*{1.3.22}]\label{cisinskimaintheorem}
	Given a set of monomorphisms \(S\) and a separating interval \(I\), there exists a model structure on \(\psh{A}\) in which the cofibrations are the monomorphisms, the fibrant objects are the objects \(A\) such that the terminal map \(A\to \ast\) belongs to the class \(\operatorname{rlp}(\mathbf{An}_I(S))\), and in which a map \(f:A\to A^\prime\) with \(A^\prime\) fibrant is a fibration if and only if \(f\) belongs to \(\operatorname{rlp}(\mathbf{An}_I(S))\).
\end{thm}
\begin{defn}
	A \emph{Cisinski model structure} is any model structure constructed using \Cref{cisinskimaintheorem}.
\end{defn}
\begin{cor}
	Taking \(S=\varnothing\) and \(I\) to be the subobject classifier \(\mathfrak{L}\) of \(\psh{A}\) with the two canonical sections \(\varnothing,\id:\ast \to \mathfrak{L}\), we obtain the minimal Cisinski model structure.  More generally, we can replace \(\mathfrak{L}\) with any injective separating interval.  
\end{cor}
\begin{defn}
	An \(\mathcal{A}\)-localizer \(\mathsf{W}\) is a class of maps of \(\psh{A}\) satisfying the following axioms
	\begin{itemize}
		\item The class \(\mathsf{W}\) satisfies \(3\)-for-\(2\).
		\item Every trivial fibration belongs to \(\mathsf{W}\).
		\item The class of monomorphisms in \(\mathsf{W}\) is closed under pushout and transfinite composition.
	\end{itemize}
	If \(S\) is a set of morphisms of \(\psh{\mathcal{A}}\), there exists a minimal localizer containing \(S\), which we call the \emph{localizer generated by S} and denote by \(\mathsf{W}(S)\).  We say that a localizer \(\mathsf{W}\) is \emph{accessible} if it generated by a set of morphisms.
\end{defn}
\begin{thm}[\cite{cisinski-book}*{Theorem 1.4.3}]
	Given any set of morphisms \(S\) of \(\psh{\mathcal{A}}\), the localizer \(\mathsf{W}(S)\) is the class of weak equivalences for a Cisinski model structure.  Moreover, this model structure is the left-Bousfield localization of the minimal Cisinski model structure at the set \(S\).
\end{thm}
\subsection{Simplicial Completion}
The localizers on \(\mathcal{A}\) have a non-free component, namely that the class of trivial fibrations must always belong to \(\mathsf{W}\).  The theory of simplicial completions allows us to embed the class of \(\mathcal{A}\)-localizers into a larger class of localizers that doesn't suffer from this defect.  These are models for \emph{free homotopy theories} modeled on \(\mathcal{A}\).  The idea here is to replace the interval object with an external interval object.
\begin{defn}
	We define the \emph{free homotopy theory} on \(\mathcal{A}\) generated by \(S\) to be \(\mathcal{A}\times \Delta\)-localizer generated the interval object \(\Delta^1 \defeq \ast \square \Delta^1,\) where \(\square\) denotes the external product.  
\end{defn}
\begin{rem}[\cite{cisinski-book}*{3.4.50}]
	By well-known combinatorial arguments, it can be seen that taking \(\Delta^1\) to be the separating interval object forces all objects \(\Delta^n=\ast \square \Delta^n\) to be weakly contractible. The free homotopy theory construction therefore adds new representables but homotopically nullifies all of them.  We can therefore view it as a way to present a homotopy theory for presheaves on \(\mathcal{A}\) without automatically forcing all of the trivial fibrations to be weak equivalences. 

	The free homotopy theory on \(\mathcal{A}\) is in general radically different from the homotopy theory given by the injective model structure on simplicial presheaves, which we will see later is its regular completion.  Cisinski gives the example where \(\mathcal{A}=B\mathbb{G}\) for a group \(\mathbb{G}\).  The difference between the free homotopy theory and its regular completion in this case is the difference between equivariant homotopy theory and higher Galois theory.  That is, the free homotopy theory presents ordinary equivariant homotopy theory, while the regular completion of the free homotopy theory on \(B\mathbb{G}\) models non-abelian \(\mathbb{G}\)-representations. 
\end{rem}
\begin{defn}
	Given an \(\mathcal{A}\)-localizer \(\mathsf{W}\), we define the \emph{simplicial completion} of \(\mathsf{W}\) to be the \(\mathcal{A}\times \Delta\) localizer \(\mathsf{W}_\Delta\) generated by the class of maps of simplicial objects \(X\to X^\prime\) such that \(X_n \to X^\prime_n\) belongs to \(\mathsf{W}\) for each \(i\geq 0\) together with the projection maps \(X\times \Delta^1\) for all simplicial presheaves \(X\) on \(A\).  We say that an \(\mathcal{A}\times \Delta\)-localizer is \emph{discrete} if it is the simplicial completion of a localizer on \(\mathcal{A}\).  
\end{defn}
\begin{prop}
	If the localizer \(\mathsf{W}\) is accessible, so is \(\mathsf{W}_\Delta\).  
\end{prop}
\begin{prop}
	If the localizer \(\mathsf{W}\) is the minimal \(\mathcal{A}\)-localizer, then \(\mathsf{W}_\Delta\) is the localization of the free homotopy theory on \(\mathcal{A}\) at the set of maps \(\Lambda_I(\varnothing,\mathscr{M})\) for any injective separating interval object \(I\) and any cellular model \(\mathscr{M}\).  
\end{prop}
\begin{prop}[\cite{cisinski-book}*{Proposition 2.3.27}]
	If \(\mathsf{W}\) is any accessible localizer, the functor \(p^\ast:\psh{\mathcal{A}} \to \psh{\mathcal{A}\times \Delta}\) induced by the projection \(\mathcal{A}\times \Delta \to \mathcal{A}\) is a left Quillen equivalence.  Also, by choosing a Reedy-cofibrant cosimplicial resolution \(D^\bullet\) of the terminal object \(\ast\) of \(A\) with respect to the minimal localizer \(\mathsf{W}_{\mathrm{min}}\), the functor 
	\[
		\operatorname{Real}_D: \psh{\mathcal{A}\times \Delta} \to \psh{\mathcal{A}}
	\]
	induced by left Kan extension of the functor defined by the rule
	\[
		(A,[n])\mapsto A\times D^n
	\]
	is also a left Quillen equivalence.  
\end{prop}
\begin{cor}
	The simplicial completion defines a bijective Galois connection between \(\mathcal{A}\)-localizers and discrete \(\mathcal{A}\times \Delta\)-localizers.
\end{cor}
\subsection{Regularity}
An important property of simplicial sets is no longer present in the case of a general \(\mathcal{A}\)-localizer, namely the property that every object is the canonical homotopy colimit of its diagram of representables.  This leads to the following definition:
\begin{defn}
	A presheaf \(X\) on \(\mathcal{A}\) is called \(\mathsf{W}\)-\emph{regular} with respect to a localizer \(\mathsf{W}\) if the canonical map 
	\[
		\hocoliml_{A\to X\in \overcat{\mathcal{A}}{X}} A \to \coliml_{A\to X\in \overcat{\mathcal{A}}{X}} A \to X
	\]
	is a \(\mathsf{W}\)-equivalence.
	A localizer \(\mathsf{W}\) on \(\mathcal{A}\) is called a \emph{regular localizer} if every presheaf \(X\) on \(A\) is \(\mathsf{W}\)-regular.
\end{defn}
\begin{prop}[\cite{cisinski-book}*{Remark 3.4.14}]\label{regularinclusion}
	If \(\mathsf{W} \subseteq \mathsf{W}^\prime\) is an inclusion of localizers and \(\mathsf{W}\) is regular, then so too is \(\mathsf{W^\prime}\).  
\end{prop}
\begin{defn}\label{regcompletion}
	The \emph{regular completion} of a localizer \(\mathsf{W}\) is the smallest regular localizer containing \(\mathsf{W}\).  In particular, it follows from the preceding proposition that the regular completion is the smallest localizer generated by \(\mathsf{W}\) and the regular completion of the minimal localizer.
\end{defn}
\begin{prop}[\cite{cisinski-book}*{Corollary 3.4.24}]
	The regular completion of an accessible localizer is accessible.
\end{prop}
\begin{prop}[\cite{cisinski-book}*{Proposition 3.4.34}]\label{injregular}
	The localizer of the injective model structure on simplicial presheaves on \(\mathcal{A}\) consisting of the maps of simplicial presheaves \(X\to X^\prime\) whose components are weak homotopy equivalences \(X_A\to  X^\prime_A\) is the regular completion of the localizer of the free homotopy theory on \(\mathcal{A}\).
\end{prop}
\begin{cor}
	The Cisinski model structure obtained from the simplicial completion \(\mathsf{W}_\Delta\) of an accessible localizer \(\mathsf{W}\) on \(\mathcal{A}\) is a left-Bousfield localization of the injective model structure on simplicial presheaves if and only if it is regular.  In particular, the Galois correspondence between localizers on \(\mathcal{A}\) and localizers containing the simplicial completion of the minimal localizer restricts to a bijective Galois correspondence between regular localizers on \(\mathcal{A}\) and discrete localizers on \(\mathcal{A}\times \Delta\) containing the objectwise weak homotopy equivalences.
\end{cor}
We also make note of the following technical fact:
\begin{prop}[\cite{cisinski-book}*{Corollary 3.4.41}]\label{filteredcolims}
	Let \(\mathcal{A}\) be a small category, and let \(\mathsf{W}\) be a regular \(\mathcal{A}\)-localizer.  Then \(\mathsf{W}\) is closed under filtered colimits.
\end{prop}
\subsection{Skeletal categories}
In this section, we recall Cisinski's theory of skeletal categories (cat\'egories squelettiques).  These are generalized Reedy categories \(\mathcal{A}\) with a dimension grading and satisfying certain axioms. Under the strong condition of \emph{normality}, the category \(\psh{\mathcal{A}}\) admits a canonical cellular model given by the boundary inclusions.  Under a further strong assumption of \emph{regularity}, every \(\mathcal{A}\)-localizer will be shown to be regular.
\begin{defn}
	A \emph{skeletal category} is given by the data of a small category \(\mathcal{A}\), subcategories \(\mathcal{A}^-\) and \(\mathcal{A}^+\) together with a grading function \(\operatorname{dim}:\Ob \mathcal{A} \to \mathbf{N}\) satisfying the following axioms:
	\begin{itemize}
		\item Every isomorphism belongs to both \(\mathcal{A}^-\) and \(\mathcal{A}^+\).
		\item If \(f:A\to A^\prime\) belongs to \(\mathcal{A}^+\) (resp. \(\mathcal{A}^-\)), then \(\dim(A) \leq \dim(A^\prime)\) (resp. \(\dim(A^\prime) \leq \dim(A)\)).
		\item Every map \(f\) of \(\mathcal{A}\) admits a factorization, unique up to unique isomorphism of factorizations, into a composite \(\delta \circ \pi\) with \(\delta \in \mathcal{A}^+\) and \(\delta \in \mathcal{A}^-\).
		\item Two arrows \(f,g:A\to A^\prime\) of \(\mathcal{A}^-\) are equal if and only if they have the same sections.
	\end{itemize}
\end{defn}
\begin{defn}
	Given a natural number \(n\) and a presheaf \(X\) on \(\mathcal{A}\), we define the \(n\)-skeleton to be the sieve
	\[
		\operatorname{Sk}^n(X)_A\defeq \{u:A\to X \mid \exists \alpha:A\to A^\prime, \quad \dim(A^\prime)\leq n, \quad \exists u^\prime:A^\prime \to X, \quad u=u^\prime\circ \alpha\}.
	\]
	If \(A\) is a representable object of \(\mathcal{A}\), we define the boundary \(\partial A\) of \(A\) to be  \(\operatorname{Sk}^{\dim(A)-1}(A)\), and we denote its inclusion by \(\delta^A:\partial A \hookrightarrow A\).  
\end{defn}
We take the following as a definition, but it is in fact a characterization from \cite{cisinski-book}*{8.1.37}
\begin{defn}
	A skeletal category is called \emph{normal} if its objects have no nontrivial automorphisms.
\end{defn}
\begin{prop}[\cite{cisinski-book}*{Proposition 8.1.37}]\label{normskelcat}
	If \(\mathcal{A}\) is a normal skeletal category, then the set of maps \(\mathscr{M}\defeq \{\delta^A\}_{A\in \mathcal{A}}\) gives a cellular model for \(\psh{\mathcal{A}}\). Moreover, the class of monomorphisms of \(\mathcal{A}\) is exactly \(\operatorname{Cell}(\mathscr{M})\).
\end{prop}
\begin{rem}
	Cisinski shows that whenever \(X\) is a normal presheaf (we omit this definition, but in the case where \(\mathcal{A}\) is normal skeletal, every presheaf satisfies this property), its \(n\)-skeleton can be computed as the image of \(X\) under the composite  adjunction induced by the inclusion of the full subcategory \(\mathcal{A}_{\leq n}\hookrightarrow \mathcal{A}\), similar to the case of \(\Delta\).  In particular, normal skeletal categories have a well-behaved skeleton-coskeleton adjunction.  
\end{rem}
\begin{defn}\label{regskelcat}
	We say that presheaf \(X\) on a skeletal category is \emph{regular} if every nondegenerate section of \(X\) is monic. Additionally, we say that a skeletal category is \emph{regular} if it is normal and every representable presheaf is regular.
\end{defn}
\begin{thm}[\cite{cisinski-book}*{Proposition 8.2.9}]
	Every localizer \(\mathsf{W}\) on a regular skeletal category is regular.
\end{thm}


\end{document}